\documentclass[mnsc,nonblindrev]{informs3-VVM}

\OneAndAHalfSpacedXI %

\usepackage[margin=1in]{geometry}
\usepackage{amsmath}
\usepackage{mathtools}
\usepackage[caption=false]{subfig}
\usepackage{bbm}
\usepackage{breqn}
\usepackage{amsfonts}
\usepackage{amssymb}
\usepackage{graphicx}
\usepackage{xcolor}
\usepackage{url}

\usepackage{ctable}

\usepackage{multirow}

\usepackage{algorithm}
\usepackage{algorithmic}

\usepackage{enumitem} 
\usepackage{longtable}
\usepackage[labelfont=bf]{caption}

\usepackage{natbib}
\bibpunct[, ]{(}{)}{,}{a}{}{,}%
\def\bibfont{\small}%
\def\stop{\textbf{stop}}
\def\go{\textbf{continue}}

\def\ab{\mathbf{a}}
\def\bb{\mathbf{b}}
\def\cb{\mathbf{c}}

\def\Bb{\mathbf{B}}
\def\eb{\mathbf{e}}
\def\yb{\mathbf{y}}
\def\ub{\mathbf{u}}
\def\vb{\mathbf{v}}
\def\xb{\mathbf{x}}
\def\wb{\mathbf{w}}

\def\Vb{\mathbf{V}}

\def\zerob{\mathbf{0}}

\def\epsilonb{\boldsymbol \epsilon}

\def\Exp{\mathbb{E}}
\def\Ibb{\mathbb{I}}
\def\Pbb{\mathbb{P}}
\def\Dbb{\mathbb{D}}

\def\Acal{\mathcal{A}}
\def\Bcal{\mathcal{B}}
\def\Fcal{\mathcal{F}}

\def\Xcal{\mathcal{X}}

\def\Halmos{$\square$}

\def\conv{\mathrm{conv}}
\def\ext{\mathrm{ext}}
\def\sign{\mathrm{sign}}
\def\Logistic{\mathrm{Logistic}}

\TheoremsNumberedThrough     %
\ECRepeatTheorems

\EquationsNumberedThrough    %
\RUNAUTHOR{Guan and Mi\v{s}i\'{c}}

\begin{document}
\RUNTITLE{Randomized Policy Optimization for Optimal Stopping}

\TITLE{Randomized Policy Optimization for Optimal Stopping}

\ARTICLEAUTHORS{%
\AUTHOR{Xinyi Guan}
\AFF{UCLA Anderson School of Management, University of California, Los Angeles, California 90095, United States, \EMAIL{xinyi.guan.phd@anderson.ucla.edu}}
\AUTHOR{Velibor V. Mi\v{s}i\'{c}}
\AFF{UCLA Anderson School of Management, University of California, Los Angeles, California 90095, United States, \EMAIL{velibor.misic@anderson.ucla.edu}} %

} %

\ABSTRACT{%
Optimal stopping is the problem of determining when to stop a stochastic system in order to maximize reward, which is of practical importance in domains such as finance, operations management and healthcare. Existing methods for high-dimensional optimal stopping that are popular in practice produce deterministic linear policies -- policies that deterministically stop based on the sign of a weighted sum of basis functions -- but are not guaranteed to find the optimal policy within this policy class given a fixed basis function architecture. In this paper, we propose a new methodology for optimal stopping based on \emph{randomized} linear policies, which choose to stop with a probability that is determined by a weighted sum of basis functions. We motivate these policies by establishing that under mild conditions, given a fixed basis function architecture, optimizing over randomized linear policies is equivalent to optimizing over deterministic linear policies. We formulate the problem of learning randomized linear policies from data as a smooth non-convex sample average approximation (SAA) problem.  We theoretically prove the almost sure convergence of our randomized policy SAA problem and establish bounds on the out-of-sample performance of randomized policies obtained from our SAA problem based on Rademacher complexity. We also show that the SAA problem is in general NP-Hard, and consequently develop a practical heuristic for solving our randomized policy problem. Through numerical experiments on a benchmark family of option pricing problem instances, we show that our approach can substantially outperform state-of-the-art methods.

}%


\maketitle

\section{Introduction}
\label{sec:introduction}

Optimal stopping is the problem of deciding at what time to stop a stochastic system in order to maximize the expected reward. Specifically, we are given a stochastic system, that starts at an initial state and transitions randomly from one state to another in discrete time, and a reward function, which maps each state at each time to a real value. In each period, we must decide whether to stop the system, or allow it to continue for one more period. If we choose to stop the system, we obtain the reward given by the reward function for the current state; otherwise, we obtain no reward, but we may potentially stop the system at a later period for a higher reward. Our goal is to find a policy, which is a mapping from the state at each period to the decision to stop or continue, so as to maximize the expected reward. 

Optimal stopping problems are found in many application domains, such as finance, operations and healthcare. For example, in finance, an important application of optimal stopping is the problem of option pricing. In this problem, an option holder has the right to buy an asset (if it is a call option) or to sell an asset (if it is a put option) at some strike price. The stochastic system corresponds to the asset, and the system state corresponds to the asset's current price. The option holder's problem is to decide when to exercise the option, which is akin to stopping, so as to garner the greatest expected payoff. The price that an option writer should charge for the option is exactly the highest expected payoff that one can obtain from an optimal exercise policy of the option. As another example, in operations management, consider a firm that needs to decide when to introduce a new product to a market. In this problem, the system corresponds to market conditions, and the system state would correspond to (say) the unit production cost and the predicted market share that the product would capture, which evolve stochastically over time as more and more competitors enter this market. At each period, the firm can decide to introduce the product into the market, which corresponds to stopping the system, and the reward corresponds to the profit obtained from this market. The problem is then to find a policy that determines whether to introduce the product or wait, so as to maximize the profit from introducing the product. %

High-dimensional optimal stopping problems can in theory be solved exactly by dynamic programming. This approach involves obtaining the optimal value function, which maps the state at each period to the highest possible expected reward that can be attained conditional on starting at that state in that period, or the optimal continuation value function, which maps the state at each period to the highest possible expected reward that can be attained conditional on choosing to continue out of that state in that period. An optimal policy can then be found by considering the greedy policy with respect to the optimal value function or optimal continuation value function. However, this approach is untenable in practice for high-dimensional optimal stopping problems due to the curse of dimensionality. 

As a result, a number of approaches based on approximate dynamic programming (ADP) have been proposed to solve high-dimensional optimal stopping problems, wherein one considers a policy that is greedy with respect to an approximate value function or continuation value function. Of these methods, the most prevalent ADP method is the least squares Monte Carlo (LSM) approach proposed by \cite{longstaff2001valuing}. This approach involves simulating a set of sample paths or trajectories of the system, and then iterating from the last period in the horizon to the first. At each period $t$, one uses least squares to obtain a regression model that predicts the continuation value based on the current state, using the sample of trajectories. One then compares the prediction with the reward from stopping in the current period in each trajectory. If the reward from stopping is higher than the predicted continuation value, we choose to stop; otherwise, we choose to continue. Based on this decision, we update the continuation value, and we repeat the process again at period $t-1$. The algorithm continues in this way, until we reach the first period. The resulting policy is then to take the action that is greedy with respect to the approximate continuation value function. 

From a theoretical standpoint, if one were given an infinite sample of trajectories and one could solve the least squares problem at each stage of the LSM algorithm over an unrestricted function class, then the regression model that one would obtain would exactly coincide with the optimal continuation value function. This is due to the fact that the conditional expectation function $m(x) = \Exp[Y \mid X = x]$ minimizes squared error, i.e., it solves the optimization problem $\min_{m} \Exp[ (Y - m(X))^2]$. In such an idealized situation, the policy produced by LSM would indeed be optimal. 

In practice, one must work with a finite sample of trajectories, and the regression function is constrained to be within the span of a finite collection of basis functions that are specified by the decision maker. Thus, the policy that is produced by LSM is a policy in which one decides to stop or continue by comparing the reward to a weighted sum of basis functions. This is significant for two reasons: (i) it is no longer the case that the policy produced by LSM is an optimal policy; and (ii) even when we restrict our focus to the corresponding policy class that LSM operates in -- policies that stop if and only if the reward is greater than a weighted combination of basis functions -- the policy produced by LSM may not be optimal within that class. This occurs because in LSM, the approximate continuation value function is obtained by minimizing squared loss, which does not account for the fact that this approximation will be used as part of a policy, and ultimately does not guarantee good out-of-sample policy performance. 

This motivates the following question: \emph{how can one obtain LSM-like policies that perform better than LSM}? The policy produced by LSM belongs to a broader family of policies that we refer to as \emph{deterministic linear policies}: policies that deterministically recommend to stop or continue at each period depending on whether a weighted sum of basis functions is positive or negative. (This class subsumes LSM policies if one includes the immediate reward at each period as a basis function.) Given a sample of trajectories, an immediate approach to obtaining a good policy from this class would be to formulate a sample average approximation (SAA) problem: optimize over the weights defining the deterministic linear policy, so as to maximize the sample average estimate of the expected reward of the policy. The drawback of this approach is that due to the discrete nature of how this family of policies works, the SAA problem is a challenging discrete optimization problem. Such a problem would be infeasible to solve for the sample sizes that are typically found in practical optimal stopping applications. 

As an alternative to deterministic linear policies, one can also consider \emph{randomized linear policies}. These are policies that probabilistically choose to stop or continue at each period, where the probability of stopping is given by a logistic probability and the logit that defines this probability is a weighted sum of basis functions. Just like the deterministic linear policy case, one can also formulate an SAA problem to maximize the sample average reward with respect to the weights that define this randomized policy. Although the resulting SAA problem is still a challenging non-convex problem, the objective function is now smooth and from a computational standpoint, one can now at least solve the problem heuristically using any of a number of practically successful gradient-based methods.

In this paper, we propose a new methodology for solving optimal stopping problems from data that is based on optimizing over the class of randomized linear policies. We make the following specific contributions:
\begin{enumerate}
	\item \textbf{Model}: We propose the class of randomized linear policies for optimal stopping problems, and formulate the problem of learning such a policy from data as an SAA problem with a smooth, non-convex objective function. We prove that under mild conditions, solving the randomized linear policy SAA problem is equivalent to solving the deterministic linear policy SAA problem, in that the optimal objectives of the two problems are equivalent; under an additional condition, we also show that the true randomized linear policy problem and the true deterministic linear policy problem, where sample averages are replaced by expectations, are also equivalent in objective value. 
	\item \textbf{Statistical guarantees}: We provide two statistical guarantees for our randomized policy SAA problem. First, we show that our learning problem is consistent: as the number of trajectories in our training sample grows, the optimal objective value and optimal solution converge almost surely to the optimal objective value and optimal solution set, respectively, of the true stochastic optimization problem, where sample averages are replaced with expectations. Second, we develop a generalization bound on the out-of-sample objective value of a randomized policy obtained from our SAA problem based on Rademacher complexity, and develop several different bounds on the Rademacher complexity for different choices of the set of feasible weights. 
	\item \textbf{Heuristic}: We prove that in general, our randomized policy SAA problem is NP-Hard, which follows from a reduction from the MAX-3SAT problem. Consequently, we propose a backward optimization algorithm for solving the problem heuristically, which optimizes the weights defining the randomized policy in stages, starting with the weights corresponding to the last period and working its way to the first stage. 
	\item \textbf{Numerical experiments}: Using a benchmark family of Bermudan max-call option pricing instances used in the recent literature, we show that our approach yields policies that in general are substantially better than policies produced by LSM, and are as good or better than policies produced by the pathwise optimization method \citep{desai2012pathwise}, a state-of-the-art method based on martingale duality. 
\end{enumerate}

The rest of this paper is organized as follows. In Section~\ref{sec:literature_review}, we review the relevant literature in optimal stopping, as well as other recent related work. In Section~\ref{sec:problem_definition}, we formally define the optimal stopping problem, define the deterministic linear policy problem in its sample average and true stochastic forms, define the randomized linear policy problem in its sample average and true stochastic forms, and prove that the randomized linear policy problem and deterministic linear problem are equivalent. In Section~\ref{sec:statistical_properties}, we prove that our randomized policy SAA problem is consistent and develop our generalization guarantees. In Section~\ref{sec:solution_methodology}, we show that our randomized policy SAA problem is NP-Hard, and present our backward optimization algorithm for solving it. In Section~\ref{sec:numerics}, we present the results of our numerical study on option pricing instances. Lastly, in Section~\ref{sec:conclusion}, we conclude and discuss some potential directions for future research.

\section{Literature Review}
\label{sec:literature_review}

Our paper is closely related to three streams of research: the optimal stopping and ADP literature; prediction-and-optimization literature; and non-convex optimization literature.\\

\noindent \textbf{Optimal stopping and approximate dynamic programming (ADP).} Optimal stopping problems have been extensively studied in many fields such as statistics, operations research and mathematical finance. In theory, optimal stopping problems can be solved by dynamic programming, but in practice, the curse of dimensionality renders this approach infeasible for all but the simplest optimal stopping problems. As a result, there has been much attention towards developing good approximate dynamic programming (ADP) methods for optimal stopping. %

In the context of optimal stopping, the most popular family of ADP methods is that of simulation-regression. The idea of simulation-regression methods is to simulate a sample of trajectories of the system state and use least squares regression to approximate the optimal continuation value function (i.e., the optimal expected reward from choosing to continue for a given current state) at each step. The paper of \cite{carriere1996valuation} was the first to introduce this type of approach for the valuation of American options, using non-parametric regression; later, \cite{longstaff2001valuing} and \cite{tsitsiklis2001regression} independently considered this approach in the setting where the continuation value function is approximated as a linear combination of basis functions.

Besides simulation-regression, another important stream of ADP methods for optimal stopping is based on the idea of martingale duality. The main idea in this body of work is to relax the non-anticipativity of the policy, but to then penalize the use of future information through a martingale process. In doing so, one obtains an upper bound on the optimal reward, and in some cases one can also obtain policies that perform well. We refer the reader to \cite{rogers2002monte}, \cite{andersen2004primal}, \cite{haugh2004pricing}, \cite{chen2007additive}, \cite{brown2010information}, \cite{desai2012pathwise} for salient examples of this methodology, and to the recent review paper of \cite{brown2022information} for a detailed overview of this technique as it applies to stochastic dynamic programming more broadly. 

Lastly, other recent research has considered approaches distinct from the above two streams. The paper of \cite{ciocan2020interpretable} considers a method for directly obtaining optimal stopping policies from a sample of trajectories in the form of a binary tree. In a different direction, the paper of \cite{sturt2021nonparametric} proposes a method for obtaining threshold policies for low-dimensional optimal stopping problems using robust optimization. 

Our methodology is most closely related to the simulation-regression approach and in particular, the least-squares Monte Carlo (LSM) approach of \cite{longstaff2001valuing}. There are several differences between our methodology and LSM. One difference is that our methodology involves the use of randomized policies, whereas the policy produced by LSM is deterministic. Aside from this, the key philosophical difference between our work and the LSM approach is that while LSM produces a policy in an indirect way -- by approximating the continuation value function using least squares -- our methodology involves formulating an SAA problem and obtaining a policy that \emph{directly} maximizes an estimate of the expected reward obtained with respect to a sample of trajectories. In terms of algorithms, the backward algorithm for heuristically solving our SAA problem that we present in Section~\ref{sec:solution_methodology} is reminiscent of the LSM algorithm, but instead of solving a least squares problem, one solves a non-convex problem where the objective function is given by a weighted sum of logistic response functions. \\

\noindent \textbf{Predict-then-optimize}. Outside of optimal stopping, our paper relates to the literature on combining prediction and optimization. In many analytics problems, the ``predict-then-optimize'' paradigm is often used: one first builds a predictive model by minimizing a loss function that measures predictive performance (for example, squared error), and then utilizes that predictive model in a subsequent optimization problem to obtain a decision. There are many papers that apply this type of approach (see, for example, \citealt{ferreira2016analytics}, \citealt{cohen2017impact}, \citealt{bertsimas2020predictive}). 

However, as pointed out in the recent paper of \cite{elmachtoub2021smart}, this type of predict-then-optimize paradigm can lead to suboptimal decisions, since the predictive model is trained using a loss function that does not account for how the predictive model will be used in the downstream optimization problem. The paper of \cite{elmachtoub2021smart} proposes a Smart Predict-then-Optimize (SPO) framework, where the predictive model is estimated so as to minimize decision/prescriptive loss rather than predictive loss, and numerically shows that the SPO framework can result in significantly better out-of-sample performance. 

Our paper is partially inspired by the observation that the LSM algorithm bears a resemblance to the standard predict-then-optimize paradigm. In the LSM approach,  one first predicts the continuation value based on squared error and then uses that prediction within a greedy policy. However, minimizing squared error does not necessarily translate into good prescriptive performance of the prediction model. Therefore, in order to find a good policy, we consider the problem of directly optimizing in-sample reward over the space of randomized linear policies. \\

\noindent \textbf{Non-convex optimization}. Lastly, our paper is related to the growing literature on non-convex optimization. In the machine learning community, there has been considerable interest in how to solve non-convex optimization problems, since many learning tasks can be naturally expressed as non-convex optimization problems. Since non-convex optimization problems are in general NP-Hard, a popular approach for tackling such problems is based on convex relaxation, where one relaxes the problem in some way to obtain a convex problem that is more tractable. However, as pointed out by \cite{jain2017non}, such convex relaxations generally change the problem drastically, and thus the solution of relaxation can perform poorly for the original problem. Because of this, there has been much recent work on directly solving the non-convex problems via approximate algorithms. Efficient techniques used in non-convex optimization approach include generalized projected gradient descent \citep{candes2015phase}, generalized alternating minimization \citep{netrapalli2015phase}, and stochastic optimization techniques \citep{ge2015escaping}. Although these approaches are not guaranteed to find the global optimum in general, it has been empirically observed that approximately optimal solutions to the true non-convex problem are often better than exactly optimal solutions to a convex relaxation of the problem \citep{jain2017non}.

In our paper, the optimal stopping problem of learning randomized policies from sample data is formulated as a non-convex optimization problem. We follow the spirit of non-convex optimization approaches and propose a backward optimization heuristic to directly work with this non-convex problem, which sequentially optimizes over the weights in each time period. In our implementation of this method, the weights in each time period are approximately optimized using the Adam algorithm \citep{kingma2014adam}, a first-order method that is widely used for non-convex optimization problems, particularly those arising in the training of deep neural networks. Although our heuristic is not guaranteed to find a globally optimal solution, we find numerically that the resulting policies can significantly outperform those obtained by LSM.

\section{Problem Definition}
\label{sec:problem_definition}

In this section, we begin by defining our optimal stopping problem (Section~\ref{subsec:problem_definition_optimal_stopping}). We then define the family of deterministic linear policies, and the problems of optimizing over deterministic linear policies given complete knowledge of the stochastic process (Section~\ref{subsec:problem_definition_deterministic}) and given a sample of trajectories (Section~\ref{subsec:problem_definition_deterministic_SAA}). In Section~\ref{subsec:problem_definition_randomized}, we define the family of randomized linear policies and analogously to the deterministic linear policy case, we define the true stochastic optimization problem for this policy class and its finite sample counterpart. Finally, in Section~\ref{subsec:problem_definition_equivalence}, we state our main equivalence results, which assert that (i) the sample average approximation problems over deterministic and randomized linear policies are equivalent and (ii) the true stochastic optimization problems over deterministic and randomized linear policies are equivalent.

\subsection{Optimal stopping problem}
\label{subsec:problem_definition_optimal_stopping}

We consider a stochastic system that evolves over a discrete time horizon of $T$ periods. Each period is denoted by $t$, and ranges in $[T]$, where we use the notation $[n]$ to denote the set $\{1,\dots, n\}$ for any integer $n$. We use $\xb$ to denote the state of the system, and $\xb(t)$ to denote the state of the system in each period, which belongs to a state space $\Xcal$. At each period, we can choose to stop the system or to continue for one more period. If we choose to stop, we receive a nonnegative reward $g(t,\xb)$ that is a function of the period $t$ and the current state $\xb$. If we continue, we do not receive a reward. The action space of the problem is therefore $\Acal = \{\stop, \go\}$. %

The decision maker has the ability to specify a deterministic policy $\pi: [T] \times \Xcal \to \Acal$, which is a mapping from the current period and state we are in to one of the two actions. The policy $\pi$ defines a stopping time $\tau_{\pi}$, which is a random variable that represents the time in $[T]$ at which the decision maker stops:
\begin{equation}
\tau_{\pi} = \min\{ t \in [T] \mid \pi(t, \xb(t)) = \stop \}.
\end{equation}
We denote the case that the system is never stopped by $\tau_{\pi} = +\infty$, and we assume that the reward is zero in this case, i.e., $g(+\infty,\xb) = 0$ for all $\xb \in \Xcal$. 

Letting $\Pi$ denote the set of all policies, the decision maker's goal is to specify the policy $\pi$ that maximizes the expected discounted reward, which can be written as the following optimization problem:
\begin{equation}
\underset{\pi \in \Pi}{\text{supremum}} \ \Exp[ g(\tau_{\pi}, \xb(\tau_{\pi})) ]. \label{prob:optimal_stopping_abstract}
\end{equation}
We make two important remarks regarding our optimal stopping problem~\eqref{prob:optimal_stopping_abstract}. First, we note that our formulation does not include a discount factor, which is common in the optimal stopping literature. Our motivation for this modeling choice was to simplify the mathematical exposition and to make certain expressions that appear later on less cumbersome. We also note that this is not a restrictive modeling choice, as the reward function $g$ is time dependent, and so one can specify it so as to incorporate discounting. Second, for the entirety of the paper, we shall assume that $g$ is uniformly bounded, which we formalize in the following assumption.

\begin{assumption}
	There exists a finite upper bound $\bar{G}$ such that for any $t \in [T]$, $\xb \in \Xcal$, $0 \leq g(t,\xb) \leq \bar{G}$.
	\label{assumption:bounded_payoff}
\end{assumption}

\subsection{Deterministic linear policies}
\label{subsec:problem_definition_deterministic}

The optimal stopping problem~\eqref{prob:optimal_stopping_abstract} is a challenging problem to solve because the set of policies is unrestricted. Rather than working with the set of all policies, we will consider the set of policies that can be described using a linear combination of basis functions. Specifically, let us define $\phi_1,\dots, \phi_K: \Xcal \to \mathbb{R}$ to be a collection of basis functions, which map a state to a real number; for convenience, we will use $\Phi(\xb) = (\phi_1(\xb), \dots, \phi_K(\xb))$ to denote the vector of basis functions. Let us also define $\bb_t =(b_{t,1}, \dots, b_{t,K}) \in \mathbb{R}^K$ to be a $K$-dimensional vector of weights corresponding to the policy at period $t \in [T]$, and additionally, let us use $\bb$ to denote the collection of $\bb_t$ vectors, i.e., $\bb = (\bb_1,\dots, \bb_T)$. We can then define the policy $\pi_{\bb}$ as the policy that recommends stopping whenever the weighted combination of basis functions, where the weights come from $\bb$, is positive:
\begin{equation}
\pi_{\bb}(t, \xb) = \left\{ \begin{array}{llll} 
\stop &&& \text{if}\ \sum_{k=1}^K b_{t,k} \phi_k(\xb(t)) > 0, \\
\go &&& \text{otherwise}. 
\end{array} \right.
\label{eq:deterministic_policy_definition}
\end{equation}
We let $\Bcal \subseteq \mathbb{R}^{KT}$ be the set of feasible weight vectors, and let $\Pi_{\Bcal}$ be the corresponding set of linear policies:
\begin{equation*}
\Pi_{\Bcal} = \{ \pi_{\bb} \mid \bb \in \Bcal \}.
\end{equation*}
The linear policy optimal stopping problem can then be written as:
\begin{equation}
\underset{\pi \in \Pi_{\Bcal} }{\text{supremum}} \ \Exp[ g(\tau_{\pi}, \xb(\tau_{\pi})) ]. \label{prob:full_deterministic_abstract}
\end{equation}
Note that we can re-write this problem without the use of the stopping time $\tau_{\pi}$, and to make the dependence on $\bb$ more explicit, as follows:
\begin{equation}
\underset{\bb \in \Bcal }{\text{supremum}} \ \Exp \left[  \sum_{t=1}^{T}  g(t, \xb(t)) \cdot \prod_{t'=1}^{t-1} \Ibb\{ \bb_{t'} \bullet \Phi(\xb(t')) \leq 0\} \cdot \Ibb\{ \bb_t \bullet \Phi(\xb(t)) > 0\}  \right], \label{prob:full_deterministic}
\end{equation}
where we use $\Ibb\{\cdot\}$ to denote the indicator function (i.e., $\Ibb\{A\} = 1$ if $A$ is true, and 0 if $A$ is false), and for notational convenience, we use $\bullet$ to denote inner products, i.e., for $\ab, \bb \in \mathbb{R}^n$, $\ab \bullet \bb = \sum_{i=1}^n a_i b_i$. Note that the term $\prod_{t'=1}^{t-1} \Ibb\{ \bb_{t'} \bullet \Phi(\xb(t')) \leq 0\} \cdot \Ibb\{ \bb_t \bullet \Phi(\xb(t)) > 0\} $ is equal to 1 if and only if $\tau_{\pi} = t$; thus, this problem is equivalent to problem~\eqref{prob:full_deterministic_abstract}. We also use $J_D(\bb)$ to denote the objective value of problem~\eqref{prob:full_deterministic} at a fixed weight vector $\bb$.

\subsection{Data-driven optimization over deterministic linear policies}
\label{subsec:problem_definition_deterministic_SAA}

While problem~\eqref{prob:full_deterministic} is a simplification of the general optimal stopping problem~\eqref{prob:optimal_stopping_abstract}, it is still challenging to solve as it requires one to compute expectations over the stochastic process $\{ \xb(t) \}_{t=1}^T$ exactly. More specifically, this problem is challenging because the stochastic process is sufficiently complicated that optimizing over the objective function of problem~\eqref{prob:full_deterministic} is computationally difficult, or because the stochastic process itself is not known exactly. Thus, rather than considering the exact version of the problem, one can consider solving a sample-average approximation (SAA) version of the problem, wherein one has access to a set of trajectories of the stochastic process.

To define this problem, we assume that we have access to a set of $\Omega$ trajectories and that each trajectory is indexed by $\omega$, which ranges from 1 to $\Omega$. Each trajectory $\omega$ corresponds to a sequence of states $\xb(\omega,1), \xb(\omega,2), \dots, \xb(\omega,t)$. Given a policy and a trajectory $\omega$, we define the stopping time for policy $\pi$ in trajectory $\omega$ as 
\begin{equation*}
\tau_{\pi,\omega} = \min \{ t \in [T] \mid \pi(t, \xb(\omega,t)) = \stop \}.
\end{equation*}
Our SAA problem to determine the optimal linear policy is then
\begin{equation}
\underset{\pi \in \Pi_{\Bcal}}{\text{supremum}} \ \frac{1}{\Omega} \sum_{\omega=1}^{\Omega} g(\tau_{\pi,\omega}, \xb(\omega, \tau_{\pi,\omega})). 
 \label{prob:SAA_deterministic_abstract}
\end{equation}
Similarly to problem~\eqref{prob:full_deterministic}, we can re-write problem~\eqref{prob:SAA_deterministic_abstract} as an optimization problem over $\bb$ as follows:
\begin{equation}
\underset{\bb \in \Bcal}{\text{supremum}} \ \frac{1}{\Omega} \sum_{\omega=1}^{\Omega}\sum_{t=1}^{T}  g(t, \xb(\omega, t)) \cdot \prod_{t'=1}^{t-1} \Ibb\{ \bb_{t'} \bullet \Phi(\xb(\omega,t')) \leq 0\} \cdot \Ibb\{ \bb_t \bullet \Phi(\xb(\omega,t)) > 0\}. \label{prob:SAA_deterministic}
\end{equation}
Note that the term $\prod_{t'=1}^{t-1} \Ibb\{ \bb_{t'} \bullet \Phi(\xb(\omega,t')) \leq 0\} \cdot \Ibb\{ \bb_t \bullet \Phi(\xb(\omega,t)) > 0\}$ is equal to 1 if and only if $\tau_{\pi_{\bb},\omega} = t$. Additionally, we use $\hat{J}_D(\bb)$ to denote the objective value of problem~\eqref{prob:SAA_deterministic} at a fixed weight vector $\bb$. 

By re-writing problem~\eqref{prob:SAA_deterministic_abstract} as problem~\eqref{prob:SAA_deterministic}, we can see that the deterministic policy SAA problem~\eqref{prob:SAA_deterministic} can be regarded as a type of discrete optimization problem over the weight vector $\bb$. (Note that the supremum in problem~\eqref{prob:SAA_deterministic} is always attainable and can be replaced by a maximum, since the objective function $\hat{J}_D(\cdot)$ only takes finitely many values.) While this problem can be further re-formulated as a mixed-integer optimization problem, it is unlikely that one would be able to solve such a formulation to provable full or near optimality at a large scale (with tens of thousands or hundreds of thousands of trajectories). Moreover, the gradient of the objective function in problem~\eqref{prob:SAA_deterministic}, when it is defined, is always zero due to the presence of the indicator function. This precludes the use of gradient-based methods, such as stochastic gradient descent, for solving the problem.

\subsection{Randomized linear policies}
\label{subsec:problem_definition_randomized}

Rather than solving problems \eqref{prob:full_deterministic} and \eqref{prob:SAA_deterministic}, which optimize over deterministic linear policies, we can instead consider a problem where we optimize over randomized linear policies. In particular, given a collection of coefficients $\bb = (\bb_1, \dots, \bb_T)$ where $\bb_1, \dots, \bb_T \in \mathbb{R}^K$ we consider randomized linear policies of the form
\begin{equation*}
\tilde{\pi}_{\bb}(t, \xb) = \left\{ \begin{array}{llll} \stop &&& \text{with probability}\ \sigma( \bb_t \bullet \Phi(\xb) ), \\ 
\go &&& \text{with probability}\ 1 - \sigma( \bb_t \bullet \Phi(\xb) ),  \end{array} \right.
\end{equation*}
where $\sigma(u) = e^u / (1+ e^u)$ corresponds to the logistic response function, and where the decision to stop in period $t$ is independent of periods $1, \dots, t-1$. Thus, given the coefficients in $\bb$, the randomized policy $\tilde{\pi}_{\bb}$ randomly chooses to stop with a logistic probability that depends on a weighted sum of basis functions.

The stopping time $\tau_{\tilde{\pi}}$ of a randomized policy $\tilde{\pi}$ is defined as follows. Conditional on a fixed trajectory $\{\xb(t)\}_{t=1}^T$, the stopping time $\tau_{\tilde{\pi}}$ is a random variable, whose probability distribution is given by 
\begin{align*}
\Pbb( \tau_{\tilde{\pi}} = t \mid \xb(1), \dots, \xb(T) ) & = \prod_{t'=1}^{t-1} (1 - \sigma( \bb_{t'} \bullet \Phi(\xb(t')) )) \cdot \sigma ( \bb_t \bullet \Phi(\xb(t)) ), \quad t = 1,\dots, T, \\
\Pbb( \tau_{\tilde{\pi}} = +\infty \mid \xb(1), \dots, \xb(T) ) & = \prod_{t'=1}^{T} (1 - \sigma( \bb_{t'} \bullet \Phi(\xb(t')) )).
\end{align*}
With a slight abuse of notation, let $\Bcal \subseteq \mathbb{R}^{KT}$ denote the set of feasible weight vectors for randomized policies, and define $\tilde{\Pi}_{\Bcal}$ to be the set of feasible randomized policies:
\begin{equation*}
\tilde{\Pi}_{\Bcal} = \{ \tilde{\pi}_{\bb} \mid \bb \in \Bcal \}.
\end{equation*}
Thus, the expected reward of the randomized policy $\tilde{\pi}_{\bb}$, where the expectation is taken over both the stochastic process $\{ \xb(t) \}_{t=1}^T$ and the random stopping decisions can be written as
\begin{equation}
\underset{\tilde{\pi} \in \tilde{\Pi}_{\Bcal} }{\text{supremum}} \ \Exp[  g(\tau_{\tilde{\pi}}, \xb(\tau_{\tilde{\pi}})) ], \label{prob:optimal_stopping_random_policy}
\end{equation}
or equivalently, as 
\begin{equation}
\underset{ \bb \in \Bcal }{\text{supremum}} \ \Exp\left[  \sum_{t=1}^T g(t, \xb(t)) \cdot \prod_{t'=1}^{t-1} (1 - \sigma( \bb_{t'} \bullet \Phi(\xb(t')) )) \cdot \sigma ( \bb_t \bullet \Phi(\xb(t)) ) \right], \label{prob:full_randomized}
\end{equation}
where the expectation in problem~\eqref{prob:full_randomized} is now taken only over the stochastic process  $\{ \xb(t) \}_{t=1}^T$. We shall use $J_R(\bb)$ to denote the objective function of problem~\eqref{prob:full_randomized} at a fixed $\bb \in \Bcal$. 

Similarly to the deterministic problem, we can also consider a sample-average approximation of the true stochastic optimization problem~\eqref{prob:full_randomized}. Given a sample of $\Omega$ trajectories as in Section~\ref{subsec:problem_definition_deterministic_SAA}, we can define the randomized policy SAA problem as
\begin{equation}
\underset{\bb \in \Bcal }{\text{supremum}} \ \frac{1}{\Omega} \sum_{\omega = 1}^{\Omega} \sum_{t=1}^T  g(t, \xb(\omega,t)) \cdot \prod_{t'=1}^{t-1} (1 - \sigma( \bb_{t'} \bullet \Phi(\xb(\omega,t')) )) \cdot \sigma ( \bb_t \bullet \Phi(\xb(\omega,t)) ).
\label{prob:SAA_randomized}
\end{equation}
In other words, we seek to find the coefficients $\bb = (\bb_1, \dots, \bb_T)$ so as to maximize the expected sample-average reward that arises from using these coefficients to effect randomized stopping decisions. We note that in problem~\eqref{prob:SAA_randomized}, the optimization problem is formulated using the supremum. This is necessary, because although the objective function of \eqref{prob:SAA_randomized} is continuous and bounded, the set $\Bcal$ may not be compact, and therefore there may not have an attainable maximum. We shall use $\hat{J}_R(\bb)$ to denote the objective function of the randomized policy at a fixed weight vector $\bb \in \Bcal$.

\subsection{Equivalence of deterministic and randomized policies}
\label{subsec:problem_definition_equivalence}

In this section, we investigate the connection between the deterministic policy problems laid out in Sections~\ref{subsec:problem_definition_deterministic} and \ref{subsec:problem_definition_deterministic_SAA}, and the randomized policy problems in Section~\ref{subsec:problem_definition_randomized}. It turns out that under a small set of conditions, it is possible to show that the optimal objective values of the deterministic policy SAA problem~\eqref{prob:SAA_deterministic} and the randomized policy SAA problem~\eqref{prob:SAA_randomized} are equivalent. With one additional assumption, it is also possible to show that the optimal objective values of the deterministic and randomized policy true problems (problems~\eqref{prob:full_deterministic} and \eqref{prob:full_randomized} respectively) are also equivalent.

Recall that $J_D(\cdot)$, $\hat{J}_D(\cdot)$ $J_R(\cdot)$ and $\hat{J}_R(\cdot)$ are the respective objective functions of the deterministic policy true problem~\eqref{prob:full_deterministic}, the deterministic policy SAA problem~\eqref{prob:SAA_deterministic}, the randomized policy true problem~\eqref{prob:full_randomized} and the randomized policy SAA problem~\eqref{prob:SAA_randomized}. 
For the purposes of the exposition of this section, we will use $\tilde{\bb}$ to denote a vector of weights for the randomized policy problem, while $\bb$ will be used to denote a vector of weights for the deterministic policy problem. We will also further disambiguate the sets of feasible weight vectors for the two problems by using $\Bcal$ to denote the set of feasible weight vectors for the deterministic problem, and $\tilde{\Bcal}$ the set of feasible weight vectors for the randomized problem. 

Before stating our first result, we make two assumptions. Our first assumption is that the set of feasible weight vectors for the deterministic policy and randomized policy SAA problems are the same.
\begin{assumption}
$\Bcal = \tilde{\Bcal} = \mathbb{R}^{KT}$. \label{assumption:Bcal_tildeBcal_RKT}
\end{assumption}
Our second assumption concerns the collection of basis functions.
\begin{assumption}
The first basis function $\phi_1(\cdot)$ is the constant basis function, i.e., $\phi_1(\xb) = 1$ for all $\xb \in \Xcal$. \label{assumption:constant_basis_function}
\end{assumption}
With these two assumptions, we state our first main result. 
\begin{theorem}
Under Assumptions~\ref{assumption:Bcal_tildeBcal_RKT} and \ref{assumption:constant_basis_function} the objective values of problems~\eqref{prob:SAA_deterministic} and \eqref{prob:SAA_randomized} are equal, that is,
\begin{equation*}
\sup_{\bb \in \Bcal} \hat{J}_{D}(\bb) = \sup_{\tilde{\bb} \in \tilde{\Bcal}} \hat{J}_R(\tilde{\bb}).
\end{equation*}
\label{theorem:deterministic_randomized_SAA_equal}
\end{theorem}
The proof of Theorem~\ref{theorem:deterministic_randomized_SAA_equal} (see Section~\ref{proof_theorem:deterministic_randomized_SAA_equal} of the ecompanion) is based on two key ideas: (1) given a weight vector $\bb$ of a deterministic policy, the same weight vector scaled by an arbitrarily large positive constant $\alpha$ would result in the randomized policy behaving in the same (deterministic) way, since $\sigma(u) \to 1$ as $u \to \infty$ and $\sigma(u) \to 0$ as $u \to -\infty$; and (2) given a weight vector $\tilde{\bb}$ of a randomized policy, one can view $\hat{J}_R(\tilde{\bb})$ as the expectation of a deterministic policy with a particular basis function weight chosen randomly, so applying the probabilistic method implies the existence of a weight vector for a deterministic policy that performs at least as well as the randomized policy. With regard to the assumptions, Assumption~\ref{assumption:Bcal_tildeBcal_RKT} is a technical assumption that is necessary to be able to scale a deterministic weight vector into an appropriate randomized policy, as in idea (1), while Assumption~\ref{assumption:constant_basis_function} is a technical assumption that is necessary to avoid pathological cases where $\bb_t \bullet \Phi(\xb) = 0$ and to be able to appropriately apply the probabilistic method as in idea (2). From a practical perspective, Assumption~\ref{assumption:constant_basis_function} is not too restrictive, as it is common to use a constant basis function in implementations of ADP for optimal stopping. %

Theorem~\ref{theorem:deterministic_randomized_SAA_equal} asserts that the SAA formulations of the two policy optimization problems are essentially equivalent. To establish equivalence of the true deterministic and randomized policy optimization problems \eqref{prob:full_deterministic} and \eqref{prob:full_randomized}, we need the following additional assumption, which concerns the stochastic process itself. We defer our discussion of this assumption until the statement of Theorem~\ref{theorem:deterministic_randomized_full_equal}. To state this assumption, we let $\Phi_{2:K}: \Xcal \to \mathbb{R}^{K-1}$ be defined as $\Phi_{2:K}(\xb) = (\phi_2(\xb), \dots, \phi_K(\xb))$, which is just the vector-valued mapping of the state $\xb$ to the basis function values $\phi_2(\xb)$ through $\phi_K(\xb)$ (in other words, it is just the mapping $\Phi$, only with the first basis function $\phi_1(\cdot)$ omitted). 
\begin{assumption}
For any hyperplane $A \subseteq \mathbb{R}^{K-1}$, i.e., a set of the form $A = \{ \yb \in \mathbb{R}^{K-1} \mid \cb \bullet \yb + d = 0 \}$ for some $\cb \in \mathbb{R}^{K-1}$, $d \in \mathbb{R}$, and any $t \in [T]$, $\Pbb( \Phi_{2:K}(\xb(t)) \in A) = 0$.
\label{assumption:measure_zero}
\end{assumption}
We can now state our counterpart of Theorem~\ref{theorem:deterministic_randomized_SAA_equal} for the true stochastic optimization problems~\eqref{prob:full_randomized} and \eqref{prob:full_deterministic}. 
\begin{theorem}
Under Assumptions~\ref{assumption:Bcal_tildeBcal_RKT}, \ref{assumption:constant_basis_function} and \ref{assumption:measure_zero} the objective values of the randomized problem~\eqref{prob:full_randomized} and the deterministic problem~\eqref{prob:full_deterministic} are equal, that is,
\begin{equation*}
\sup_{\bb \in \Bcal} J_D(\bb) = \sup_{\tilde{\bb} \in \tilde{\Bcal}} J_R(\tilde{\bb}).
\end{equation*} \label{theorem:deterministic_randomized_full_equal}
\end{theorem}
The proof of Theorem~\ref{theorem:deterministic_randomized_full_equal} (see Section~\ref{proof_theorem:deterministic_randomized_full_equal} of the ecompanion) is similar to the proof of Theorem~\ref{theorem:deterministic_randomized_SAA_equal}, but with several key differences. The most significant difference is that in the proof of Theorem~\ref{theorem:deterministic_randomized_SAA_equal}, we show that a given deterministic linear policy can be approximated arbitrarily closely by a randomized policy. This is facilitated by Assumption~\ref{assumption:constant_basis_function}, which allows one to avoid situations where the inner product of $\bb_t$ and $\Phi(\xb(\omega,t))$ is exactly zero in a given $\omega$ and $t$ (since there are finitely many trajectories, one can perturb a given deterministic weight vector $\bb$ into a new deterministic weight vector $\bb'$ that has the same stopping behavior but never satisfies $\bb_t \bullet \Phi(\xb(\omega,t)) = 0$ for any $\omega$ and any $t$). In the true stochastic optimization problem setting, this is no longer possible. For this reason, we introduce Assumption~\ref{assumption:measure_zero}, which requires that $\Phi_{2:K}(\xb(t))$ has probability zero of being in any given hyperplane. This assumption allows us to avoid the aforementioned pathological cases where the stochastic process is such that, for a given non-zero weight vector $\bb$ for the randomized policy problem, the inner product $\bb_t \bullet \Phi(\xb(t))$ may be exactly zero, which would mean the randomized policy would choose to stop or continue with equal probability. 

With regard to Assumption~\ref{assumption:measure_zero}, we note that this assumption holds for many, though not all, problem instances. For example, suppose that $\Xcal \subseteq \mathbb{R}^n$ and $\phi_2(\xb), \dots, \phi_K(\xb)$ are polynomials of $\xb \in \Xcal$. In this case, the set $\{ \xb \in \Xcal \mid \cb \bullet \Phi_{2:K}(\xb) + d = 0\}$ is the set of zeros of a polynomial function of $\xb$, which is a measure zero set \citep{okamoto1973distinctness}. If we further assume that $\xb(t)$ at each $t$ has a bounded density, which is the case for many commonly used stochastic processes (e.g., geometric Brownian motion), then it immediately follows that $\Pbb( \Phi_{2:K}(\xb(t)) \in A) = 0$ for any hyperplane $A \subseteq \mathbb{R}^{K-1}$. As another example, suppose that $\Xcal = \mathbb{R}^{K-1}$, and define $E$ as $E = \Phi_{2:K}(\Xcal)$, the image of $\Xcal$ under $\Phi_{2:K}(\cdot)$, which we assume to be an open subset of $\mathbb{R}^{K-1}$. Suppose also that the inverse function $\Phi^{-1}_{2:K}(\cdot)$ is defined on $E$ and is continuously differentiable. Then the event $\Phi_{2:K}(\xb(t)) \in A$ for a hyperplane $A \subseteq \mathbb{R}^{K-1}$ is equivalent to the event $\Phi_{2:K}(\xb(t)) \in A \cap E$, which is equivalent to the event $\xb(t) \in \Phi^{-1}_{2:K}(A \cap E)$. If $A$ is a hyperplane in $\mathbb{R}^{K-1}$, it has measure zero, and so does $A \cap E$; and since $\Phi^{-1}_{2:K}(\cdot)$ is continuously differentiable, $\Phi^{-1}_{2:K}(A \cap E)$ is also a measure zero set in $\mathbb{R}^{K-1}$ \citep[see Lemma 18.1 of][]{munkres1991analysis}. If we again assume that each $\xb(t)$ has a bounded density, then it again follows that $\Pbb( \xb(t) \in \Phi^{-1}_{2:K}(A \cap E)) = 0$ or equivalently, $\Pbb( \Phi_{2:K}(\xb(t)) \in A) = 0$. Where Assumption~\ref{assumption:measure_zero} could potentially fail is when the basis function mapping $\Phi_{2:K}(\cdot)$ collapses subsets of $\Xcal$ to singletons, which could cause the probability of $\Phi_{2:K}(\xb(t))$ being in certain hyperplanes to be non-zero.

We conclude this section by offering two remarks on Theorems~\ref{theorem:deterministic_randomized_SAA_equal} and \ref{theorem:deterministic_randomized_full_equal}. First, the significance of these two theorems is that in a certain sense, the problem of optimizing over deterministic policies and the problem of optimizing over randomized policies are the same. In the case of the true stochastic optimization problems, by solving the randomized problem~\eqref{prob:full_randomized}, we can obtain a policy that performs as well as the one we would obtain by solving the deterministic problem~\eqref{prob:full_deterministic}. Similarly, in the case when we are working with a finite sample of trajectories, solving the randomized SAA problem~\eqref{prob:SAA_randomized} allows us to obtain a policy that performs as well as the one we would obtain by solving the deterministic SAA problem~\eqref{prob:SAA_deterministic}. From a practical perspective, the advantage of solving the randomized policy SAA problem~\eqref{prob:SAA_randomized}, as opposed to the deterministic policy SAA problem~\eqref{prob:SAA_deterministic}, is that the objective function $\hat{J}_R(\cdot)$ is a differentiable function. Although $\hat{J}_R(\cdot)$ is non-convex due to the presence of the logistic response function $\sigma(\cdot)$, it is at least possible to approximately optimize $\hat{J}_R(\cdot)$ using gradient-based methods. The specific structure of $\hat{J}_R(\cdot)$ lends itself to an iterative algorithm that optimizes the weight vector $\tilde{\bb}$ one period at a time, starting with the last period, that is reminiscent of the least-squares Monte Carlo (LSM) method; we defer our presentation of this algorithm to Section~\ref{subsec:solution_methodology_algorithm}. 

Second, we comment a little more on the motivation of our randomized policy optimization approach, in light of Theorems~\ref{theorem:deterministic_randomized_SAA_equal} and \ref{theorem:deterministic_randomized_full_equal}. Our interest in randomized linear policies does not stem from some fundamental operational benefit that a randomized policy provides over a deterministic policy; stated differently, we do not wish to argue that in practice, a decision maker would want to make stopping decisions randomly as opposed to deterministically. Instead, our motivation for studying randomized policies is that the use of the logistic response function $\sigma$ allows use to view the randomized policy true problem~\eqref{prob:full_randomized} and the SAA problem~\eqref{prob:SAA_randomized} as differentiable or ``soft'' counterparts to the deterministic policy problems \eqref{prob:full_deterministic} and \eqref{prob:SAA_deterministic}, respectively, which are formulated using the indicator function $\Ibb\{ \cdot \}$ and involve making ``hard'' stopping decisions. Theorems~\ref{theorem:deterministic_randomized_SAA_equal} and \ref{theorem:deterministic_randomized_full_equal} show that in general, this view is justified, as the deterministic and randomized problems are equal in objective value. As we will shortly see, the randomized policy SAA problem is amenable to an analysis of its convergence and generalization properties, and as we have already mentioned, is amenable to an intuitive heuristic for approximately solving it. Later, in Section~\ref{sec:numerics}, we will see numerically that using an approximate solution of the randomized policy SAA problem within a deterministic policy performs very well and can result in significant improvements over existing approaches.

\section{Statistical properties}
\label{sec:statistical_properties}
In this section, we investigate the statistical properties of the randomized policy SAA problem~\eqref{prob:SAA_randomized}. In Section~\ref{subsection:convergence} we show that the objective value and optimal solution set of the randomized policy SAA problem converge almost surely to those of the true randomized policy problem. In Section~\ref{subsection:rademacher}, we establish guarantees on the out-of-sample performance of the solution obtained from the randomized policy SAA problem by characterizing the Rademacher complexity of the expected reward generated by a given set of weight vectors. 

\subsection{Convergence of randomized policy SAA problem}
\label{subsection:convergence}
It is natural to expect that the optimal value and optimal solutions of the SAA problem \eqref{prob:SAA_randomized} converge to their counterparts of the true optimization problem as the number of sample trajectories $\Omega\rightarrow\infty$. In this section, we provide Theorems~\ref{theorem:SAA_uniform_convergence} and \ref{theorem:SAA_solution_convergence} to establish these two convergence properties of our randomized policy SAA problem. 

We first make the following two mild assumptions to facilitate the proofs of Theorems~\ref{theorem:SAA_uniform_convergence} and \ref{theorem:SAA_solution_convergence}. 
\begin{assumption}
	There exists a constant $Q > 0$ such that for any $\xb \in \Xcal$, $\| \Phi(\xb) \|_{\infty}\leq Q$.
	\label{assumption:bounded_basis}
\end{assumption}

\begin{assumption}
$\Bcal$ is a compact subset of $\mathbb{R}^{KT}$. 
	\label{assumption:Bcal_compact}
\end{assumption}

Note that we no longer carry Assumptions~\ref{assumption:Bcal_tildeBcal_RKT}, \ref{assumption:constant_basis_function} and \ref{assumption:measure_zero}. In particular, Assumption~\ref{assumption:Bcal_tildeBcal_RKT} is not relevant to Theorems~\ref{theorem:SAA_uniform_convergence} and \ref{theorem:SAA_solution_convergence}, and Assumptions~\ref{assumption:constant_basis_function} and \ref{assumption:measure_zero} are not required to establish our results here.

With these two assumptions, we can establish the following theorem which shows the almost sure uniform convergence of $\hat{J}_R(\cdot)$ to $J_R(\cdot)$ over the set $\Bcal$. 
\begin{theorem}
\label{theorem:SAA_uniform_convergence}
Suppose that Assumptions~\ref{assumption:bounded_basis} and \ref{assumption:Bcal_compact} both hold. Then with probability one, 
\begin{equation}
\lim_{\Omega \to \infty} \sup_{\bb \in \Bcal} \ | \hat{J}_R(\bb) - J_R(\bb) | = 0.
\end{equation}
\end{theorem}
The proof of Theorem~\ref{theorem:SAA_uniform_convergence} is provided in Section~\ref{proof_theorem:SAA_uniform_convergence} of the ecompanion. It relies on the fact that the objective function $J_R(\cdot)$ in the true problem \eqref{prob:full_randomized} and the objective function $\hat{J}_R(\cdot)$ in the SAA problem \eqref{prob:SAA_randomized} have bounded Lipschitz constants, and the compactness of $\Bcal$. Thus, we can use these two properties, together with the strong law of large numbers, to show uniform convergence. 

Now, using Theorem~\ref{theorem:SAA_uniform_convergence}, it is straightforward to derive the convergence of the SAA optimal objective value, which is stated in the following corollary. 
\begin{corollary} \label{corollary:SAA_objective_convergence}
Suppose that Assumptions~\ref{assumption:bounded_basis} and \ref{assumption:Bcal_compact} both hold. Then with probability one, 
\begin{equation}
\lim_{\Omega \to \infty} \sup_{\bb \in \Bcal} \hat{J}_R(\bb)	 = \sup_{\bb \in \Bcal} J_R( \bb ).
\end{equation}	

\end{corollary}

For the convergence of the SAA optimal solutions, let us define the sets $\Bb^*$ and $\hat{\Bb}$ as 
\begin{align*}
\Bb^*  = \arg \max_{\bb \in \Bcal} J_R(\bb), \\
\hat{\Bb} = \arg \max_{\bb \in \Bcal} \hat{J}_R(\bb),
\end{align*}
that is, $\Bb^*$ is the set of optimal solutions of the true stochastic problem~\eqref{prob:full_randomized} while $\hat{\Bb}$ is the set of optimal solutions to the SAA problem~\eqref{prob:SAA_randomized}. In addition, let $\Dbb( \hat{\Bb}, \Bb^*)$ be the \emph{deviation} \citep[see Chapter 7 of][]{shapiro2014lectures} of the set $\hat{\Bb}$ from $\Bb^*$, that is,
\begin{equation*}
\Dbb( \hat{\Bb}, \Bb^* ) = \sup_{\bb \in \hat{\Bb}} \inf_{\bb' \in \Bb^*} \| \bb - \bb' \|_2.
\end{equation*}
In the above definition, the inner infimum measures the distance between a given optimal solution $\bb$ of the SAA problem~\eqref{prob:SAA_randomized} and the closest optimal solution of the true problem~\eqref{prob:full_randomized}; the outer supremum then takes the largest such distance, over all optimal solutions of the SAA problem. 

With these definitions, we can now apply Theorem 5.3 in \cite{shapiro2014lectures} to establish the following theorem:
\begin{theorem}
Suppose that Assumptions~\ref{assumption:bounded_basis} and \ref{assumption:Bcal_compact} both hold. Then with probability one, $\Dbb( \hat{\Bb}, \Bb^*) \to 0$ as $\Omega \to \infty$. 
\label{theorem:SAA_solution_convergence}
\end{theorem}

Corollary~\ref{corollary:SAA_objective_convergence} and Theorem~\ref{theorem:SAA_solution_convergence} indicate that, given a sufficiently large sample size, the weight vector obtained by solving the SAA optimization problem \eqref{prob:SAA_randomized} can be arbitrarily close to the optimal weight vector set of the true problem \eqref{prob:full_randomized}, and the corresponding optimal value of the SAA problem can be arbitrarily close to the optimal value of true problem.

\subsection{Rademacher Complexity}
\label{subsection:rademacher}
In Section \ref{subsection:convergence}, we have seen from Theorem~\ref{theorem:SAA_uniform_convergence} that the optimal SAA objective value $\sup_{\bb \in \Bcal} \hat{J}_R(\bb)$ converges with probability one to the true optimal objective value $\sup_{\bb \in \Bcal} J_R(\bb)$ as the number of trajectories goes to infinity. However, in practice, we can only have access to a finite number of sample trajectories; in other words, there always exists some gap between $J_R(\bb)$ and $\hat{J}_R(\bb)$. Therefore, it is important to investigate how far $\hat{J}_R(\bb)$ could be away from $J_R(\bb)$ for a finite sample size and find good bounds on this gap. In this section, we will use a classical data-dependent complexity estimate of a function class, Rademacher complexity, to lower-bound the value of $J_R(\bb)-\hat{J}_R(\bb)$, and provide three upper bounds on the Rademacher complexity term, corresponding to different choices of the weight vector set $\Bcal$. %

To establish this result, we require some additional definitions. We use $Y$ to denote a system realization, which is a pair consisting of the sequence of states and the sequence of rewards, that is, $Y = ( \{ \xb(t) \}_{t=1}^{T}, \{ g(t,\xb(t)) \}_{t=1}^T )$. We use $Y_1,\dots, Y_{\Omega}$ to denote the sample of system realizations. We define the function $\Gamma: \mathbb{R}^T \times [0, \bar{G}]^T \to \mathbb{R}$ as 
\begin{equation}
\Gamma( \ub, \vb) = \sum_{t=1}^T v_t \prod_{t'=1}^{t-1} (1 - \sigma(u_{t'})) \sigma(u_t),
\end{equation}
where $\ub, \vb \in \mathbb{R}^T$. For a fixed weight vector $\bb \in \Bcal$, we define the function $\psi_{\bb}: \Xcal^T \times \mathbb{R}^T \to \mathbb{R}^{2T}$ which maps a system realization $Y$ to a $2T$-dimensional vector as 
\begin{equation}
\psi_{\bb}(Y) = \left[ \begin{array}{c} \bb_1 \bullet \Phi(\xb(1)) \\ \vdots \\ \bb_T \bullet \Phi(\xb(T)) \\ g(1, \xb(1)) \\ \vdots \\ g(T, \xb(T))  \end{array} \right].
\end{equation}
We define $\Fcal = \{ \Gamma \circ \psi_{\bb} \mid \bb \in \Bcal\}$ as the class of realization-to-reward functions. Note that for a fixed weight vector $\bb$, the function value $(\Gamma \circ \psi_{\bb})(Y)$ gives exactly the expected reward of the randomized policy, where the expectation is taken over the stopping/continuation decisions, but conditional on the fixed system realization $Y$. 

Lastly, we define the empirical Rademacher complexity $\hat{R}(\Fcal)$ as 
\begin{equation}
\hat{R}(\Fcal) = \frac{1}{\Omega} \Exp_{\epsilonb} \left[ \sup_{ f \in \Fcal} \sum_{\omega=1}^{\Omega} \epsilon_{\omega} f(Y_{\omega}) \right],
\end{equation}
where $\epsilon_1,\dots,\epsilon_{\Omega}$ are independent Rademacher random variables, that is, each $\epsilon_{\omega}$ is equal to -1 or +1 with probability 1/2, and $\epsilonb$ is used to denote the vector of these random variables. We define the (ordinary) Rademacher complexity $R(\Fcal)$ as $R(\Fcal) = \Exp_{Y_1,\dots,Y_{\Omega}} [ \hat{R}(\Fcal) ]$.

Having set up the definitions of empirical Rademacher complexity and (ordinary/non-empirical) Rademacher complexity, Proposition~\ref{proposition:generalization_bound} establishes the lower bounds of $J_R(\bb)-\hat{J}_R(\bb)$ in terms of these two complexity terms.

\begin{proposition}
	Let $S = \{Y_1,\dots, Y_{\Omega}\}$ be a collection of independent and identically distributed system realizations. For all $\delta>0$, with probability at least $1-\delta$ over the sample $S$:
	\begin{equation}
	J_R(\bb)\geq\hat{J}_R(\bb)-2R(\Fcal)-\bar{G}\sqrt{\frac{\log(1/\delta)}{2\Omega}},\qquad \forall\ \mathbf{b}\in\Bcal
	\label{eq:rademacher_ordinary_generalization_bound}
	\end{equation}
	\begin{equation}
	\label{eq:rademacher_empirical_generalization_bound}
	J_R(\bb)\geq\hat{J}_R(\bb)-2\hat{R}(\Fcal)-3\bar{G}\sqrt{\frac{\log(2/\delta)}{2\Omega}},\qquad \forall\ \mathbf{b}\in\Bcal
	\end{equation}
	\label{proposition:generalization_bound}
\end{proposition}

The proof of Proposition~\ref{proposition:generalization_bound} is given in Section~\ref{proof_proposition:generalization_bound} of the ecompanion; it follows the standard proof of generalization error bounds based on Rademacher complexity in statistical learning theory. We remark here that the generalization bounds established in Proposition~\ref{proposition:generalization_bound} are different from those in classical statistical learning. Proposition~\ref{proposition:generalization_bound} provides lower bounds on the true reward $J_R(\bb)$ in the form of the sample-based estimate $\hat{J}_R(\bb)$ minus a penalty term related to the complexity of our model; whereas in classical statistical learning problems, Rademacher complexity is used to upper-bound the true error in the form of the training error plus the complexity term. The reason for this difference is that our problem is to maximize the expected reward, while the goal of classical statistical learning problem is to minimize some loss function. 

The key quantities in Proposition~\ref{proposition:generalization_bound} are the empirical and ordinary Rademacher complexities $R(\Fcal)$ and $\hat{R}(\Fcal)$. To understand how these quantities scale in the problem primitives and the structure of the admissible weight vector set $\Bcal$, we have the following result, which provides deterministic bounds on $\hat{R}(\Fcal)$. (Note that since these bounds on $\hat{R}(\Fcal)$ hold almost surely, they are also valid bounds on $R(\Fcal)$.)

\begin{theorem}
	Suppose that Assumption~\ref{assumption:bounded_basis} holds. Let $B \geq 0$. Then we have the following deterministic bounds for the empirical Rademacher complexity $\hat{R}(\Fcal)$:
	\begin{enumerate}
		\item[a)] If $\Bcal = \{ \bb \in \mathbb{R}^{KT} \mid \| \bb \|_1 \leq B\}$, then $\hat{R}(\Fcal) \leq \sqrt{2}(\bar{G} + 1) \cdot \frac{BQ \sqrt{ 2\log(2KT) } }{\sqrt{\Omega}}$. 
		\item[b)] If $\Bcal = \{ \bb \in \mathbb{R}^{KT} \mid \| \bb \|_2 \leq B\}$, then $\hat{R}(\Fcal) \leq \sqrt{2}(\bar{G} + 1) \cdot \frac{BQ \sqrt{KT} }{ \sqrt{\Omega}}$. 
		\item[c)] If $\Bcal = \{ \bb \in \mathbb{R}^{KT} \mid \| \bb \|_{\infty} \leq B\}$, then $\hat{R}(\Fcal) \leq \sqrt{2}(\bar{G} + 1) \cdot \frac{BQKT}{\sqrt{\Omega}}$.
	\end{enumerate}
	\label{theorem:main_rademacher_bound}
\end{theorem}

The proof of Theorem~\ref{theorem:main_rademacher_bound} (see Section~\ref{proof_theorem:main_rademacher_bound} of the ecompanion) consists of two main steps. The first step is relating the Rademacher complexity of $\Fcal$ to the Rademacher complexity of the class $\{ \psi_{\bb} \mid \bb \in \Bcal\}$. This involves the application of Maurer's vector contraction inequality \citep{maurer2016vector}, which is useful when a class of vector-valued functions is composed with a collection of scalar-valued Lipschitz functions, and can be used to relate the Rademacher complexity of the class of composite functions to the Rademacher complexity of the class of vector-valued functions. The outcome of this is that the Rademacher complexity of $\Fcal$ can be written in terms of the Rademacher complexity of $\{ \psi_{\bb} \mid \bb \in \Bcal\}$; in the second step, we analyze the Rademacher complexity of this latter class by exploiting the structure of $\Bcal$. 

From this result, we can see that in all three cases, the Rademacher complexity scales gracefully with the problem dimension. In the worst case (when $\Bcal$ is equal to the $L_{\infty}$ norm ball; part c), it scales at most linearly with $K$ and with $T$. This is partially driven by the fact that the function $\Gamma$ is Lipschitz continuous (with respect to the $L_2$ norm) with constant $\bar{G}+1$. Importantly, this constant does not depend on $T$. This is not obvious, because the probability of stopping at period $t$ is the product of $t$ Lipschitz continuous and bounded functions, and so by standard properties of Lipschitz functions one should expect the Lipschitz constant to depend on $T$. It turns out that one can avoid a dependence on $T$ because the products terms of the form $\prod_{t'=1}^{t-1} (1 - \sigma( \bb_{t'} \bullet \Phi(\xb(t')))) \sigma( \bb_t \bullet \Phi(\xb(t)))$ form a probability distribution. Consequently, the dependence on $T$ in the bounds in Theorem~\ref{theorem:main_rademacher_bound} arises from the structure of the set $\Bcal$, and not from the function $\Gamma$. 

\section{Solution Methodology}
\label{sec:solution_methodology}

We now turn our attention to how one can actually solve the randomized policy SAA problem~\eqref{prob:SAA_randomized}. In Section~\ref{subsec:solution_methodology_complexity}, we show that the randomized policy SAA problem is in general NP-Hard. Motivated by this, in Section~\ref{subsec:solution_methodology_algorithm} we propose an algorithm for approximately solving the SAA problem, based on backward induction. We conclude in Section~\ref{subsec:solution_methodology_comparison_LSM} by comparing our proposed heuristic algorithm with the LSM algorithm.

\subsection{Complexity of randomized policy SAA problem}
\label{subsec:solution_methodology_complexity}

Our main theoretical result on the solvability of the randomized policy SAA problem~\eqref{prob:SAA_randomized} is unfortunately a negative one.
\begin{theorem}
The randomized policy SAA problem \eqref{prob:SAA_randomized} is NP-Hard. \label{theorem:SAA_NPHard}
\end{theorem}

We make a few remarks about this result. First, our proof of Theorem~\ref{theorem:SAA_NPHard} (see Section~\ref{proof_theorem:SAA_NPHard} of the ecompanion) involves considering the decision form of the randomized policy SAA problem~\eqref{prob:SAA_randomized}, which asks whether there exists a weight vector $\bb$ that achieves at least a certain target sample-average reward. By considering this decision problem, we show that for any instance of the decision form of the MAX-3SAT problem, a well-known NP-Complete problem, we can construct a corresponding instance of the randomized policy SAA problem such that the answers to the two decision problems are identical. We note that the proof is not trivial, as the randomized policy SAA problem is in general a continuous problem, whereas MAX-3SAT is inherently discrete. In particular, showing that a positive answer to the SAA decision problem implies a positive answer to the MAX-3SAT problem involves viewing expressions involving $\sigma(\cdot)$ as expected values of expressions defined using a certain collection of i.i.d. random variables, and applying the probabilistic method to guarantee the existence of values for those random variables that can then be used to construct a solution to the MAX-3SAT problem. Most importantly, our proof does not achieve this equivalence by restricting the set of feasible weight vectors $\Bcal$ to be a discrete set: the only restriction we place is to restrict the weight vectors be equal across time (i.e., $b_{t,k} = b_{t',k}$ for $t \neq t'$), which still results in $\Bcal$ being uncountably infinite. %

Second, we note that from an intuition standpoint, it is not reasonable to expect the randomized policy SAA problem~\eqref{prob:SAA_randomized} to be tractable. As alluded to before, this problem is a non-convex optimization problem, due to the presence of the function $\sigma(\cdot)$ that is neither convex nor concave. In addition, as $\sigma(u)$ can be viewed as a continuous approximation of the step function $\Ibb\{ u \geq 0\}$, one can expect the function $\hat{J}_R(\cdot)$ to have many local optima. In the next section, we consider a heuristic approach for solving the problem.

\subsection{Backward optimization algorithm}
\label{subsec:solution_methodology_algorithm}

Motivated by the fact that our randomized policy SAA problem \eqref{prob:SAA_randomized} is theoretically intractable, we develop an iterative heuristic algorithm for solving the problem. 

The high level idea of our heuristic is to solve problem~\eqref{prob:SAA_randomized} by optimizing over the weights one period at a time, starting from the last one. In particular, recall that $\bb = (\bb_1,\dots, \bb_T)$ and with a slight abuse of notation, let $\hat{J}_R(\bb_1,\dots,\bb_T)$ denote the SAA objective value for the given collection of time-specific weight vectors. Assume also that that the set of feasible weight vectors is the Cartesian product of $T$ period-wise weight vector sets, that is, $\Bcal = \Bcal_1 \times \dots \times \Bcal_T$, where $\Bcal_1,\dots,\Bcal_T \subseteq \mathbb{R}^K$. The $t$th iteration of the algorithm involves solving the single-period problem 
\begin{equation}
\max_{\bb'_t \in \Bcal_t} \hat{J}_R(\bb_1,\dots, \bb_{t-1}, \bb'_{t}, \bb_{t+1}, \dots, \bb_T) \label{prob:backward_period_t_abstract}
\end{equation}
and updating the $t$th weight vector in $\bb$, which is $\bb_t$, with the new solution $\bb^*_t$. This process goes on from period $t = T$ all the way to $t = 1$; after the $t = 1$ iteration, the algorithm terminates. We formally define our procedure as Algorithm~\ref{algorithm:backward} below.

\begin{algorithm}
\begin{algorithmic}
\STATE Initialize $\bb_t \gets \zerob$ for all $t \in [T]$. 
\STATE Initialize $c_T(\omega) = 0$ for all $\omega \in [\Omega]$.
\FOR{ $t = T, \dots, 1$}
	\STATE Compute $p_t(\omega)$ as 
		\begin{equation}
		p_t(\omega) = \prod_{t'=1}^{t-1} (1 - \sigma(\bb_{t'} \bullet \Phi(\xb(\omega,t')))).
		\end{equation}
	\STATE Solve the problem
		\begin{equation}
		\max_{\bb_t \in \Bcal_t} \sum_{\omega=1}^{\Omega} \frac{1}{\Omega} \cdot p_t(\omega) \cdot [ g(t, \xb(\omega,t)) \cdot \sigma(\bb_t \bullet \Phi(\xb(\omega,t))) + c_t(\omega) \cdot (1 - \sigma(\bb_t \bullet \Phi(\xb(\omega,t)))) ] \label{prob:backward_period_t_algo}
		\end{equation}
		to obtain an optimal solution $\bb^*_t$. 
	\STATE Compute $c_{t-1}(\omega)$ as 
	\begin{equation}
		c_{t-1}(\omega) = g(t, \xb(\omega,t)) \cdot \sigma(\bb^*_t \bullet \Phi(\xb(\omega,t)))  + c_t(\omega) \cdot (1 - \sigma(\bb^*_t \bullet \Phi(\xb(\omega,t)))).
	\end{equation}
\ENDFOR
\end{algorithmic}
\caption{Backwards optimization algorithm for approximately solving the randomized policy SAA problem~\eqref{prob:SAA_randomized}. \label{algorithm:backward}}
\end{algorithm}

We pause to make several comments about Algorithm~\ref{algorithm:backward}. First, observe that the period $t$ problem solved in Algorithm~\ref{algorithm:backward}, problem~\eqref{prob:backward_period_t_algo}, is of a different form from problem~\eqref{prob:backward_period_t_abstract}. The two problems are equivalent in that problem~\eqref{prob:backward_period_t_algo} is a simplification of problem~\eqref{prob:backward_period_t_abstract}. In particular, $p_t(\omega)$ can be regarded as the probability, conditional on the weight vectors $\bb_1,\dots, \bb_{t-1}$, of not having stopped by period $t$ in trajectory $\omega$. By using this term, we can simplify the problem and remove the appearance of the weight vectors for periods prior to $t$. Similarly, $c_t(\omega)$ can be regarded as the expected continuation value at period $t$ in trajectory $\omega$, i.e., given that we have not stopped by period $t$, what is the expected reward (where the expectation is with respect to the randomness of the stopping decisions) from not stopping at period $t$, for the trajectory $\omega$. Using both of these, and using the fact that $\hat{J}_R(\cdot)$ includes terms that only depend on $\bb_{t'}$ for $t' < t$, we can boil problem~\eqref{prob:backward_period_t_abstract} down to problem~\eqref{prob:backward_period_t_algo}, which is of the form $\sum_{\omega} (c_{\omega} + d_{\omega} \cdot \sigma( \bb_t \bullet \Phi(\xb(\omega,t))))$.

Second, we note that problem~\eqref{prob:backward_period_t_algo} is still a challenging problem to solve, as the objective function is still non-convex. It is an instance of the sum-of-sigmoids problem (a sigmoid function being an S-shaped function, such as the logistic response function $\sigma(\cdot)$), which \cite{udell2013maximizing} show to be NP-Hard in general. Similarly, \cite{akcakus2021exact} show that a related problem, of finding a binary product attribute vector that maximizes the expected market share under a mixture-of-logits model, is NP-Hard. However, problem~\eqref{prob:backward_period_t_algo} is more manageable to solve than the complete randomized policy SAA problem~\eqref{prob:SAA_randomized}, as it involves only the weight variables for a single period ($K$ variables) as opposed to all $T$ periods ($KT$ variables). In our implementation of Algorithm~\ref{algorithm:backward}, we use the Adam algorithm \citep{kingma2014adam} to approximately solve problem~\eqref{prob:backward_period_t_algo}. %

Lastly, we comment on how we use the solution $\bb^* = (\bb^*_1, \dots, \bb^*_T)$ produced by Algorithm~\ref{algorithm:backward}. Although $\bb^*$ corresponds to a randomized policy, in our numerical experiments we will focus on using $\bb^*$ within a deterministic linear policy. In other words, we plug $\bb^*$ into a policy of the form of equation~\eqref{eq:deterministic_policy_definition}. The reason for doing this is that in general, we have empirically observed that the deterministic policy defined with $\bb^*$ performs better than the randomized policy defined with $\bb^*$. To understand the intuition for this, let us consider problem~\eqref{prob:backward_period_t_algo}. For this problem, a good weight vector $\bb_t$ at time $t$ would be one where, for most trajectories, $\bb_t \bullet \Phi(\xb(\omega,t))$ is very positive when $g(t, \xb(\omega,t))$ is higher than $c_t(\omega)$, and where $\bb_t \bullet \Phi(\xb(\omega,t))$ is very negative when $c_t(\omega)$ is higher than $g(t, \xb(\omega,t))$. When this is true for most trajectories, it is reasonable to expect that we could improve our objective value by thresholding $\bb_t \bullet \Phi(\xb(\omega,t))$, i.e., rounding $\sigma( \bb_t \bullet \Phi(\xb(\omega,t)))$ to 0 or 1, which would have the effect of making the expression in the square brackets in problem~\eqref{prob:backward_period_t_algo} generally (i.e., for most trajectories) equal to $\max \{ g(t,\xb(\omega,t)), c_t(\omega)\}$, which is a higher quantity than $g(t, \xb(\omega,t)) \cdot \sigma(\bb_t \bullet \Phi(\xb(\omega,t))) + c_t(\omega) \cdot (1 - \sigma(\bb_t \bullet \Phi(\xb(\omega,t))))$. 

Besides this consideration, as discussed in Section~\ref{subsec:problem_definition_equivalence}, optimizing over randomized policies is equivalent to optimizing over deterministic policies, and our motivation for optimizing over randomized policies is to ultimately obtain good deterministic policies in a tractable manner. Lastly, we note that using $\bb^*$ within a deterministic policy is similar to how in binary classification problems in machine learning, it is common to learn a probabilistic model whose natural output is a probability of a target class (for example, a logistic regression model), and to then threshold this probability to obtain a hard classification.

\subsection{Comparison of Algorithm~\ref{algorithm:backward} with least-squares Monte Carlo}
\label{subsec:solution_methodology_comparison_LSM}

Algorithm~\ref{algorithm:backward} shares some similarities with the least-squares Monte Carlo (LSM) algorithm of \cite{longstaff2001valuing}. For easier comparison, we state the basic LSM algorithm adapted to our problem setting as Algorithm~\ref{algorithm:LSM} below. 

\begin{algorithm}
\begin{algorithmic}
\STATE Initialize $c_{T-1}(\omega) = g(T,\xb(\omega,T))$ for all $\omega \in [\Omega]$.
\FOR{ $t = T-1, \dots, 1$}
	\STATE Solve the least-squares problem 
	\begin{equation}
	\min_{\bb_t \in \mathbb{R}^K} \frac{1}{2} \sum_{\omega = 1}^{\Omega} ( c_t(\omega) - \bb_t \bullet \Phi(\xb(\omega,t)))^2 
	\end{equation}
	to obtain an optimal solution $\bb^*_t$. 
	\STATE Compute $c_{t-1}(\omega)$ as 
	\begin{equation}
		c_{t-1}(\omega) = \left\{ \begin{array}{ll} 
				c_t(\omega) & \text{if}\ \bb^*_t \bullet \Phi(\xb(\omega,t)) \geq g(t,\xb(\omega,t)), \\
				g(t,\xb(\omega,t)) & \text{if}\ \bb^*_t \bullet \Phi(\xb(\omega,t)) < g(t,\xb(\omega,t)).
				\end{array} \right.
	\end{equation}
\ENDFOR
\end{algorithmic}
\caption{Least-squares Monte Carlo (LSM) algorithm of \cite{longstaff2001valuing}. \label{algorithm:LSM}}
\end{algorithm}

In particular, the LSM algorithm also involves iterating backwards in time, and also involves updating the continuation value using the current policy. However, a key difference is that LSM involves solving a least-squares problem to obtain basis function weights $\bb_t$, so as to predict the continuation value using those basis function weights. The stopping policy is then defined by comparing the current payoff to the predicted continuation value, where stopping is prescribed if and only if the current payoff is more than the predicted continuation value. In contrast, our algorithm involves directly optimizing over the stopping policy at a given period: in problem~\eqref{prob:backward_period_t_algo}, we look for weights $\bb_t$ for the stopping decision in the current period so that the expected reward, which accounts for both the current period's reward and the continuation value $c_t(\omega)$ that captures reward in future periods, is optimized. 

Besides this difference, it is also important to appreciate the higher level differences in the two approaches. In particular, LSM (Algorithm~\ref{algorithm:LSM}) produces a policy of the form 
\begin{equation*}
\pi(t, \xb) = \left\{ \begin{array}{ll} \stop & \text{if}\ g(t,\xb) > \bb_t \bullet \Phi(\xb(t)), \\ \go & \text{if}\ g(t,\xb) \leq \bb_t \bullet \Phi(\xb(t)). \end{array} \right.
\end{equation*}
Note that this policy can be made equivalent to a deterministic linear policy as we have defined it in Sections~\ref{subsec:problem_definition_deterministic} and \ref{subsec:problem_definition_deterministic_SAA}. Specifically, we can augment the state variable $\xb(t)$ to include an additional coordinate that is equal to $g(t,\xb(t))$ and then augment the basis function architecture $\Phi(\xb) = (\phi_1(\xb), \dots, \phi_K(\xb)))$ with a $K+1$th basis function $\phi_{K+1}(\cdot)$ that is exactly equal to this new coordinate. With these augmentations, the weight vector  $\tilde{\bb}_t = (-b_{t,1}, \dots, -b_{t,K}, +1)$ is such that 
\begin{equation*}
g(t,\xb) > \sum_{k=1}^K b_{t,k} \phi_k(\xb(t)) \ \text{if and only if}\ \sum_{k=1}^{K+1} \tilde{b}_{t,k} \phi_k(\xb(t)) > 0,
\end{equation*}
i.e., the corresponding deterministic linear policy with the $K+1$ basis functions and the weight vectors $\tilde{\bb}_1,\dots, \tilde{\bb}_{T}$ behaves identically to the LSM policy. Thus, LSM can be viewed as a method for returning a solution to the deterministic policy SAA problem~\eqref{prob:SAA_deterministic}. 

In light of this relationship, we note that, to our knowledge, there is no guarantee that the solution that LSM returns solves either the true deterministic policy problem~\eqref{prob:full_deterministic} or the deterministic policy SAA problem~\eqref{prob:SAA_deterministic}. In contrast, Algorithm~\ref{algorithm:backward} is designed to directly (albeit approximately) solve the randomized policy SAA problem~\eqref{prob:SAA_randomized}. Our results in Sections~\ref{sec:problem_definition} and \ref{sec:statistical_properties} provide theoretical justification for why this approach is desirable: under mild conditions, the true randomized policy problem~\eqref{prob:full_randomized} and its SAA counterpart \eqref{prob:SAA_randomized} are equivalent to the true deterministic policy problem~\eqref{prob:full_deterministic} and its SAA counterpart, respectively (guaranteed by our equivalence results, Theorems~\ref{theorem:deterministic_randomized_SAA_equal} and \ref{theorem:deterministic_randomized_full_equal}); as we accumulate more data, the optimal objective value and solution of the randomized policy SAA problem~\eqref{prob:SAA_randomized} converge to that of the true randomized policy problem \eqref{prob:full_randomized} (guaranteed by our consistency results, Corollary~\ref{corollary:SAA_objective_convergence} and Theorem~\ref{theorem:SAA_solution_convergence}); and optimizing the randomized policy SAA problem directly optimizes a lower bound on the out-of-sample reward that becomes tighter as one accumulates more data (guaranteed by our generalization bound and Rademacher complexity results, Proposition~\ref{proposition:generalization_bound} and Theorem~\ref{theorem:main_rademacher_bound}). Taken together, these results suggest that for a fixed basis function architecture, our method (Algorithm~\ref{algorithm:backward}) has the potential to obtain policies that deliver better out-of-sample performance than LSM. In Section~\ref{sec:numerics}, we will showcase one family of benchmark problem instances where this is indeed the case.

\section{Application to option pricing}
\label{sec:numerics}
In this section, we apply our randomized policy approach to a standard option pricing problem, previously considered in a number of papers \citep[e.g.,][]{desai2012pathwise,ciocan2020interpretable}. We define our option pricing problem in Section \ref{subsec:numerics_background}. In Section~\ref{subsec:numerics_neq1}, we illustrate the difference between our randomized policy approach and prior approaches for obtaining deterministic linear policies using a simple option pricing problem involving a single asset. Then, in Section~\ref{subsec:numerics_neq8}, we test our approach and compare it to prior approaches in a higher dimensional setting with eight assets.

We implement our methods in the Julia programming language, version 0.6.4 \citep{bezanson2017julia}. For the pathwise optimization method, we implement the pathwise linear program using the JuMP package \citep{lubin2015computing,dunning2017jump} and solve it using Gurobi, version 9.5 \citep{gurobi}. All our experiments are executed on Amazon Elastic Compute Cloud (EC2), on a single instance of type \texttt{r4.8xlarge} (Intel Xeon E5-2686 v4 processor with 32 virtual CPUs and 244 GBs of memory). 

\subsection{Background}
\label{subsec:numerics_background}
The optimal stopping problem that we will focus on is pricing a Bermudan max-call option with a knock-out barrier, which was previously studied in \cite{desai2012pathwise} and later in \cite{ciocan2020interpretable}. We consider the same family of problem instances used in those papers, and briefly review the details here. 

In this family of problem instance, the option is dependent on $n$ assets. The option is exercisable over a period of 3 calendar years with $T=54$ equally spaced exercise times. The price of each underlying asset follows a geometric Brownian motion, with the drift set equal to the annualized risk-free rate $r$ and the annualized volatility set to $\sigma$, and each asset is assumed to start at an initial price of $\bar{p}$. In all of the experiments that we will present, we shall assume $r = 5\%$ and $\sigma = 20\%$, as in \cite{desai2012pathwise}, and we will also assume the pairwise correlation between the assets to be zero. We use $p_i(t)$ denote the price of asset $i$ at exercise time $t$. %

The option has a strike price $K$ and a knock-out barrier price $B$. The payoff of the option at any given time is determined by the strike price $K$, the knock-out barrier value $B$ and the maximum price among the $n$ underlying assets. If at time $t$ the maximum price of the $n$ underlying assets exceeds the barrier price $B$, the option is ``knocked out'' and the payoff becomes zero for all times $\tilde{t}\geq t$. We let $y(t)$ be an indicator variable that is 1 if the option has not been knocked out by time $t$ and zero otherwise:
\begin{equation}
y(t)=\Ibb\left\{\max_{1\leq i\leq n,1\leq t^'\leq t}p_i(t^')< B\right\}
\end{equation}
We let $g'(t)$ denote the (undiscounted) payoff from exercising the option at time $t$, which is defined as follows:
\begin{equation}
g'(t)=y(t)\cdot\max\left\{0, \max_{1\leq i\leq n}p_i(t) - K\right\}. 
\end{equation}
All payoffs are assumed to be discounted continuously according to the risk-free rate. This implies a discrete discount factor $\beta = \exp( -r \times 3/54) = 0.99723$. We can thus define the discounted reward $g(t)$ to be $g(t) = \beta^t \cdot g'(t)$, which can be thought of as the payoff denominated in dollars corresponding to time $t = 0$.

We compare three different methods: our randomized policy optimization (RPO) approach, the least-squares Monte Carlo (LSM) method of \cite{longstaff2001valuing} and the pathwise optimization (PO) method of \cite{desai2012pathwise}. We test of each of these methods with a variety of basis functions. In our presentation of our results, we will denote the different sets of basis functions as follows:
\begin{itemize}
	\item \textsc{one}: the constant basis function, equal to 1 for every state.
	\item \textsc{prices}: the price $p_i(t)$ of asset $i$ for $i \in [n]$.
	\item \textsc{payoff}: the undiscounted payoff $g'(t)$. 
	\item \textsc{KOind}: the knock-out (KO) indicator variable $y(t)$. 
	\item \textsc{pricesKO}: the KO adjusted prices $p_i(t) \cdot y(t)$ for $i \in [n]$. 
	\item \textsc{maxpriceKO} and \textsc{max2priceKO}: the largest and second largest KO adjusted prices. 
	\item \textsc{prices2KO}: the KO adjusted second-order price terms, $p_i(t) \cdot p_j(t) \cdot y(t)$ for $1 \leq i \leq j \leq n$. 
\end{itemize}

In our implementation of the RPO approach, we use the backward algorithm, Algorithm~\ref{algorithm:backward}. We use the coefficients obtained directly within a deterministic policy. We solve problem~\eqref{prob:backward_period_t_algo} using a custom implementation of Adam, a momentum-based first-order method \citep{kingma2014adam,goodfellow2016deep}. We follow the parameter defaults in \cite{kingma2014adam}, with the exception of the step size, for which we use $10^{-1}$, as opposed to $10^{-3}$. Additionally, we do not apply any minibatching, and compute the full gradient for the entire sample of $\Omega$ trajectories. For each solve of problem~\eqref{prob:backward_period_t_algo}, we warm start the Adam algorithm using the coefficients obtained by LSM; we describe our warm starting scheme in more detail in Section~\ref{subsec:additional_numerics_warmstart} of the ecompanion. %

In our implementation of the pathwise optimization method, we follow \cite{desai2012pathwise} in generating 500 inner samples. 

\subsection{Experiment \#1: An illustrative example with $n = 1$} 
\label{subsec:numerics_neq1}

In our first experiment, to demonstrate the difference between our approach and incumbent approaches, we consider an instance of the option with $n = 1$ asset; thus, the undiscounted payoff and knock-out indicators can be written simply as
\begin{equation}
g'(t)=y(t)\cdot\max\left\{0, p_1(t) - K\right\},
\end{equation}
\begin{equation}
y(t) = \Ibb\left\{\max_{1\leq t' \leq t}p_1(t^')< B\right\}.
\end{equation}
We set $K = 100$ and $B = 150$, and vary $\bar{p}$ in the set $\{90, 100, 110\}$. For each initial price $\bar{p}$, we perform 10 replications, where in each replication we generate a set of $\Omega = 100,000$ trajectories to train each policy, and 100,000 trajectories for out-of-sample testing. 

We test LSM with two basis function architectures: (i) \textsc{one}, and (ii) \textsc{one} and \textsc{payoff}. Note that both of these basis function architectures imply an exercise policy that involves simply comparing the undiscounted payoff $g'(t)$ to a constant, state-independent threshold. In particular, for (i), the exercise policy prescribes $\stop$ if and only if
\begin{align*}
g(t) & \geq b_{\textsc{one}} \cdot 1 \\
& = b_{\textsc{one}},
\end{align*}
which is equivalent to 
\begin{equation*}
g'(t) \geq  \beta^{-t} b_{\textsc{one}}. 
\end{equation*}
For (ii), the exercise policy prescribes $\stop$ if and only if 
\begin{align*}
g(t) & \geq b_{\textsc{one}} \cdot 1 + b_{\textsc{payoff}} \cdot g'(t).
\end{align*}
Using the fact that $g(t) = \beta^{t} g'(t)$, we can re-arrange the above inequality into the following threshold rule in terms of the undiscounted payoff:
\begin{align*}
g'(t) \geq \frac{b_{\textsc{one}}}{\beta^{t} - b_{\textsc{payoff}}},
\end{align*}
which holds if $\beta^{t} - b_{\textsc{payoff}} > 0$.

For the pathwise optimization method, we test it with the same two basis function architectures as LSM. Since the pathwise optimization-based policy is also a greedy policy based on an approximate continuation value function, one can again represent the policies obtained with the architectures (i) and (ii) as constant threshold policies. In addition to the policies, we also use the pathwise optimization solution to compute an upper bound on the optimal reward using an independent set of 100,000 trajectories (see \citealt{desai2012pathwise}).

For the randomized policy approach, we test it with a single basis function architecture, consisting of \textsc{one} and \textsc{payoff}. This results in an exercise policy where $\stop$ is recommended if and only if 
\begin{equation*}
b_{\textsc{one}} \times 1 + b_{\textsc{payoff}} \times g'(t) > 0,
\end{equation*}
which is equivalent to the threshold rule
\begin{equation*}
g'(t) > - \frac{b_{\textsc{one}}}{b_{\textsc{payoff}}}
\end{equation*}
if $b_{\textsc{payoff}} > 0$.

Table~\ref{table:DFM_neq1_performance} shows the out-of-sample performance of the different methods under the different basis function architectures, as well as the pathwise optimization upper bounds. For each combination of a policy (a combination of one of the three methods -- LSM, PO and RPO -- and a basis function architecture) and an initial price $\bar{p}$, we report the average out-of-sample reward over the ten replications. We additionally report the standard error over those ten replications in parentheses. 

From this table, we can see that even though the three methods -- LSM, pathwise optimization and the randomized policy approach -- produce policies within the same policy class, there are significant differences in performance. In particular, the policy produced by the randomized policy approach significantly outperforms LSM and pathwise optimization. Comparing to LSM with \textsc{one}, the randomized policy approach with \textsc{one} and \textsc{payoff} attains an expected discounted reward that is as much as 89\% higher. Comparing to LSM with \textsc{one} and \textsc{payoff}, which in general performs better than LSM with \textsc{one}, the improvement by the randomized policy approach is as much as 7.7\%. Comparing to PO with \textsc{one} and with \textsc{one} and \textsc{payoff}, the randomized policy approach attains an improvement of up to 29\% and 35\%, respectively. In addition, the PO upper bounds are close to the performance of the randomized policy approach (for all three initial prices, the RPO lower bound is within 2.3\% of the tightest PO upper bound). This suggests that for this problem setting, the policy is nearly optimal. This experiment highlights the fact that even for a simple problem instance involving only a single asset and the simplest possible policy class, the LSM method can return a policy that is substantially suboptimal.

\begin{table}
\centering
\begin{tabular}{llccc} \toprule
& & \multicolumn{3}{c}{Initial price} \\
Method & Basis functions & $\bar{p} = 90$ & $\bar{p} = 100$ & $\bar{p} = 110$ \\ \midrule
  LSM & \textsc{one} & 6.47\enskip (0.010) & 10.82\enskip (0.011) & 16.47\enskip (0.008) \\ 
  LSM & \textsc{one, payoff} & 11.37\enskip (0.020) & 16.64\enskip (0.024) & 22.01\enskip (0.018) \\[0.25em]
  PO & \textsc{one} & 9.47\enskip (0.017) & 14.79\enskip (0.017) & 20.67\enskip (0.014) \\
  PO & \textsc{one, payoff} & 9.07\enskip (0.032) & 16.01\enskip (0.029) & 22.73\enskip (0.023) \\[0.25em] 
  RPO & \textsc{one, payoff} & \bfseries 12.25\enskip (0.018) & \bfseries 17.51\enskip (0.023) & \bfseries 23.04\enskip (0.018) \\[0.25em] 
  PO-UB & \textsc{one} & 18.26\enskip (0.018) & 25.47\enskip (0.012) & 32.49\enskip (0.012) \\ 
  PO-UB & \textsc{one, payoff} & 12.54\enskip (0.009) & 17.88\enskip (0.009) & 23.55\enskip (0.005) \\ \bottomrule
\end{tabular}
\caption{Out-of-sample performance of different policies in $n = 1$ experiment. \label{table:DFM_neq1_performance}}
\end{table}

It is also interesting to consider what the thresholds produced by the different methods look like. Figure~\ref{figure:DFM_neq1_threshold_plot} plots the thresholds for the five different policies at each period in the time horizon, for a single replication with $\bar{p} = 110$. We can see that there are substantial differences in the policies. The thresholds for the LSM policies are generally lower than those of the RPO policy, which implies that the LSM policies in general stop earlier in the time horizon, when the reward will generally be lower. The PO policy with \textsc{one} also results in thresholds that are lower than the RPO policy. On the other hand, the PO policy with \textsc{one} and \textsc{payoff} results in thresholds that are higher than those from RPO for roughly the first 40 periods; as a result, the PO policy may miss opportunities to stop earlier in the horizon. Interestingly, the thresholds for the LSM and PO policies begin rapidly decaying earlier in the time horizon than RPO (for LSM with \textsc{one}, LSM with \textsc{one} and \textsc{payoff}, and PO with \textsc{one}, this starts right around the beginning of the horizon; for PO with \textsc{one} and \textsc{payoff}, this starts at around $t = 34$). For RPO, there is a slow and steady decrease in the threshold until about $t = 48$, where the threshold begins to decrease much more quickly.

\begin{figure}
\begin{center}
\includegraphics[width=\textwidth]{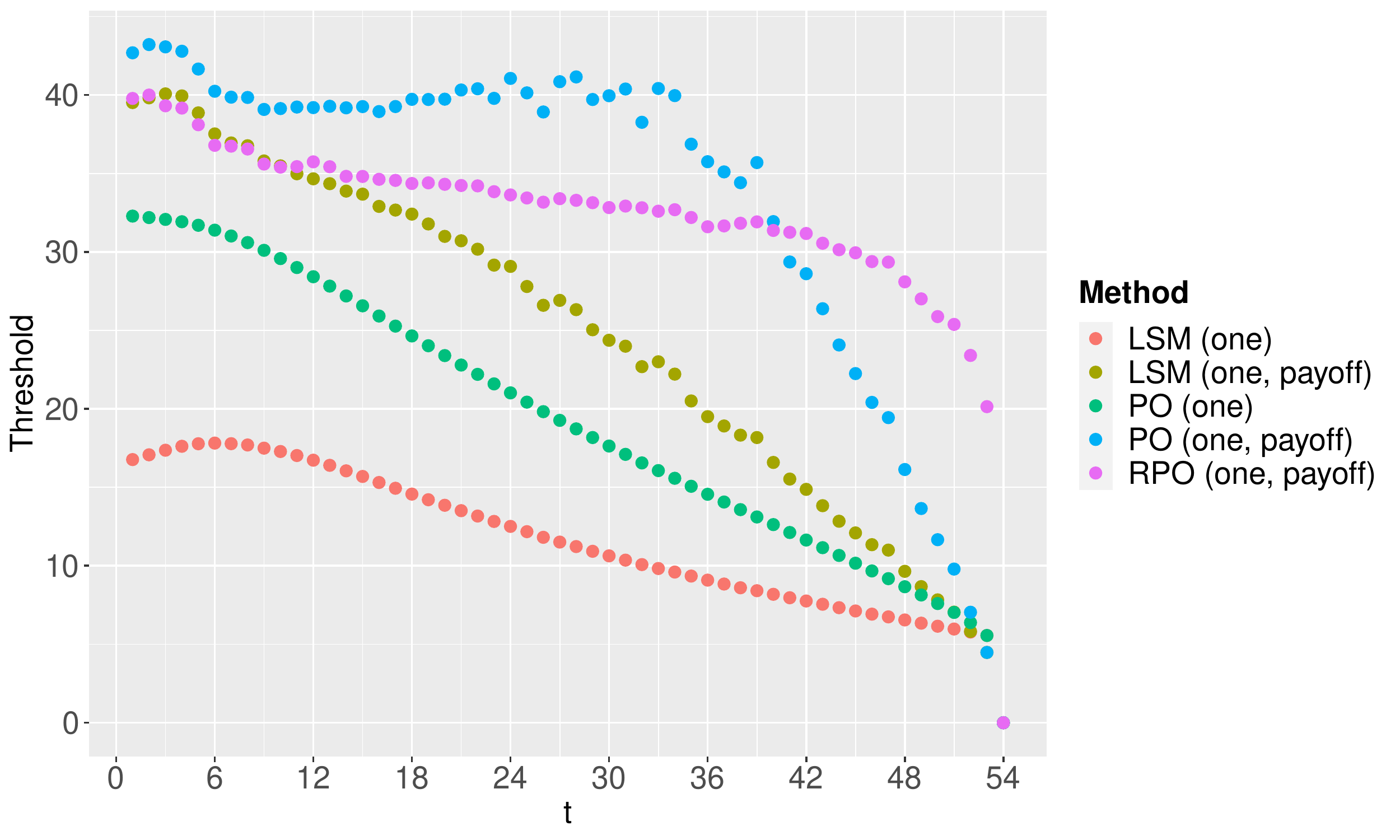}
\end{center}
\caption{Plot of thresholds for policies in $n = 1$ experiment. \label{figure:DFM_neq1_threshold_plot}}
\end{figure}

\subsection{Experiment \#2: multiple assets}
\label{subsec:numerics_neq8}

In our second experiment, we consider instances of our option pricing problem with more than one asset. We specifically consider instances with $n$ varying in $\{4, 8, 16\}$. As in the previous experiment, we vary $\bar{p}$ in $\{90, 100, 110\}$ and set the strike price $K = 100$. Following \cite{desai2012pathwise}, we set the barrier price $B = 170$. For each initial price $\bar{p}$ and each value of $n$, we perform ten replications, where in each replication we generate a training set of $\Omega = 20,000$ trajectories, and a testing set of 100,000 trajectories. In what follows, we focus on the results for $n = 8$, and relegate the performance results for $n = 4$ and $n = 16$ to Section~\ref{subsec:additional_numerics_neq4_16} of the ecompanion.

We again test the LSM, PO and RPO methods with a variety of basis function architectures. We also obtain upper bounds from the PO method by reporting the objective value of the pathwise optimization linear program, which is a biased upper bound on the expected reward. We opt for this simpler approach over producing an unbiased upper bound (by generating an independent set of trajectories and the corresponding inner paths; see \citealt{desai2012pathwise}) due to the significant computation time required in generating the inner paths. We note that this inexact approach has also been used in other work that has implemented the PO method \citep{ciocan2020interpretable}. 

Table~\ref{table:DFM_performance} reports the out-of-sample performance of the LSM, PO and RPO methods, as well as the (biased) PO upper bound, for the different basis function architectures. Note that the table is organized so that groups of policies corresponding to the same policy class are grouped together. (For example, LSM/PO with \textsc{one} and \textsc{prices}, LSM/PO with \textsc{one}, \textsc{prices} and \textsc{payoff}, and RPO with \textsc{one}, \textsc{prices} and \textsc{payoff} appear together.) 

\begin{table}
\centering
\footnotesize
\begin{tabular}{lllll} \toprule
& & \multicolumn{3}{c}{Initial price} \\
Method & Basis function architecture & $\bar{p} = 90$ & $\bar{p} = 100$ & $\bar{p} = 110$ \\ \midrule
LSM & \textsc{one} & 33.77\enskip (0.023) & 38.67\enskip (0.010) & 43.13\enskip (0.013) \\ 
  LSM & \textsc{one, payoff} & 41.18\enskip (0.033) & 43.21\enskip (0.037) & 45.00\enskip (0.027) \\ 
  PO & \textsc{one} & 41.08\enskip (0.015) & 45.91\enskip (0.021) & 48.84\enskip (0.016) \\ 
  PO & \textsc{one, payoff} & 22.25\enskip (0.177) & 16.07\enskip (0.144) & 11.57\enskip (0.119) \\ 
  RPO & \textsc{one, payoff} & \bfseries 45.30\enskip (0.022) & \bfseries 51.10\enskip (0.012) & \bfseries 53.46\enskip (0.053) \\[0.5em] 
  PO-UB & \textsc{one} & 52.19\enskip (0.021) & 57.45\enskip (0.020) & 60.35\enskip (0.010) \\ 
  PO-UB & \textsc{one, payoff} & 46.37\enskip (0.024) & 52.68\enskip (0.051) & 56.02\enskip (0.047) \\ \midrule
  
  LSM & \textsc{prices} & 33.81\enskip (0.024) & 38.54\enskip (0.013) & 43.02\enskip (0.013) \\ 
  LSM & \textsc{prices, payoff} & 39.56\enskip (0.030) & 41.74\enskip (0.033) & 44.12\enskip (0.025) \\ 
  PO & \textsc{prices} & 40.93\enskip (0.016) & 44.83\enskip (0.014) & 47.49\enskip (0.016) \\ 
  PO & \textsc{prices, payoff} & 22.28\enskip (0.124) & 15.89\enskip (0.116) & 11.04\enskip (0.091) \\ 
  RPO & \textsc{prices, payoff} & \bfseries 44.49\enskip (0.018) & \bfseries 49.77\enskip (0.029) & \bfseries 52.23\enskip (0.035) \\[0.5em]  
  PO-UB & \textsc{prices} & 51.40\enskip (0.023) & 57.20\enskip (0.011) & 60.32\enskip (0.010) \\ 
  PO-UB & \textsc{prices, payoff} & 46.36\enskip (0.024) & 52.64\enskip (0.050) & 55.94\enskip (0.045) \\ \midrule
  
      LSM & \textsc{pricesKO} & 41.42\enskip (0.017) & 49.35\enskip (0.017) & 53.10\enskip (0.009) \\ 
  LSM & \textsc{pricesKO, payoff} & 44.04\enskip (0.017) & 49.62\enskip (0.012) & 52.67\enskip (0.006) \\ 
  PO & \textsc{pricesKO} & 44.32\enskip (0.017) & 49.82\enskip (0.015) & 52.77\enskip (0.018) \\ 
  PO & \textsc{pricesKO, payoff} & 44.18\enskip (0.017) & 50.06\enskip (0.015) & 53.19\enskip (0.007) \\ 
  RPO & \textsc{pricesKO, payoff} & \bfseries 44.53\enskip (0.019) & \bfseries 50.11\enskip (0.013) & \bfseries 53.27\enskip (0.010) \\[0.5em]  
  PO-UB & \textsc{pricesKO} & 48.63\enskip (0.015) & 53.12\enskip (0.010) & 55.57\enskip (0.011) \\ 
  PO-UB & \textsc{pricesKO, payoff} & 46.15\enskip (0.023) & 52.06\enskip (0.034) & 55.08\enskip (0.024) \\ \midrule
  
LSM & \textsc{KOind} & 39.37\enskip (0.020) & 48.09\enskip (0.030) & 53.26\enskip (0.017) \\ 
  LSM & \textsc{KOind, payoff} & 44.26\enskip (0.018) & 50.07\enskip (0.016) & 53.19\enskip (0.010) \\ 
  PO & \textsc{KOind} & 43.87\enskip (0.017) & 50.85\enskip (0.013) & 54.35\enskip (0.009) \\ 
  PO & \textsc{KOind, payoff} & 44.79\enskip (0.025) & 50.89\enskip (0.013) & 53.91\enskip (0.008) \\ 
  RPO & \textsc{KOind, payoff} & \bfseries 45.45\enskip (0.023) & \bfseries 51.37\enskip (0.011) & \bfseries 54.50\enskip (0.010) \\[0.5em]  
  PO-UB & \textsc{KOind} & 49.29\enskip (0.016) & 53.47\enskip (0.015) & 55.69\enskip (0.009) \\ 
  PO-UB & \textsc{KOind, payoff} & 46.15\enskip (0.023) & 52.07\enskip (0.033) & 55.05\enskip (0.021) \\ \midrule
  
    LSM & \textsc{pricesKO, KOind} & 41.84\enskip (0.015) & 49.37\enskip (0.021) & 53.46\enskip (0.009) \\ 
  LSM & \textsc{pricesKO, KOind, payoff} & 43.77\enskip (0.019) & 49.87\enskip (0.018) & 53.11\enskip (0.007) \\ 
  PO & \textsc{pricesKO, KOind} & 44.01\enskip (0.018) & \bfseries 50.91\enskip (0.013) & \bfseries 54.27\enskip (0.008) \\ 
  PO & \textsc{pricesKO, KOind, payoff} & 43.98\enskip (0.021) & 50.69\enskip (0.012) & 53.84\enskip (0.007) \\ 
  RPO & \textsc{pricesKO, KOind, payoff} & \bfseries 44.08\enskip (0.023) & 50.57\enskip (0.031) & 54.23\enskip (0.010) \\[0.5em]  
  PO-UB & \textsc{pricesKO, KOind} & 48.45\enskip (0.020) & 53.09\enskip (0.011) & 55.56\enskip (0.010) \\ 
  PO-UB & \textsc{pricesKO, KOind, payoff} & 46.14\enskip (0.022) & 52.05\enskip (0.033) & 55.04\enskip (0.022) \\ \midrule

  LSM & \textsc{pricesKO, prices2KO, KOind} & 43.32\enskip (0.022) & 49.86\enskip (0.019) & 53.26\enskip (0.013) \\ 
  LSM & \textsc{pricesKO, prices2KO, KOind, payoff} & 44.05\enskip (0.022) & 49.92\enskip (0.019) & 53.14\enskip (0.012) \\ 
  PO & \textsc{pricesKO, prices2KO, KOind} & 44.33\enskip (0.018) & \bfseries 50.78\enskip (0.014) & 53.93\enskip (0.006) \\ 
  PO & \textsc{pricesKO, prices2KO, KOind, payoff} & \bfseries 44.65\enskip (0.018) & 50.65\enskip (0.016) & 53.77\enskip (0.008) \\ 
  RPO & \textsc{pricesKO, prices2KO, KOind, payoff} & 44.62\enskip (0.015) & 50.74\enskip (0.021) & \bfseries 54.03\enskip (0.013) \\[0.5em]  
  PO-UB & \textsc{pricesKO, prices2KO, KOind} & 47.09\enskip (0.016) & 52.43\enskip (0.019) & 55.25\enskip (0.010) \\ 
  PO-UB & \textsc{pricesKO, prices2KO, KOind, payoff} & 46.09\enskip (0.022) & 51.98\enskip (0.033) & 55.00\enskip (0.022) \\ \midrule

  LSM & \textsc{pricesKO, KOind, maxpriceKO, max2priceKO} & 43.83\enskip (0.018) & 49.89\enskip (0.023) & 53.10\enskip (0.008) \\ 
  LSM & \textsc{pricesKO, KOind, maxpriceKO, max2priceKO, payoff} & 43.83\enskip (0.017) & 49.88\enskip (0.022) & 53.10\enskip (0.008) \\ 
  PO & \textsc{pricesKO, KOind, maxpriceKO, max2priceKO} & 43.90\enskip (0.026) & \bfseries 50.66\enskip (0.014) & 53.83\enskip (0.008) \\ 
  PO & \textsc{pricesKO, KOind, maxpriceKO, max2priceKO, payoff} & 44.04\enskip (0.023) & 50.65\enskip (0.015) & 53.82\enskip (0.007) \\ 
  RPO & \textsc{pricesKO, KOind, maxpriceKO, max2priceKO, payoff} & \bfseries 44.14\enskip (0.016) & 50.55\enskip (0.030) & \bfseries54.20\enskip (0.010) \\[0.5em]
  PO-UB & \textsc{pricesKO, KOind, maxpriceKO, max2priceKO} & 46.13\enskip (0.017) & 52.04\enskip (0.033) & 55.04\enskip (0.022) \\ 
  PO-UB & \textsc{pricesKO, KOind, maxpriceKO, max2priceKO, payoff} & 46.12\enskip (0.023) & 52.04\enskip (0.033) & 55.03\enskip (0.021) \\ %

  \bottomrule
\end{tabular}
\caption{Out-of-sample performance for different policies, for $n = 8$ assets. \label{table:DFM_performance}}
\end{table}

From this table, we can see that within each policy class, the RPO method in general outperforms the LSM method. In some cases the difference can be substantial: for example, with $\bar{p} = 90$ and the policy class corresponding to linear functions of \textsc{KOind} and \textsc{payoff}, the best LSM policy achieves a reward of 44.26 whereas RPO achieves a reward of 45.45, which is an improvement of 2.6\%. Relative to the PO method, the performance of the RPO method in most cases is better, and in a few cases is slightly worse (for example, for $\bar{p} = 110$ and the \textsc{pricesKO}, \textsc{KOind} and \textsc{payoff} policy class, the best PO policy attains a reward of 54.27 compared to 54.23 for the RPO policy). 

In addition to the comparison of the methods within a fixed policy class, it is also insightful to compare the methods across policy classes, i.e., to think of what is the best attainable performance across any basis function architecture. In this regard, the highest rewards for all three initial prices are attained by the RPO method with \textsc{KOind} and \textsc{payoff} as the basis functions (45.45 for $\bar{p} = 90$, 51.37 for $\bar{p} = 100$, 54.50 for $\bar{p} = 110$). The best performance for the LSM method across any of the basis function architectures is substantially lower (44.26 for $\bar{p} = 90$, 50.07 for $\bar{p} = 100$, 53.46 for $\bar{p} = 110$). The best performance for the PO method is better, but still lower (44.79 for $\bar{p} = 90$, 50.91 for $\bar{p} = 100$, 54.35 for $\bar{p} = 110$). 

Beside the performance, it is also useful to compare the methods in terms of computation time. Table~\ref{table:DFM_computation_time} below shows the average computation time for each of the methods. For LSM, this is just the time to apply the LSM algorithm. For PO, this time includes the time to solve the PO linear program using Gurobi and the time to execute the regression, as well as the time to generate the inner paths and the time to formulate problem in JuMP. For RPO, this time is the time to apply the backward algorithm (Algorithm~\ref{algorithm:backward}), which includes the time to solve the stage $t$ problem~\eqref{prob:backward_period_t_algo} using Adam, but does not include the time to obtain the initial starting point using LSM. 

\begin{table}
\centering
\footnotesize
\begin{tabular}{lllll} \toprule
& & \multicolumn{3}{c}{Initial price} \\
Method & Basis function architecture & $\bar{p} = 90$ & $\bar{p} = 100$ & $\bar{p} = 110$ \\ \midrule
  LSM & \textsc{one} & 2.34\enskip (0.248) & 1.45\enskip (0.020) & 2.23\enskip (0.245) \\ 
  LSM & \textsc{one, payoff} & 2.22\enskip (0.238) & 2.21\enskip (0.219) & 2.32\enskip (0.223) \\ 
  PO & \textsc{one} & 737.90\enskip (49.467) & 632.84\enskip (39.330) & 734.50\enskip (55.495) \\ 
  PO & \textsc{one, payoff} & 821.11\enskip (20.357) & 652.49\enskip (22.014) & 727.22\enskip (25.883) \\ 
  RPO & \textsc{one, payoff} & 80.65\enskip (5.193) & 332.82\enskip (28.783) & 305.96\enskip (21.902) \\ \midrule
LSM & \textsc{KOind} & 2.32\enskip (0.159) & 3.01\enskip (0.207) & 2.37\enskip (0.267) \\ 
  LSM & \textsc{KOind, payoff} & 2.79\enskip (0.318) & 3.73\enskip (0.357) & 3.70\enskip (0.335) \\ 
  PO & \textsc{KOind} & 899.30\enskip (21.038) & 851.05\enskip (18.996) & 838.18\enskip (27.319) \\ 
  PO & \textsc{KOind, payoff} & 944.45\enskip (26.243) & 819.51\enskip (20.691) & 822.24\enskip (25.099) \\ 
  RPO & \textsc{KOind, payoff} & 4.87\enskip (0.541) & 16.86\enskip (1.794) & 20.41\enskip (1.977) \\ \midrule
  LSM & \textsc{prices} & 2.22\enskip (0.190) & 3.09\enskip (0.286) & 3.02\enskip (0.218) \\ 
  LSM & \textsc{prices, payoff} & 3.18\enskip (0.294) & 4.35\enskip (0.291) & 3.25\enskip (0.203) \\ 
  PO & \textsc{prices} & 902.42\enskip (27.296) & 913.33\enskip (25.858) & 938.11\enskip (18.843) \\ 
  PO & \textsc{prices, payoff} & 1098.98\enskip (21.250) & 943.94\enskip (33.928) & 1070.94\enskip (32.581) \\ 
  RPO & \textsc{prices, payoff} & 174.12\enskip (11.573) & 565.27\enskip (21.773) & 584.55\enskip (24.491) \\ \midrule
  LSM & \textsc{pricesKO} & 3.64\enskip (0.334) & 4.32\enskip (0.493) & 4.12\enskip (0.383) \\ 
  LSM & \textsc{pricesKO, payoff} & 2.43\enskip (0.248) & 3.07\enskip (0.367) & 3.23\enskip (0.303) \\ 
  PO & \textsc{pricesKO} & 1163.28\enskip (19.574) & 1092.48\enskip (13.862) & 1093.83\enskip (17.630) \\ 
  PO & \textsc{pricesKO, payoff} & 1269.60\enskip (25.366) & 1155.69\enskip (22.082) & 1158.41\enskip (34.489) \\ 
  RPO & \textsc{pricesKO, payoff} & 6.55\enskip (0.719) & 14.52\enskip (1.077) & 14.56\enskip (1.329) \\ \midrule
  LSM & \textsc{pricesKO, KOind} & 3.90\enskip (0.302) & 5.05\enskip (0.347) & 4.62\enskip (0.346) \\ 
  LSM & \textsc{pricesKO, KOind, payoff} & 3.03\enskip (0.454) & 3.84\enskip (0.165) & 3.30\enskip (0.371) \\ 
  PO & \textsc{pricesKO, KOind} & 1268.21\enskip (35.401) & 1099.33\enskip (27.684) & 1114.49\enskip (23.758) \\ 
  PO & \textsc{pricesKO, KOind, payoff} & 1395.44\enskip (23.614) & 1251.74\enskip (35.740) & 1202.02\enskip (23.449) \\ 
  RPO & \textsc{pricesKO, KOind, payoff} & 10.46\enskip (1.729) & 22.45\enskip (1.831) & 22.43\enskip (2.184) \\ \midrule
  LSM & \textsc{pricesKO, prices2KO, KOind} & 8.41\enskip (0.353) & 7.96\enskip (0.429) & 8.02\enskip (0.552) \\ 
  LSM & \textsc{pricesKO, prices2KO, KOind, payoff} & 6.61\enskip (0.334) & 11.39\enskip (1.338) & 8.59\enskip (0.520) \\ 
  PO & \textsc{pricesKO, prices2KO, KOind} & 4712.40\enskip (48.063) & 4303.43\enskip (186.668) & 4824.68\enskip (190.066) \\ 
  PO & \textsc{pricesKO, prices2KO, KOind, payoff} & 3347.31\enskip (21.754) & 4884.33\enskip (188.597) & 4787.41\enskip (150.816) \\ 
  RPO & \textsc{pricesKO, prices2KO, KOind, payoff} & 38.18\enskip (2.942) & 87.08\enskip (8.357) & 66.91\enskip (5.050) \\ \midrule
  LSM & \textsc{pricesKO, KOind, maxpriceKO, max2priceKO} & 2.63\enskip (0.136) & 4.39\enskip (0.388) & 4.95\enskip (0.431) \\ 
  LSM & \textsc{pricesKO, KOind, maxpriceKO, max2priceKO, payoff} & 2.82\enskip (0.176) & 5.48\enskip (0.480) & 5.44\enskip (0.513) \\ 
  PO & \textsc{pricesKO, KOind, maxpriceKO, max2priceKO} & 1026.37\enskip (9.217) & 1561.72\enskip (50.908) & 1534.24\enskip (23.832) \\ 
  PO & \textsc{pricesKO, KOind, maxpriceKO, max2priceKO, payoff} & 1012.27\enskip (6.559) & 1597.34\enskip (37.282) & 1491.44\enskip (28.252) \\ 
  RPO & \textsc{pricesKO, KOind, maxpriceKO, max2priceKO, payoff} & 12.37\enskip (0.710) & 37.34\enskip (3.715) & 33.59\enskip (3.213) \\ \bottomrule
\end{tabular}

\caption{Computation time for different policies, for $n = 8$ assets. \label{table:DFM_computation_time} }
\end{table}

From this table, we can see that LSM in general requires the least amount of computation time, requiring no more than 12 seconds on average. The RPO method requires more time, but in all cases its computation time is reasonable: in general, it requires no more than 585 seconds (approximately 10 minutes) on average. We note that the computation time for RPO is in general not monotonic in in the size of the basis function architecture: for example, RPO with \textsc{one} and \textsc{payoff} (total of 2 basis functions) requires more time than RPO with \textsc{pricesKO}, \textsc{prices2KO}, \textsc{KOind}, \textsc{payoff} (total of 46 basis functions). This is likely due to the non-convex nature of the objective function in the period $t$ problem of the backward algorithm. In particular, with \textsc{one} and \textsc{payoff}, the initial starting point produced by LSM could be further away from an approximately stationary point and Adam may require more iterations before termination, whereas with \textsc{pricesKO}, \textsc{prices2KO}, \textsc{KOind} and \textsc{payoff} it may be closer and Adam may terminate more quickly. 

Comparing to the PO method, we can see that the PO method requires a significantly larger amount of time than RPO, with the average computation time ranging from about 632 seconds ($\bar{p} = 100$, PO with \textsc{one}; just over 10 minutes) to 4884 seconds ($\bar{p} = 100$, PO with \textsc{pricesKO}, \textsc{prices2KO}, \textsc{KOind}, \textsc{payoff}; roughly 80 minutes). The majority of this time comes from the generation of the inner paths, which is in general a computationally intensive task. Although RPO occasionally performs slightly worse than PO as we saw in Table~\ref{table:DFM_performance}, RPO may still be preferable to PO for obtaining good policies due to the significant computation times required by PO. 

Lastly, we note that the computation time of the RPO method is sensitive to several implementation decisions. As alluded to above, the choice of starting point for problem~\eqref{prob:backward_period_t_algo}, as well as the number of starting points used, will directly affect the time required for Adam to converge. Another decision is the step size used for Adam. In our experimentation, a smaller step size would lead to slower convergence, but would generally result in better solutions.

\section{Conclusion}
\label{sec:conclusion}
In this paper, we consider the problem of designing randomized policies for high-dimensional optimal stopping problems. We formulate the problem as an SAA problem, prove its convergence properties and establish generalization error bounds on the out-of-sample reward. Based on the NP-Hardness of the SAA problem, we develop a backward optimization heuristic for approximately solving the SAA problem. We show in the numerical experiments that our heuristic can achieve better performance than the LSM method and is better or comparable to the PO method. 

There are at least two interesting directions for future work. First, it would be interesting to further understand the behavior of the non-convex objective function of the randomized policy SAA problem and of the period $t$ problem in the backward optimization heuristic, and to understand how one can obtain high quality solutions to both of these problems. In particular, our experimentation suggests that quality of the solution in the period $t$ problem is fairly sensitive to the choice of starting point, so it would be interesting to explore other ways of selecting initial points, as well as other methods beside Adam for solving the period $t$ problem. Second, it would be interesting to explore whether our methodology can be generalized to other stochastic dynamic programming problems outside of optimal stopping.

\bibliographystyle{plainnat}
\bibliography{RPOS_literature}

\ECSwitch

\ECHead{Electronic companion for ``Randomized Policy Optimization for Optimal Stopping''}

\section{Proofs}

\subsection{Proof of Theorem~\ref{theorem:deterministic_randomized_SAA_equal}}
\label{proof_theorem:deterministic_randomized_SAA_equal}

We prove this result in two steps. We first show that $\max_{\bb \in \Bcal} \hat{J}_D(\bb) \leq \sup_{\tilde{\bb} \in \tilde{\Bcal}} \hat{J}_R(\tilde{\bb})$, and then show that $\sup_{\bb \in \Bcal} \hat{J}_D(\bb) \geq \sup_{\tilde{\bb} \in \tilde{\Bcal}} \hat{J}_R(\tilde{\bb})$.\\

\emph{Proof of $\sup_{\bb \in \Bcal} \hat{J}_D(\bb) \leq \sup_{\tilde{\bb} \in \tilde{\Bcal}} \hat{J}_R(\tilde{\bb})$:} To establish this, fix any deterministic policy weight vector $\bb \in \Bcal$. 

Without loss of generality, we can assume that $\bb_t \bullet \Phi(\xb(\omega,t))$ satisfies either  $\bb_t \bullet \Phi(\xb(\omega,t)) > 0$ or $\bb_t \bullet \Phi(\xb(\omega,t)) < 0$ for each $\omega$ and $t$. (Stated differently, $\bb_t \bullet \Phi(\xb(\omega,t))$ cannot be exactly equal to zero.) If this is not the case, then using Assumption~\ref{assumption:constant_basis_function}, we can modify the weight $b_{t,1}$ of the constant basis function $\phi_1(\xb) = 1$ for any period $t$ such that the condition is satisfied, and the sample-average reward $\hat{J}_D(\bb)$ remains unchanged.

Now, consider the randomized policy weight vector $\bb'$ defined as $\bb' = \alpha \bb$, where $\alpha > 0$. Observe now that, since $\bb_t \bullet \Phi(\xb(\omega,t)) > 0$ or $\bb_t \bullet \Phi(\xb(\omega,t)) < 0$ for each $\omega$ and $t$, we have that 
\begin{align*}
\lim_{\alpha \to +\infty} \sigma( \bb'_t \bullet \Phi(\xb(\omega,t))) & = \lim_{\alpha \to +\infty} \sigma( \alpha \bb_t \bullet \Phi(\xb(\omega,t))) \\
& = \left\{ \begin{array}{ll} +1 & \text{if}\ \bb_t \bullet \Phi(\xb(\omega,t)) > 0, \\
0 & \text{if} \ \bb_t \bullet \Phi(\xb(\omega,t)) \leq 0 \\ \end{array} \right. \\
& = \Ibb\{ \bb_t \bullet \Phi(\xb(\omega,t)) > 0 \}.
\end{align*}
Consequently, we have that
\begin{align*}
\lim_{\alpha \to +\infty} \hat{J}_R(\bb') & = \lim_{\alpha \to +\infty} \hat{J}_R(\alpha \bb) \\
& =  \lim_{\alpha \to +\infty} \frac{1}{\Omega} \sum_{\omega = 1}^{\Omega} \sum_{t=1}^T  g(t, \xb(\omega,t)) \cdot \prod_{t'=1}^{t-1} (1 - \sigma( \alpha \bb_{t'} \bullet \Phi(\xb(\omega,t')) )) \cdot \sigma ( \alpha \bb_t \bullet \Phi(\xb(\omega,t)) ) \\
& = \frac{1}{\Omega} \sum_{\omega = 1}^{\Omega} \sum_{t=1}^T g(t, \xb(\omega,t)) \cdot \prod_{t'=1}^{t-1} (1 - \Ibb\{ \bb_{t'} \bullet \Phi(\xb(\omega,t')) > 0 \} ) \cdot \Ibb\{ \bb_t \bullet \Phi(\xb(\omega,t)) > 0\} \\
& = \hat{J}_D(\bb). 
\end{align*}
Since $\bb' \in \tilde{\Bcal} = \mathbb{R}^{KT}$, we have that $\hat{J}_R( \alpha \bb ) \leq \sup_{\tilde{\bb} \in \tilde{\Bcal}} \hat{J}_R(\tilde{\bb})$ for all $\alpha > 0$; as a result, the limit of $\hat{J}_R(\alpha \bb)$ as $\alpha \to \infty$ must also be upper bounded by $\sup_{\tilde{\bb} \in \tilde{\Bcal}} \hat{J}_R(\tilde{\bb})$. We thus have that $\sup_{\tilde{\bb} \in \tilde{\Bcal}} \hat{J}_R(\tilde{\bb})$ is an upper bound on $\hat{J}_D(\bb)$ for any $\bb \in \Bcal$. 

By the definition of the supremum, it therefore follows that
\begin{equation}
\sup_{\bb \in \Bcal} \hat{J}_D(\bb) \leq \sup_{\tilde{\bb} \in \tilde{\Bcal}} \hat{J}_R(\tilde{\bb}). 
\end{equation}

\emph{Proof of $\sup_{\bb \in \Bcal} \hat{J}_D(\bb) \geq \sup_{\tilde{\bb} \in \tilde{\Bcal}} \hat{J}_R(\tilde{\bb})$:} To establish this inequality, fix a randomized policy weight vector $\tilde{\bb}$ from $\tilde{\Bcal}$. The key idea in the proof is that the logistic response function $\sigma(\cdot)$ can also be viewed as the cumulative distribution function (CDF) of a logistic random variable. Recall that a logistic random variable, $\xi \sim \Logistic(\mu, s)$, where $\mu$ is the location parameter and $s$ is the scale parameter, has CDF given by
\begin{equation*}
\Pbb( \xi < t) = \frac{ e^{ (t - \mu)/s }}{ 1 + e^{ (t - \mu)/s}}.
\end{equation*}
Thus, the logistic response function $\sigma(\cdot)$ corresponds to a $\Logistic(0,1)$ random variable. 

Armed with this insight, let us define $T$ i.i.d. $\Logistic(0,1)$ random variables, $\xi_1, \dots, \xi_T$. Observe that we can write the reward of the randomized policy as 
\begin{align}
\hat{J}_R(\tilde{\bb}) & = \frac{1}{\Omega} \sum_{\omega = 1}^{\Omega} \sum_{t=1}^T  g(t, \xb(\omega,t)) \cdot \prod_{t'=1}^{t-1} (1 - \sigma( \tilde{\bb}_{t'} \bullet \Phi(\xb(\omega,t')) )) \cdot \sigma ( \tilde{\bb}_t \bullet \Phi(\xb(\omega,t)) ) \nonumber \\
& = \frac{1}{\Omega} \sum_{\omega = 1}^{\Omega} \sum_{t=1}^T  g(t, \xb(\omega,t)) \cdot \prod_{t'=1}^{t-1} \Pbb( \xi_{t'} \geq \tilde{\bb}_{t'} \bullet \Phi( \xb(\omega,t')) )  \cdot \Pbb( \xi_t < \tilde{\bb}_t \bullet \Phi(\xb(\omega,t))  ) \nonumber \\
& = \frac{1}{\Omega} \sum_{\omega = 1}^{\Omega} \sum_{t=1}^T  g(t, \xb(\omega,t)) \cdot \prod_{t'=1}^{t-1} \Exp[ \Ibb\{ \xi_{t'} \geq \tilde{\bb}_{t'} \bullet \Phi( \xb(\omega,t')) \} ]  \cdot \Exp[ \Ibb \{  \xi_t < \tilde{\bb}_t \bullet \Phi(\xb(\omega,t))  \} ] \nonumber \\
& = \frac{1}{\Omega} \sum_{\omega = 1}^{\Omega} \sum_{t=1}^T  g(t, \xb(\omega,t)) \cdot  \Exp \left[ \prod_{t'=1}^{t-1} \Ibb\{ \xi_{t'} \geq \tilde{\bb}_{t'} \bullet \Phi( \xb(\omega,t')) \}   \cdot  \Ibb \{  \xi_t < \tilde{\bb}_t \bullet \Phi(\xb(\omega,t))  \} \right] \nonumber \\
& = \Exp \left[ \frac{1}{\Omega} \sum_{\omega = 1}^{\Omega} \sum_{t=1}^T  g(t, \xb(\omega,t)) \cdot   \prod_{t'=1}^{t-1} \Ibb\{ \xi_{t'} \geq \tilde{\bb}_{t'} \bullet \Phi( \xb(\omega,t')) \}   \cdot  \Ibb \{  \xi_t < \tilde{\bb}_t \bullet \Phi(\xb(\omega,t))  \} \right] \label{eq:J_R_exp_logistic_RVs}
\end{align}
where the second equality follows by the definition of each $\xi_t$ as a $\Logistic(0,1)$ random variable; the third by the fact that $\Pbb(A) = \Exp[\Ibb\{A\}]$ for any event $A$; the fourth by the fact that $\xi_1,\dots, \xi_T$ are independent; and the fifth by the linearity of expectation.

We now observe that there must exist values $\bar{\xi}_1,\dots, \bar{\xi}_T$ for which the random variable in \eqref{eq:J_R_exp_logistic_RVs} is at least its expected value, i.e., 
\begin{align*}
& \Exp \left[ \frac{1}{\Omega} \sum_{\omega = 1}^{\Omega} \sum_{t=1}^T  g(t, \xb(\omega,t)) \cdot   \prod_{t'=1}^{t-1} \Ibb\{ \xi_{t'} \geq \tilde{\bb}_{t'} \bullet \Phi( \xb(\omega,t')) \}   \cdot  \Ibb \{  \xi_t < \tilde{\bb}_t \bullet \Phi(\xb(\omega,t))  \} \right] \\
& \leq  \frac{1}{\Omega} \sum_{\omega = 1}^{\Omega} \sum_{t=1}^T  g(t, \xb(\omega,t)) \cdot   \prod_{t'=1}^{t-1} \Ibb\{ \bar{\xi}_{t'} \geq \tilde{\bb}_{t'} \bullet \Phi( \xb(\omega,t')) \}   \cdot  \Ibb \{  \bar{\xi}_t < \tilde{\bb}_t \bullet \Phi(\xb(\omega,t))  \}.
\end{align*}
Finally, let us define a deterministic policy weight vector $\bb$ as 
\begin{equation*}
b_{t,k} = \left\{ \begin{array}{ll} \tilde{b}_{t,k} - \bar{\xi}_t & \text{if}\ k = 1, \\
\tilde{b}_{t,k} & \text{if}\ k \neq 1, \end{array} \right.
\end{equation*}
for each $t$ and $k$. In other words, we decrease the weight on the constant basis function exactly by $\bar{\xi}_t$, the realized value of the $t$th logistic random variable. (Note that this construction is made possible by Assumption~\ref{assumption:constant_basis_function}.) By constructing $\bb$ in this way, we obtain that 
\begin{align*}
& \bar{\xi}_t < \tilde{\bb}_t \bullet \Phi(\xb(\omega,t)) \\
&  \Leftrightarrow \tilde{\bb}_t \bullet \Phi(\xb(\omega,t)) - \bar{\xi}_t > 0 \\
& \Leftrightarrow \bb_t \bullet \Phi(\xb(\omega,t)) > 0
\end{align*}
for each $\omega$ and $t$. We thus have that 
\begin{align*}
\hat{J}_R(\tilde{\bb}) & = \Exp \left[ \frac{1}{\Omega} \sum_{\omega = 1}^{\Omega} \sum_{t=1}^T  g(t, \xb(\omega,t)) \cdot   \prod_{t'=1}^{t-1} \Ibb\{ \xi_{t'} \geq \tilde{\bb}_{t'} \bullet \Phi( \xb(\omega,t')) \}   \cdot  \Ibb \{  \xi_t < \tilde{\bb}_t \bullet \Phi(\xb(\omega,t))  \} \right] \\
& \leq  \frac{1}{\Omega} \sum_{\omega = 1}^{\Omega} \sum_{t=1}^T  g(t, \xb(\omega,t)) \cdot   \prod_{t'=1}^{t-1} \Ibb\{ \bar{\xi}_{t'} \geq \tilde{\bb}_{t'} \bullet \Phi( \xb(\omega,t')) \}   \cdot  \Ibb \{  \bar{\xi}_t < \tilde{\bb}_t \bullet \Phi(\xb(\omega,t))  \} \\
& = \frac{1}{\Omega} \sum_{\omega = 1}^{\Omega} \sum_{t=1}^T  g(t, \xb(\omega,t)) \cdot   \prod_{t'=1}^{t-1} \Ibb\{ \bb_{t'} \bullet \Phi( \xb(\omega,t')) \leq 0 \}   \cdot  \Ibb \{  \bb_t \bullet \Phi(\xb(\omega,t)) >0  \} \\
& = \hat{J}_D(\bb)
\end{align*}
As a result, the reward of a randomized policy weight vector $\tilde{\bb}$ can be bounded by the reward of a deterministic policy weight vector $\bb$. Thus, $\sup_{\bb \in \Bcal} \hat{J}_D(\bb)$ is a valid upper bound on $\hat{J}_R(\tilde{\bb})$ for any $\tilde{\bb} \in \Bcal$. By the definition of the supremum as the least upper bound, we consequently have
\begin{equation}
\sup_{\tilde{\bb} \in \tilde{\Bcal}} \hat{J}_R(\tilde{\bb}) \leq \sup_{\bb \in \Bcal} \hat{J}_D(\bb).
\end{equation}

Since we have shown both inequalities, it follows $\sup_{\tilde{\bb} \in \tilde{\Bcal}} \hat{J}_R(\tilde{\bb}) = \sup_{\bb \in \Bcal} \hat{J}_D(\bb)$, as required. \Halmos

\subsection{Proof of Theorem~\ref{theorem:deterministic_randomized_full_equal}}
\label{proof_theorem:deterministic_randomized_full_equal}
	
We prove this in two steps: first, by showing that $\sup_{\bb \in \Bcal} J_D(\bb) \leq \sup_{\tilde{\bb} \in \tilde{\Bcal}} J_R(\tilde{\bb})$, and then by showing that  $\sup_{\bb \in \Bcal} J_D(\bb) \geq \sup_{\tilde{\bb} \in \tilde{\Bcal}} J_R(\tilde{\bb})$. 
	
\emph{Step 1: $\sup_{\bb \in \Bcal} J_D(\bb) \leq \sup_{\tilde{\bb} \in \tilde{\Bcal}} J_R(\tilde{\bb})$}. Let $\bb \in \Bcal$. Let $\alpha > 0$ be a constant, and define $\tilde{\bb}$ as follows:
\begin{equation*}
	\tilde{\bb}_t = \left\{ \begin{array}{ll} 
	\alpha \bb_t & \text{if}\ \bb_t \neq \zerob, \\
	-\alpha\eb_1 & \text{if}\ \bb_t = \zerob, \end{array} \right.
\end{equation*}
	where $\zerob$ is a $K$-dimensional vector of zeros and $\eb_1 = (1,0,\dots,0)$ is the first standard basis vector for $\mathbb{R}^K$. 
	
	Let $I = \{ t \in [T] \mid \bb_t \neq \zerob \}$, and for each $t \in I$, define the set $Q_t$ as 
	\begin{equation}
	Q_t = \{ (y_2,\dots, y_K) \in \mathbb{R}^{K-1} \mid b_{t,1} + \sum_{k=2}^K y_{k} b_{t,k} = 0 \}.
	\end{equation}
	Observe that $Q_t$ is a hyperplane in $\mathbb{R}^{K-1}$, so by Assumption~\ref{assumption:measure_zero}, we have that 
	\begin{equation}
	\Pbb( \Phi_{2:K}(\xb(t)) \in Q_t) = 0.
	\end{equation}
	We note that the event $\Phi_{2:K}(\xb(t)) \in Q_t$ is exactly the event that the inner product of $\bb_t$ and $\Phi(\xb(t))$ is equal to zero (i.e., we are on the boundary between choosing to stop or to continue): in particular, we have that 
	\begin{align*}
	\Phi_{2:K}(\xb(t)) \in Q_t \\
	\Leftrightarrow b_{t,1} + \sum_{k=2}^K \phi_k ( \xb(t)) b_{t,k} = 0 \\
	\Leftrightarrow \sum_{k=1}^K \phi_k(\xb(t)) b_{t,k} = 0 \\
	\Leftrightarrow \bb_t \bullet \Phi( \xb(t)) = 0
	\end{align*}
	where the third step follows because $\phi_1(\xb) = 1$ for all $\xb \in \Xcal$ (this is Assumption~\ref{assumption:constant_basis_function}). 
	
	Let $E$ be the event defined as 
	\begin{equation}
	E = \bigcup_{t \in I} \{ \Phi_{2:K} (\xb(t)) \in Q_t \}.
	\end{equation}
	Observe that $\Pbb(E) = 0$ since
	\begin{align*}
	\Pbb( E ) & = \Pbb \left( \bigcup_{t \in I} \{ \Phi_{2:K} (\xb(t)) \in Q_t \} \right) \\
	& \leq \sum_{t \in I} \Pbb( \Phi_{2:K} (\xb(t)) \in Q_t ) \\
	& = 0,
	\end{align*}
	where the inequality follows by the countable subadditivity of $\Pbb$.
	
	Observe also that for any $(\xb(1), \dots, \xb(T)) \notin E$, we have the following behavior: if $\bb_t \neq \zerob$, then 
\begin{align*}
	& \lim_{\alpha \to +\infty} \sigma( \tilde{\bb}_t \bullet \Phi(\xb(t))) \\
	& = \lim_{\alpha \to +\infty} \sigma( \alpha \bb_t \bullet \Phi(\xb(t))) \\
	& = \left\{ \begin{array}{ll} 1 & \text{if}\ \bb_t \bullet \Phi(\xb(t)) > 0, \\
	0 & \text{if}\ \bb_t \bullet \Phi(\xb(t)) \leq 0, \end{array} \right. \\
	& = \Ibb\{ \bb_t \bullet \Phi(\xb(t)) > 0 \}.
\end{align*}
	Otherwise, if $\bb_t = \zerob$, then
	\begin{align*}
	& \lim_{\alpha \to +\infty} \sigma( \tilde{\bb}_t \bullet \Phi(\xb(t))) \\
	& = \lim_{\alpha \to +\infty} \sigma( - \alpha \eb_1 \bullet \Phi(\xb(t))) \\
	& = \lim_{\alpha \to +\infty} \sigma( -\alpha) \\
	& = 0 \\
	& = \Ibb\{ \bb_t \bullet \Phi(\xb(t)) > 0\}.
	\end{align*}
	Therefore, for any $(\xb(1),\dots, \xb(T)) \notin E$, we have
	\begin{align*}
	& \lim_{\alpha \to +\infty} \sum_{t=1}^T g(t,\xb(t)) \cdot \prod_{t'=1}^{t-1} (1 - \sigma(\tilde{\bb}_t \bullet \Phi(\xb(t')))) \cdot \sigma( \tilde{\bb}_t \bullet \Phi(\xb(t))) \\
	& = \sum_{t=1}^T g(t,\xb(t)) \cdot \prod_{t'=1}^{t-1} \Ibb\{ \bb_{t'} \bullet \Phi( \xb(t')) \leq 0\} \cdot \Ibb\{ \bb_t \bullet \Phi(\xb(t)) > 0\}.
	\end{align*}
	In addition, for all $(\xb(1),\dots, \xb(T))$, the term in the limit obeys the bound 
	\begin{align*}
	& \sum_{t=1}^T g(t,\xb(t)) \cdot \prod_{t'=1}^{t-1} (1 - \sigma(\tilde{\bb}_t \bullet \Phi(\xb(t')))) \cdot \sigma( \tilde{\bb}_t \bullet \Phi(\xb(t))) \\
	& \leq \sum_{t=1}^T g(t,\xb(t)) \\
	& \leq T \cdot \bar{G},
	\end{align*}
	where the first inequality holds because $0 \leq \sigma(u) \leq 1$ for any real $u$, and the second holds by Assumption~\ref{assumption:bounded_payoff}. 
	
	Therefore, by applying the bounded convergence theorem, we can assert that 
	\begin{align}
	& \lim_{\alpha \to +\infty} J_R(\tilde{\bb}) \nonumber \\ 
	& = \lim_{\alpha \to +\infty} \Exp \left[ \sum_{t=1}^T g(t,\xb(t)) \cdot \prod_{t'=1}^{t-1} (1 - \sigma(\tilde{\bb}_t \bullet \Phi(\xb(t')))) \cdot \sigma( \tilde{\bb}_t \bullet \Phi(\xb(t))) \right] \label{eq:bct_fn} \\
	& = \Exp \left[ \sum_{t=1}^T g(t,\xb(t)) \cdot \prod_{t'=1}^{t-1} \Ibb\{ \bb_{t'} \bullet \Phi( \xb(t')) \leq 0\} \cdot \Ibb\{ \bb_t \bullet \Phi(\xb(t)) > 0\}  \right] \label{eq:bct_f_limit} \\
	& = J_D(\bb). \nonumber
	\end{align}
Note that in our application of the bounded convergence theorem, we are using the fact that the functions of $(\xb(1), \dots, \xb(T))$ whose expectation defines $J_R(\tilde{\bb})$ in \eqref{eq:bct_fn} converge pointwise to the function of $(\xb(1),\dots, \xb(T))$ whose expectation defines $J_D(\bb)$ in \eqref{eq:bct_f_limit} almost everywhere with respect to the probability measure of $(\xb(1), \dots, \xb(T))$. (The only set of values of $(\xb(1), \dots, \xb(T))$ on which the pointwise convergence does not hold is $E$, for which we have already established that $\Pbb(E) = 0$.)
	
Thus, $\lim_{\alpha \to +\infty} J_R(\tilde{\bb}) = J_D(\bb)$. Since $J_R(\tilde{\bb}) \leq \sup_{\bb' \in \tilde{\Bcal} } J_R(\bb')$ by the definition of the supremum, it then follows that for any $\alpha > 0$, 
\begin{align*}
\lim_{\alpha \to +\infty} J_R(\tilde{\bb}) \leq \sup_{\bb' \in \tilde{\Bcal}} J_R(\bb'),
\end{align*}
which implies that 
\begin{align*}
J_D(\bb) \leq \sup_{\bb' \in \tilde{\Bcal}} J_R(\bb').
\end{align*}
Since $\bb$ was arbitrary, we thus have that 
\begin{align*}
\sup_{\bb \in \Bcal} J_D(\bb) \leq \sup_{\bb' \in \tilde{\Bcal}} J_R(\bb')
\end{align*}
as required. \\

\emph{Step 2: $\sup_{\bb \in \Bcal} J_D(\bb) \geq \sup_{\tilde{\bb} \in \tilde{\Bcal}} J_R(\tilde{\bb})$}. To show this, let $\tilde{\bb}$ be any set of random policy weights in $\tilde{\Bcal}$. As in the proof of Theorem~\ref{theorem:deterministic_randomized_SAA_equal}, let us define random variables $\xi_1,\dots, \xi_T$ that are i.i.d. standard logistic random variables, that is, for each $t \in [T]$, we have:
\begin{equation*}
\Pbb( \xi_t < s) = \sigma(s)
\end{equation*}
for all $s \in \mathbb{R}$. Then observe that for a fixed trajectory $\xb(1),\dots, \xb(T)$, we can write the reward of the randomized policy with weights $\tilde{\bb}$ as 
\begin{align}
& = \sum_{t=1}^T g(t, \xb(t)) \cdot \prod_{t'=1}^{t-1} (1 - \sigma( \tilde{\bb}_{t'} \bullet \Phi(\xb(t'))) ) \cdot \sigma( \tilde{\bb}_{t} \bullet \Phi(\xb(t))) \nonumber\\ 
& = \sum_{t=1}^T g(t, \xb(t)) \cdot \prod_{t'=1}^{t-1} \Pbb( \xi_{t'} \geq \tilde{\bb}_{t'} \bullet \Phi(
\xb(t'))) \cdot \Pbb( \xi_t < \tilde{\bb}_t \bullet \Phi(\xb(t))) \nonumber\\
& = \sum_{t=1}^T g(t, \xb(t)) \cdot \prod_{t'=1}^{t-1} \Exp[ \Ibb\{ \xi_{t'} \geq \tilde{\bb}_{t'} \bullet \Phi(
\xb(t'))\} ] \cdot \Exp[ \Ibb\{ \xi_t < \tilde{\bb}_t \bullet \Phi(\xb(t))\} ] \nonumber\\
& = \sum_{t=1}^T g(t, \xb(t)) \cdot \Exp_{\xi_1, \dots, \xi_T} \left[ \prod_{t'=1}^{t-1}  \Ibb\{ \xi_{t'} \geq \tilde{\bb}_{t'} \bullet \Phi(
\xb(t'))\}  \cdot \Ibb\{ \xi_t < \tilde{\bb}_t \bullet \Phi(\xb(t))\} \right] \nonumber\\
& = \Exp_{\xi_1,\dots,\xi_T} \left[ \sum_{t=1}^T g(t, \xb(t)) \cdot \prod_{t'=1}^{t-1}  \Ibb\{ \xi_{t'} \geq \tilde{\bb}_{t'} \bullet \Phi(
\xb(t'))\}  \cdot \Ibb\{ \xi_t < \tilde{\bb}_t \bullet \Phi(\xb(t))\} \right].  \label{eq:single_path_reward}
\end{align}
We thus have that 
\begin{align*}
J_R(\tilde{\bb}) & = \Exp_{ \xb(1),\dots,\xb(T)} \left[ \sum_{t=1}^T g(t, \xb(t)) \cdot \prod_{t'=1}^{t-1} (1 - \sigma( \tilde{\bb}_{t'} \bullet \Phi(\xb(t'))) ) \cdot \sigma( \tilde{\bb}_{t} \bullet \Phi(\xb(t))) \right] \\
& = \Exp_{ \xb(1),\dots,\xb(T)} \left[ \Exp_{\xi_1,\dots,\xi_T} \left[ \sum_{t=1}^T g(t, \xb(t)) \cdot \prod_{t'=1}^{t-1}  \Ibb\{ \xi_{t'} \geq \tilde{\bb}_{t'} \bullet \Phi(
\xb(t'))\}  \cdot \Ibb\{ \xi_t < \tilde{\bb}_t \bullet \Phi(\xb(t))\} \right] \right] \\
& = \Exp_{\xi_1,\dots,\xi_T} \left[ \Exp_{ \xb(1),\dots,\xb(T)} \left[ \sum_{t=1}^T g(t, \xb(t)) \cdot \prod_{t'=1}^{t-1}  \Ibb\{ \xi_{t'} \geq \tilde{\bb}_{t'} \bullet \Phi(
\xb(t'))\}  \cdot \Ibb\{ \xi_t < \tilde{\bb}_t \bullet \Phi(\xb(t))\} \right] \right] 
\end{align*}
where the interchange of expectations in the last step follows by Fubini's theorem, since the random variable~\eqref{eq:single_path_reward} is always nonnegative. 

By the definition of expected value, there must exist a realization $\xi'_1, \dots, \xi'_T$ such that
\begin{align*}
&  \Exp_{\xi_1,\dots,\xi_T} \left[ \Exp_{ \xb(1),\dots,\xb(T)} \left[ \sum_{t=1}^T g(t, \xb(t)) \cdot \prod_{t'=1}^{t-1}  \Ibb\{ \xi_{t'} \geq \tilde{\bb}_{t'} \bullet \Phi(
\xb(t'))\}  \cdot \Ibb\{ \xi_t < \tilde{\bb}_t \bullet \Phi(\xb(t))\} \right] \right] \\
& \leq  \Exp_{ \xb(1),\dots,\xb(T)} \left[ \sum_{t=1}^T g(t, \xb(t)) \cdot \prod_{t'=1}^{t-1}  \Ibb\{ \xi'_{t'} \geq \tilde{\bb}_{t'} \bullet \Phi(
\xb(t'))\}  \cdot \Ibb\{ \xi'_t < \tilde{\bb}_t \bullet \Phi(\xb(t))\} \right]. 
\end{align*}

Now, let us define a weight vector $\bb$ for the deterministic probelm as follows:
\begin{equation}
b_{t,k} = \left\{ \begin{array}{ll} \tilde{b}_{t,k} & \text{if}\ k \neq 1, \\
\tilde{b}_{t,1} - \xi'_t & \text{if}\ k = 1, \end{array} \right.
\end{equation}
where we recall that the index $k = 1$ corresponds to the constant basis function $\phi_1(\cdot) = 1$. Observe that by the manner in which we have defined $\bb$, we have that 
\begin{align*}
& \Ibb\{ \xi_t \geq \tilde{\bb}_t \bullet \Phi(\xb(t)) \} \\
& = \Ibb\{ 0 \geq \tilde{\bb}_t \bullet \Phi(\xb(t)) - \xi_t \} \\
& = \Ibb\{ 0 \geq \bb_t \bullet \Phi(\xb(t))  \}.
\end{align*}
Thus, we have that 
\begin{align*}
J_R(\tilde{\bb}) & \leq \Exp_{ \xb(1),\dots,\xb(T)} \left[ \sum_{t=1}^T g(t, \xb(t)) \cdot \prod_{t'=1}^{t-1}  \Ibb\{ \xi'_{t'} \geq \tilde{\bb}_{t'} \bullet \Phi(
\xb(t'))\}  \cdot \Ibb\{ \xi'_t < \tilde{\bb}_t \bullet \Phi(\xb(t))\} \right] \\
& = \Exp_{ \xb(1),\dots,\xb(T)} \left[ \sum_{t=1}^T g(t, \xb(t)) \cdot \prod_{t'=1}^{t-1}  \Ibb\{ \bb_{t'} \bullet \Phi(
\xb(t')) \leq 0 \}  \cdot \Ibb\{ \bb_t \bullet \Phi(\xb(t)) > 0 \} \right] \\
& = J_D(\bb) \\
& \leq \sup_{\bb' \in \Bcal} J_D(\bb').
\end{align*}
Since $\tilde{\bb}$ was arbitrary, this implies that $\sup_{\bb' \in \Bcal} J_D(\bb')$ is an upper bound on $J_R(\tilde{\bb})$ for all $\tilde{\bb} \in \tilde{\Bcal}$, and thus that 
\begin{equation}
\sup_{\tilde{\bb} \in \tilde{\Bcal}} J_R(\tilde{\bb}) \leq \sup_{\bb \in \Bcal} J_D(\bb),
\end{equation}
as required. \Halmos

\subsection{Proof of Theorem~\ref{theorem:SAA_uniform_convergence}}
\label{proof_theorem:SAA_uniform_convergence}

To establish this result, we will show that the functions $\hat{J}_R(\cdot)$ and $J_R(\cdot)$ are Lipschitz continuous, and use this together with the compactness of $\Bcal$ to establish uniform convergence of $\hat{J}_R(\cdot)$ to $J_R(\cdot)$. To establish that these two functions are Lipschitz continuous, we need three preliminary results. The first is a basic result that the product of bounded Lipschitz continuous functions is a Lipschitz continuous function. Note that for this result and all other results in this section of the ecompanion, Lipschitz continuity is understood with respect to the $L_1$ norm, i.e., $f(\bb)$ is said to be Lipschitz continuous if there exists an $L > 0$ such that $| f(\bb) - f(\bb')| \leq L \| \bb - \bb' \|_1$ for all $\bb, \bb'$. 

\begin{lemma}
	Suppose that $f, h: \Bcal \to \mathbb{R}$ are Lipschitz continuous functions with Lipschitz constants $L_f$, and $L_h$, respectively, and are also uniformly bounded by constants $K_f$ and $K_h$, i.e., $\sup_{\bb \in \Bcal} |f(\bb)| \leq K_f$, $\sup_{\bb \in \Bcal} |h(\bb)| \leq K_h$. Then the function $w: \Bcal \to \mathbb{R}$ defined as $w(\bb) = f(\bb) h(\bb)$ is also Lipschitz continuous with Lipschitz constant $L_w = K_f L_h + K_h L_f$. \label{lemma:product_of_lipschitz_is_lipschitz}
\end{lemma}
\begin{proof}{Proof of Lemma~\ref{lemma:product_of_lipschitz_is_lipschitz}:}
Let $\bb, \bar{\bb} \in \Bcal$ and consider $|w(\bb) - w(\bar{\bb})|$: 
	\begin{align*}
	|w(\bb) - w(\bar{\bb})| & = |f(\bb)h(\bb) - f(\bar{\bb}) h(\bar{\bb})| \\
	& = |f(\bb)h(\bb) - f(\bb) h(\bar{\bb}) + f(\bb) h(\bar{\bb}) - f(\bar{\bb}) h(\bar{\bb})| \\
	& \leq |f(\bb)| \cdot |h(\bb) - h(\bar{\bb})| + |f(\bb)  - f(\bar{\bb})| \cdot |h(\bar{\bb})| \\
	& \leq K_f \cdot L_h \| \bb - \bar{\bb} \| + L_f \| \bb - \bar{\bb} \| \cdot K_h \\
	& = (K_f L_h + L_f K_h) \| \bb - \bar{\bb} \|,
	\end{align*}
	as required. \Halmos
\end{proof}

The second result that we will use is that the probabilities of stopping and continuing at time $t$ and at a state $\xb \in \Xcal$ in a randomized policy are Lipschitz continuous with respect to $\bb$. 

\begin{lemma}
	Suppose that Assumption~\ref{assumption:bounded_basis} holds. For any $t \in [T]$ and $\xb \in \Xcal$, the functions $f$ and $h$ defined as 
	\begin{align*}
	f(\bb) & = \sigma( \bb_t \bullet \Phi(\xb)), \\
	h(\bb) & = 1-\sigma( \bb_t \bullet \Phi(\xb)),
	\end{align*}
	are Lipschitz continuous with Lipschitz constant $Q$. \label{lemma:sigma_bb_bullet_phi_is_lipschitz}
\end{lemma}
\begin{proof}{Proof of Lemma~\ref{lemma:sigma_bb_bullet_phi_is_lipschitz}:}
	Observe that for $f$, the gradient of $f$ satisfies 
	\begin{align*}
	\nabla_{\bb_t} f(\bb) & = \Phi(\xb) \sigma( \bb_t \bullet \Phi(\xb)), \\
	\nabla_{\bb_{t'}} f(\bb) & = 0, \quad \forall t' \neq t. 
	\end{align*}
	Therefore, by Assumption~\ref{assumption:bounded_basis},
	\begin{equation*}
	\| \nabla f(\bb) \|_{\infty} = \| \nabla_{\bb_t} f(\bb) \|_{\infty} \leq \| \Phi(\xb) \|_{\infty} \leq Q. 
	\end{equation*}
	Now, consider $\bb$ and $\bar{\bb}$ in $\Bcal$. Since $f$ is a differentiable function, it follows by the mean value theorem that there exists a $\bb' \in \mathbb{R}^{KT}$ such that 
	\begin{equation}
	f(\bb) - f(\bar{\bb}) = \nabla f(\bb')^T (\bb - \bar{\bb}).
	\end{equation}
	We thus have
	\begin{align}
	| f(\bb) - f(\bar{\bb}) | & = | \nabla f(\bb')^T (\bb - \bar{\bb}) | \\
	& \leq \| \nabla f(\bb') \|_{\infty} \| \bb - \bar{\bb} \|_1 \\
	& \leq Q \| \bb - \bar{\bb} \|_1,
	\end{align}
	where the first inequality follows by the Cauchy-Schwartz inequality, and the second by our earlier result that the norm of the gradient of $f$ is bounded everywhere by $Q$. Thus, $f$ is Lipschitz continuous with constant $Q$. The proof for $h$ follows by an almost identical argument. \Halmos
\end{proof}

\begin{lemma}
	Suppose Assumption~\ref{assumption:bounded_basis} holds. Fix any $ (\xb(1), \dots, \xb(T)) \in \Xcal^T$, and any $t \in [T]$. The function $H_t(\cdot)$ defined as 
	\begin{equation*}
	H_t(\bb) = \prod_{t'=1}^t (1 - \sigma(b_{t'} \bullet \Phi(\xb(t'))))
	\end{equation*}
	is Lipschitz continuous with constant $tQ$. \label{lemma:Ht_is_lipschitz_continuous}
\end{lemma}

\begin{proof}{Proof of Lemma~\ref{lemma:Ht_is_lipschitz_continuous}:}
	We will prove this by induction on $t$. The base case is when $t = 1$. In this case, $H_1(\bb) = 1 - \sigma(\bb_1 \bullet \Phi(\xb(1)))$. By Lemma~\ref{lemma:sigma_bb_bullet_phi_is_lipschitz}, this function is Lipschitz continuous with constant $Q$, as required. 
	
	To establish the claim for $t \geq 2$, suppose that $H_{t-1}(\cdot)$ is Lipschitz continuous with constant $(t-1)Q$. We now need to establish that $H_{t}(\cdot)$ is Lipschitz continuous with constant $tQ$. 
	
	To see this, observe that we can write $H_t(\bb) = H_{t-1}(\bb) \cdot (1 - \sigma(\bb_t \bullet \Phi(\xb(t))))$. The function $H_{t-1}(\cdot)$ and the function $h(\bb) = 1 - \sigma( \bb_t \bullet \Phi(\xb(t)))$ are both bounded in absolute value by 1. Additionally, by Lemma~\ref{lemma:sigma_bb_bullet_phi_is_lipschitz}, the function $h(\cdot)$ is Lipschitz continuous with constant $Q$. Together with the induction hypothesis that $H_{t-1}(\cdot)$ is Lipschitz continuous with constant $(t-1)Q$, we can invoke Lemma~\ref{lemma:product_of_lipschitz_is_lipschitz} to assert that $H_t(\cdot)$ is Lipschitz continuous with constant $(t-1)Q \cdot 1 + Q \cdot 1 = tQ$. \Halmos
\end{proof}

\begin{lemma}
	Suppose Assumption~\ref{assumption:bounded_basis} holds. The function $\hat{J}_R(\cdot)$ is Lipschitz continuous with Lipschitz constant $L = \bar{G} T^2 Q$.  \label{lemma:Jhat_lipschitz_continuous}
\end{lemma}

\begin{proof}{Proof of Lemma~\ref{lemma:Jhat_lipschitz_continuous}:}
	Let $\bb, \bar{\bb} \in \Bcal$. We have
	\begin{align*}
	| \hat{J}_R(\bb) - \hat{J}_R(\bar{\bb})| 
	&= | \frac{1}{\Omega} \sum_{\omega=1}^{\Omega} \sum_{t=1}^T g(t, \xb(\omega,t)) \prod_{t'=1}^{t-1} (1 - \sigma(\bb_{t'} \bullet \Phi(\xb(\omega,t')))) \sigma(\bb_t \bullet \Phi(\xb(\omega,t))) \\&\qquad - 
	\frac{1}{\Omega} \sum_{\omega=1}^{\Omega} \sum_{t=1}^T g(t, \xb(\omega,t)) \prod_{t'=1}^{t-1} (1 - \sigma(\bar{\bb}_{t'} \bullet \Phi(\xb(\omega,t')))) \sigma(\bar{\bb}_t \bullet \Phi(\xb(\omega,t))) |  \\
	& \leq \frac{1}{\Omega} \sum_{\omega = 1}^\Omega \sum_{t=1}^T g(t, \xb(\omega,t)) | \prod_{t'=1}^{t-1} (1 - \sigma(\bb_{t'} \bullet \Phi(\xb(\omega,t')))) \sigma(\bb_t \bullet \Phi(\xb(\omega,t)))  \\&\qquad - \prod_{t'=1}^{t-1} (1 - \sigma(\bar{\bb}_{t'} \bullet \Phi(\xb(\omega,t')))) \sigma(\bar{\bb}_t \bullet \Phi(\xb(\omega,t)))  | \\
	& \leq \frac{1}{\Omega} \sum_{\omega=1}^{\Omega} \sum_{t=1}^T g(t, \xb(\omega,t)) tQ \| \bb - \bar{\bb} \|_{1} \\
	& \leq \frac{1}{\Omega} \sum_{\omega=1}^{\Omega} \sum_{t=1}^T \bar{G} T Q \| \bb - \bar{\bb} \|_{1} \\
	& = \frac{1}{\Omega} \cdot \Omega \cdot T \cdot \bar{G} T Q \| \bb - \bar{\bb} \|_{1} \\
	& = \bar{G} T^2 Q \| \bb - \bar{\bb} \|_{1} \\
	\end{align*}
	where the first inequality is just the triangle inequality; the second inequality follows by applying Lemmas~\ref{lemma:Ht_is_lipschitz_continuous}, \ref{lemma:sigma_bb_bullet_phi_is_lipschitz} and \ref{lemma:product_of_lipschitz_is_lipschitz} together; and the remaining steps follow by algebra and using the definition of $\bar{G}$ as a universal upper bound on $g(t,\xb)$ (Assumption~\ref{assumption:bounded_payoff}). \Halmos
\end{proof}

\begin{lemma}
	The function $J_R(\cdot)$ is Lipschitz continuous with Lipschitz constant $L = \bar{G} T^2 Q$. \label{lemma:J_lipschitz_continuous}
\end{lemma}
\begin{proof}{Proof of Lemma~\ref{lemma:J_lipschitz_continuous}:}
	Let $\bb, \bar{\bb} \in \Bcal$. Using similar logic as the proof of Lemma~\ref{lemma:Jhat_lipschitz_continuous}, we have
	\begin{align*}
	&| J_R(\bb) - J_R(\bar{\bb}) |\\
	& = \left| \Exp \left[ \sum_{t=1}^T g(t, \xb(t)) \prod_{t'=1}^{t-1} (1 - \sigma(\bb_{t'} \bullet \Phi(\xb(t')))) \sigma(\bb_t \bullet \Phi(\xb(t)))\right] \right. \\
	& \left. \qquad- 
	\Exp \left[ \sum_{t=1}^T g(t, \xb(t)) \prod_{t'=1}^{t-1} (1 - \sigma(\bar{\bb}_{t'} \bullet \Phi(\xb(t')))) \sigma(\bar{\bb}_t \bullet \Phi(\xb(t))) \right] \right| \\
	& \leq \Exp\left[ \sum_{t=1}^T g(t, \xb(t))  \cdot \left| \prod_{t'=1}^{t-1} (1 - \sigma(\bb_{t'} \bullet \Phi(\xb(t')))) \sigma(\bb_t \bullet \Phi(\xb(t))) - \prod_{t'=1}^{t-1} (1 - \sigma(\bar{\bb}_{t'} \bullet \Phi(\xb(t')))) \sigma(\bar{\bb}_t \bullet \Phi(\xb(t)))  \right|  \right]\\
	& \leq \Exp \left[ \sum_{t=1}^T \bar{G} T Q \| \bb - \bar{\bb} \|_{1}  \right] \\
	& = \bar{G} T^2 Q \| \bb - \bar{\bb} \|_{1},
	\end{align*}
	as required. \Halmos
\end{proof}

With Lemma~\ref{lemma:Jhat_lipschitz_continuous} and \ref{lemma:J_lipschitz_continuous}, we can prove the following theorem, which will be the final stepping stone to Theorem~\ref{theorem:SAA_uniform_convergence}. 

\begin{theorem}
Suppose that Assumptions~\ref{assumption:bounded_basis} and \ref{assumption:Bcal_compact} both hold. Fix any $\epsilon>0$. With probability one, there exists a finite sample size $N$ such that for all $\Omega \geq N$,
\begin{equation}
\sup_{\bb \in \Bcal}\  | J_R(\bb) - \hat{J}_R(\bb) | \leq \epsilon.
\end{equation}
\label{theorem:SAA_uniform_convergence_fixed_epsilon}
\end{theorem}

\begin{proof}{Proof of Theorem~\ref{theorem:SAA_uniform_convergence_fixed_epsilon}:}
	
	For the given $\epsilon$, set $\delta = \epsilon / (3L)$ where $L = \bar{G} T^2 Q$ is the Lipschitz constant of both $\hat{J}_R(\cdot)$ and $J_R(\cdot)$. Since $\Bcal$ is compact (Assumption~\ref{assumption:Bcal_compact}), there exist finitely many points $\bb^1, \dots, \bb^M$ such that $\Bcal \subseteq \bigcup_{m=1}^M B(\bb^m, \delta)$, where $B(\bb, r) = \{ \bb' \in \Bcal \mid \| \bb' - \bb \|_1 < r\}$ is the open ball of radius $r$ in the $L_1$ norm. 
	
	For each point $\bb^m$, the strong law of large numbers guarantees that $\hat{J}_R(\bb^m)$ converges to $J_R(\bb^m)$ almost surely. Thus, almost surely, there exists an integer $N_m$ such that for all $\Omega > N_m$, $| \hat{J}_R(\bb^m) - J_R(\bb^m)| < \epsilon / 3$. Let $N = \max\{ N_1,\dots, N_M\}$. Then, almost surely, for all $\Omega > N$, it holds that $| \hat{J}_R(\bb^m) - J_R(\bb^m)| < \epsilon / 3$ for all $m \in [M]$. 
	
	Now, consider any $\bb \in \Bcal$. By the definition of $\{\bb^1,\dots, \bb^M\}$ as a $\delta$-net of $\Bcal$, there exists an $m$ such that $\bb \in B(\bb^m, \delta)$. For all $\Omega > N$, we therefore have
	\begin{align*}
	|\hat{J}_R(\bb) - J_R(\bb)| & = |\hat{J}_R(\bb) - \hat{J}_R(\bb^m) + \hat{J}_R(\bb^m) - J_R(\bb^m) + J_R(\bb^m) - J_R(\bb)| \\
	& \leq |\hat{J}_R(\bb) - \hat{J}_R(\bb^m)| + |\hat{J}_R(\bb^m) - J_R(\bb^m)| + |J_R(\bb^m) - J_R(\bb)| \\
	& \leq L \| \bb - \bb^m \|_1 + \frac{\epsilon}{3} + L \| \bb - \bb^m \|_1 \\
	& \leq L \cdot \frac{\epsilon}{3L} + \frac{\epsilon}{3} + L \cdot \frac{\epsilon}{3L} \\
	& = \frac{\epsilon}{3} + \frac{\epsilon}{3} + \frac{\epsilon}{3} \\
	& = \epsilon
	\end{align*}
	where the second step follows by the triangle inequality; the third step follows by using the Lipschitz continuity of $\hat{J}_R(\cdot)$ and $J_R(\cdot)$ from Lemmas~\ref{lemma:Jhat_lipschitz_continuous} and \ref{lemma:J_lipschitz_continuous} respectively, as well as the almost sure convergence of $\hat{J}_R(\cdot)$ to $J_R(\cdot)$ at $\bb^m$; the fourth step by our definition of $\bb^m$ as the point in the $\delta$-net containing $\bb$; and the remaining steps by algebra. 
	
	Since $\bb$ was arbitrary, it follows that almost surely, for all $\Omega > N$ and all $\bb \in \Bcal$, that $| \hat{J}_R(\bb) - J_R(\bb)| < \epsilon$. This completes the proof. \Halmos
\end{proof}

Using Theorem~\ref{theorem:SAA_uniform_convergence_fixed_epsilon}, we now finally prove Theorem~\ref{theorem:SAA_uniform_convergence}.

\begin{proof}{Proof of Theorem~\ref{theorem:SAA_uniform_convergence}:}
To show that $\sup_{\bb \in \Bcal} | \hat{J}_R(\bb) - J_R(\bb)| \to 0$ as $\Omega \to \infty$ almost surely, we observe that this event can be written as 
\begin{equation*}
\bigcap_{\epsilon > 0} \bigcup_{N =1}^{\infty} \bigcap_{\Omega > N} \left\{  \sup_{\bb \in \Bcal}\ | \hat{J}_R(\bb) - J_R(\bb)| < \epsilon \right\},
\end{equation*}
which is equivalent to 
\begin{equation}
\bigcap_{k=1}^{\infty} \bigcup_{N =1}^{\infty} \bigcap_{\Omega > N} \left\{ \sup_{\bb \in \Bcal}\ | \hat{J}_R(\bb) - J_R(\bb)| < \frac{1}{2^k} \right\}. \label{eq:convergence_countable_event}
\end{equation}
The event in \eqref{eq:convergence_countable_event} is the countable intersection of events of the form $\bigcup_{N =1}^{\infty} \bigcap_{\Omega > N} \{ | \hat{J}_R(\bb) - J_R(\bb)| < 1/ 2^k \}$, each of which occurs with probability one by Theorem~\ref{theorem:SAA_uniform_convergence_fixed_epsilon}. Therefore, event~\eqref{eq:convergence_countable_event} occurs with probability 1, which establishes the required result. \Halmos
\end{proof}

\subsection{Proof of Corollary~\ref{corollary:SAA_objective_convergence}}

We will first show that if $\hat{J}_R(\cdot)$ converges uniformly to $J_R(\cdot)$ on $\Bcal$, then it must be the case that $\sup_{\bb \in \Bcal} \hat{J}_R(\bb)$ converges to $\sup_{\bb \in \Bcal} J_R(\bb)$.

Let $\epsilon > 0$. Then there exists an integer $N$ such that for all $\Omega > N$, $\sup_{\bb \in \Bcal} | \hat{J}_R(\bb) - J_R(\bb) | < \epsilon / 2$. 

Let $\Omega > N$. Suppose without loss of generality that $\sup_{\bb \in \Bcal} \hat{J}_R(\bb) \leq \sup_{\bb \in \Bcal} J_R(\bb)$. Let $\tilde{\bb} \in \Bcal$ be a weight vector such that 
\begin{equation*}
J_R(\tilde{\bb}) \geq \sup_{\bb \in \Bcal} J_R(\bb) -  \frac{\epsilon}{ 2},
\end{equation*}
or equivalently, 
\begin{equation*}
J_R(\tilde{\bb}) + \frac{\epsilon}{2} \geq \sup_{\bb \in \Bcal} J_R(\bb).
\end{equation*}
Then we have
\begin{align*}
\left| \sup_{\bb \in \Bcal} J_R(\bb) - \sup_{\bb \in \Bcal} \hat{J}_R(\bb) \right| & = \sup_{\bb \in \Bcal} J_R(\bb) - \sup_{\bb \in \Bcal} \hat{J}_R(\bb) \\
& \leq J_R(\tilde{\bb}) + \frac{\epsilon}{2} - \hat{J}_R(\tilde{\bb}) \\
& \leq \sup_{\bb \in \Bcal} | \hat{J}_R(\bb) - J_R(\bb) | + \frac{\epsilon}{2} \\
& \leq \frac{\epsilon}{2} + \frac{\epsilon}{2} \\
& = \epsilon.
\end{align*}
(In the case that $\sup_{\bb \in \Bcal} \hat{J}_R(\bb) \geq \sup_{\bb \in \Bcal} J_R(\bb)$, the same steps go through, with the modification that $\tilde{\bb}$ is chosen to be within $\epsilon/2$ of $\sup \hat{J}_R(\bb)$, i.e., $\tilde{\bb}$ satisfies $\hat{J}_R(\tilde{\bb}) \geq \sup_{\bb \in \Bcal} \hat{J}_R(\bb) - \epsilon / 2$.)

Thus, we have shown that whenever $\sup_{\bb \in \Bcal} | \hat{J}_R(\bb) - J_R(\bb) | \to 0$ as $\Omega \to \infty$, we also must have that $\sup_{\bb \in \Bcal} \hat{J}_R(\bb) \to \sup_{\bb \in \Bcal} J_R(\bb)$ as $\Omega \to \infty$. Since the former occurs with probability one by Theorem~\ref{theorem:SAA_uniform_convergence}, then it must be the case that $\lim_{\Omega \to \infty} \sup_{\bb \in \Bcal} \hat{J}_R(\bb) = \sup_{\bb \in \Bcal} J_R(\bb)$ also occurs with probability one. \Halmos

\subsection{Proof of Theorem~\ref{theorem:SAA_solution_convergence}}
By Theorem 5.3 from \cite{shapiro2014lectures}, since (i) the set $\mathbf{B}^*$ of the optimal solutions of $\sup_{\mathbf{b}\in\Bcal}J_R(\bb)$ is nonempty and $\mathbf{B}^*\subseteq\Bcal$; (ii) $J_R(\cdot)$ is continuous on $\Bcal$ as $J_R(\bb)$ is a Lipschitz continuous function of $\mathbf{b}\in\Bcal$, and $J_R(\bb)$ is finite valued as we assume the reward $g(t,\mathbf{x})$ has a finite upper bound; (iii) $\hat{J}_R(\cdot)$ converges uniformly to $J_R(\cdot)$ with probability one by Theorem~\ref{theorem:SAA_uniform_convergence}; and (iv) with probability one, for $\Omega$ large enough, the set $\hat{\mathbf{B}}_{\Omega}$ is nonempty and $\hat{\mathbf{B}}\subseteq\Bcal$; then with probability one, $\mathbb{D}(\hat{\mathbf{B}},\mathbf{B}^*) \to 0$  as $\Omega \to \infty$.\Halmos

\subsection{Proof of Proposition~\ref{proposition:generalization_bound}}
\label{proof_proposition:generalization_bound}

Our proof of Proposition~\ref{proposition:generalization_bound} follows the proof of Rademacher complexity-based generalization bounds in statistical learning (see for example Theorem 3.1 in \citealt{mohri2018foundations}). For completeness, we provide the proof here. 

Given an i.i.d. sample of system realizations $S = (Y_1,\dots, Y_{\Omega})$, let $D(S)$ be the random variable defined as 
\begin{equation*}
D(S) = \sup_{f \in \Fcal} \left( \frac{1}{\Omega} \sum_{\omega=1}^{\Omega} f(Y_{\omega}) - \Exp[ f(Y) ] \right),
\end{equation*}
where $Y$ is a random variable that represents a single system realization. Our goal will be to obtain a high probability bound on $D(S)$. We will proceed in three steps: first, we will bound the deviation of $D(S)$ from its mean $\Exp[D(S)]$; second, we will bound $\Exp[D(S)]$; and finally, we will put these two inequalities together, and show how they imply our main inequalities in terms of $J_R(\cdot)$ and $\hat{J}_R(\cdot)$. \\ %

\noindent \emph{Step 1}. Let $S'_i = (Y_1,\dots, Y'_i, \dots, Y_{\Omega})$ be a sample of system realizations that differs from $S$ only in the $i$th trajectory. It is straightforward to show that
\begin{equation*}
D(S) - D(S'_i) \leq \frac{\bar{G}}{\Omega},
\end{equation*}
and that by symmetry, $D(S'_i) - D(S) \leq \bar{G} / \Omega$ as well. Together, these two inequalities imply that $D(S)$ satisfies the bounded differences property: for any $i \in \{1,\dots, \Omega\}$, any $S$ and any $Y'_i$, we have $| D(S'_i) - D(S) | \leq \bar{G} / \Omega$. 

Thus, McDiarmid's inequality implies that with probability at least $1 - \delta$ over the sample of system realizations $S$, the following inequality holds:
\begin{equation*}
D(S) - \Exp[ D(S)] \leq \bar{G} \sqrt{ \frac{\log(1/\delta)}{2\Omega} }.
\end{equation*}
\vspace{0.5em}

\noindent \emph{Step 2}. We now bound $\Exp[ D(S) ]$. Let $S' = (Y'_1, \dots, Y'_{\Omega})$ be a second i.i.d. sample of $\Omega$ system realizations. We then have
\begin{align*}
\Exp[ D(S) ] & = \Exp_{S} \left[ \sup_{f \in \Fcal} \left( \frac{1}{\Omega} \sum_{\omega=1}^{\Omega} f(Y_{\omega}) - \Exp[ f(Y)] \right) \right] \\
 & = \Exp_{S} \left[ \sup_{f \in \Fcal} \left( \frac{1}{\Omega} \sum_{\omega=1}^{\Omega} f(Y_{\omega}) - \Exp_{S'}[ \frac{1}{\Omega} \sum_{\omega=1}^{\Omega} f(Y_{\omega}) ] \right) \right] \\
&  \leq \Exp_{S,S'} \left[ \sup_{f \in \Fcal} \left( \frac{1}{\Omega} \sum_{\omega=1}^{\Omega} f(Y_{\omega}) -  \frac{1}{\Omega} \sum_{\omega=1}^{\Omega} f(Y'_{\omega})  \right) \right] \\
 & = \Exp_{S,S', \epsilonb} \left[ \sup_{f \in \Fcal}  \frac{1}{\Omega} \sum_{\omega=1}^{\Omega} \epsilon_{\omega} (f(Y_{\omega}) -  f(Y'_{\omega}) )  \right] \\
& \leq \Exp_{S,S', \epsilonb} \left[ \sup_{f \in \Fcal}  \frac{1}{\Omega} \sum_{\omega=1}^{\Omega} \epsilon_{\omega} f(Y_{\omega})  \right] + \Exp_{S,S', \epsilonb} \left[ \sup_{f \in \Fcal}  \frac{1}{\Omega} \sum_{\omega=1}^{\Omega} \epsilon_{\omega} f(Y'_{\omega})  \right]  \\
& = 2 R(\Fcal),
\end{align*}
where $\epsilonb = (\epsilon_1,\dots, \epsilon_{\Omega})$ denotes an i.i.d. set of Rademacher random variables, that is, each $\epsilon_{\omega}$ satisfies $\Pbb( \epsilon_{\omega} = +1) = 1/2$, $\Pbb( \epsilon_{\omega} = -1) = 1/2$. \\

\noindent \emph{Step 3}. Using the results from Step 1 and Step 2, we have that $D(S) \leq 2 R(\Fcal) + \bar{G} \sqrt{ \log(1/\delta) / (2\Omega) }$. By the definition of $D$, this implies that 
\begin{equation*}
\frac{1}{\Omega} \sum_{\omega = 1}^{\Omega} f(Y_{\omega}) - \Exp[f(Y)] \leq 2 R(\Fcal) + \bar{G} \sqrt{ \frac{\log(1/\delta)}{2\Omega} }, \quad \forall f \in \Fcal,
\end{equation*}
or equivalently,
\begin{equation}
 \Exp[f(Y)] \geq \frac{1}{\Omega} \sum_{\omega = 1}^{\Omega} f(Y_{\omega}) - 2 R(\Fcal) - \bar{G} \sqrt{ \frac{\log(1/\delta)}{2\Omega} }, \quad \forall f \in \Fcal.
 \label{eq:generalization_f_Y}
\end{equation}
Note that by the definition of $\Fcal$, $f = \Gamma \circ \psi_{\bb}$ for some $\bb \in \Bcal$, and thus 
\begin{align*}
\Exp[ f(Y) ] & = \Exp[ (\Gamma \circ \psi_{\bb})(Y) ] \\
& = J_R(\bb),
\end{align*}
and 
\begin{align*}
\frac{1}{\Omega} \sum_{\omega = 1}^{\Omega} f(Y_{\omega}) & = \frac{1}{\Omega} \sum_{\omega = 1}^{\Omega} (\Gamma \circ \psi_{\bb}) (Y_{\omega}) \\
& = \hat{J}_R(\bb).
\end{align*}
Thus, \eqref{eq:generalization_f_Y} is equivalent to
\begin{equation*}
J_R(\bb) \geq \hat{J}_R(\bb) - 2 R(\Fcal) - \bar{G} \sqrt{ \frac{\log(1/\delta)}{2\Omega} }, \quad \forall \bb \in \Bcal,
\end{equation*}
which is exactly inequality~\eqref{eq:rademacher_ordinary_generalization_bound}. 

To establish inequality~\eqref{eq:rademacher_empirical_generalization_bound}, let $\hat{R}_S(\Fcal)$ be the empirical Rademacher complexity with respect to a sample of system realizations $S$. It is straightforward to verify that $\hat{R}_S(\Fcal)$ satisfies the bounded differences property with the bound $\bar{G}/\Omega$: for any sample $S'_i$ that differs from $S$ in only the $i$th trajectory, $|\hat{R}_S(\Fcal) - \hat{R}_{S'_i}(\Fcal)| \leq \bar{G} / \Omega$. By then applying McDiarmid's inequality, we can bound the deviation of $\hat{R}_S(\Fcal)$ from $R(\Fcal)$: we have
\begin{equation}
R(\Fcal) - \hat{R}_S(\Fcal) \leq \bar{G} \sqrt{ \frac{\log(1/\delta)}{2\Omega} }, \label{eq:bound_empirical_from_ordinary_Rademacher}
\end{equation}
with probability at least $1 - \delta$ over the sample of trajectories $S$. 

By now plugging in $\delta / 2$ instead of $\delta$ in both inequality~\eqref{eq:generalization_f_Y} and inequality~\eqref{eq:bound_empirical_from_ordinary_Rademacher} and combining them with the union bound, we obtain that with probability at least $1 - \delta$,
\begin{equation}
 \Exp[f(Y)] \geq \frac{1}{\Omega} \sum_{\omega = 1}^{\Omega} f(Y_{\omega}) - 2 \hat{R}_S(\Fcal) - 3\bar{G} \sqrt{ \frac{\log(2/\delta)}{2\Omega} }, \quad \forall f \in \Fcal.
 \label{eq:generalization_f_Y_empirical}
\end{equation}
This is equivalent to 
\begin{equation}
J_R(\bb) \geq \hat{J}_R(\bb) - 2 \hat{R}_S(\Fcal) - 3\bar{G} \sqrt{ \frac{\log(2/\delta)}{2\Omega} }, \quad \forall \bb \in \Bcal,
\end{equation}
which is exactly inequality~\eqref{eq:rademacher_empirical_generalization_bound}. \Halmos

\subsection{Proof of Theorem~\ref{theorem:main_rademacher_bound}}
\label{proof_theorem:main_rademacher_bound}

To prove Theorem~\ref{theorem:main_rademacher_bound}, we need to first establish a number of auxiliary results. Our first result is that the function $\Gamma$, which maps the vector produced by $\psi_{\bb}$ to an expected reward, is Lipschitz continuous with a particular constant. Note that for this result, Lipschitz continuity is understood with respect to the $L_2$ norm, as this will be needed later for the application of Maurer's contraction inequality. 

\begin{lemma}
	The function $\Gamma: \mathbb{R}^T \times [0, \bar{G}]^T \to \mathbb{R}$ is Lipschitz continuous with Lipschitz constant $\bar{G}+1$. \label{lemma:Gamma_Lipschitz}
\end{lemma}

\begin{proof}{Proof of Lemma~\ref{lemma:Gamma_Lipschitz}:}	
	To prove this, we will show that the $L_2$ norm of the gradient of $\Gamma$ can be bounded by $\bar{G}+1$. To begin, let us consider the partial derivatives of $\Gamma$:
	\begin{align}
	\frac{\partial}{\partial v_t} \Gamma & = \prod_{t'=1}^{t-1} (1 - \sigma(u_t)) \sigma(u_t), \\
	\frac{\partial}{\partial u_t} \Gamma & = v_t \sigma(u_t) (1 - \sigma(u_t)) \prod_{t'=1}^{t-1}(1 - \sigma(u_{t'})) - \sum_{t'=t+1} v_{t'} \sigma(u_t)(1 - \sigma(u_t)) \prod_{t''=1}^{t-1}(1 - \sigma(u_{t''})) \prod_{t''=t+1}^{t'-1} (1 - \sigma(u_{t''})) \sigma(u_{t'})
	\end{align}
	Observe that we can further re-arrange the partial derivative with respect to $u_t$ as
	\begin{align*}
	\frac{\partial}{\partial u_t} \Gamma & = \left[ \prod_{t'=1}^{t-1} (1 - \sigma(u_{t'})) \sigma(u_t) \right] \cdot \left[ v_t - \sum_{t'=t}^T v_{t'} \prod_{t''=t}^{t'-1} (1 - \sigma(u_{t''})) \sigma(u_{t'}) \right].
	\end{align*}
	For a fixed $t$, let us define $A_t$ as 
	\begin{equation}
	A_t = v_t - \sum_{t'=t}^T v_{t'} \prod_{t''=t}^{t'-1} (1 - \sigma(u_{t''})) \sigma(u_{t'}),
	\end{equation}
	and let us define $\tilde{p}_{t'}$ for each $t' \in \{t, \dots, T\}$ as 
	\begin{equation}
	\tilde{p}_{t'} = \prod_{t''=t}^{t'-1} (1 - \sigma(u_{t''})) \sigma(u_{t'}).
	\end{equation}
	We can thus re-write $A_t$ as $A_t = v_t - \sum_{t'=t}^T v_{t'} \tilde{p}_{t'}$, which allows us to bound it from above as follows:
	\begin{align*}
	A_t & = v_t - \sum_{t'=t}^T v_{t'} \tilde{p}_{t'} \\
	& \leq \bar{G} - \sum_{t'=t}^T 0 \tilde{p}_{t'} \\
	& = \bar{G},
	\end{align*}
	where we also use the fact that each $v_t$ is bounded between 0 and $\bar{G}$. 
	
	We can also bound $A_t$ from below as follows:
	\begin{align*}
	A_t & = v_t - \sum_{t'=t}^T v_{t'} \tilde{p}_{t'} \\
	& \geq 0 - \sum_{t'=t}^T \bar{G} \tilde{p}_{t'} \\
	& \geq - \bar{G},
	\end{align*}
	where the first inequality follows because each $v_t$ is bounded between 0 and $\bar{g}$, and the second inequality follows because each $\tilde{p}_{t'} \geq 0$ and $\sum_{t'=t}^T \tilde{p}_{t'} \leq 1$. (Each $\tilde{p}_{t'}$ can be thought of as the probability of stopping at $t'$ according to the logits given in $\ub$, conditional on starting from period $t$.) Thus, we have that $|A_t| \leq \bar{G}$. 
	
	Having defined and bounded $A_t$, let us additionally define $p_t$ as 
	\begin{equation}
	p_t = \prod_{t'=1}^{t-1} (1 - \sigma(u_{t'})) \sigma(u_t). 
	\end{equation}
	Similarly to the $\tilde{p}_{t'}$ values, it is straightforward to establish that $\sum_{t=1}^T p_t \leq 1$. With $p_t$ now defined, we can write the partial derivatives of $\Gamma$ more compactly as 
	\begin{align}
	\frac{\partial}{\partial v_t} \Gamma & = p_t, \\
	\frac{\partial}{\partial u_t} \Gamma & = p_t A_t.
	\end{align}
	We can now proceed to bound the gradient of $\Gamma$. We have
	\begin{align*}
	\| \nabla \Gamma \|_2 & = \left\| \left[ \begin{array}{cc} \nabla_\ub \Gamma \\ \nabla_\vb \Gamma \end{array} \right] \right\|_2 \\
	& \leq \| \nabla_\ub \Gamma \|_2 + \| \nabla_\vb \Gamma \|_2 \\ 
	& = \left\| \left[ \begin{array}{c} p_1 A_1 \\ \vdots \\ p_T A_T \end{array} \right] \right\|_2 + \left\| \left[ \begin{array}{c} p_1 \\ \dots \\ p_T \end{array} \right] \right\|_2 \\
	& = \sqrt{ p_1^2 A_1^2 + \dots + p_T^2 A_T^2 } + \sqrt{ p_1^2 + \dots p_T^2 } \\
	& \leq \sqrt{ p_1 A_1^2 + \dots p_T A_T^2 } + \sqrt{ p_1 + \dots + p_T } \\
	& \leq \sqrt{ p_1 \bar{G}^2 + \dots p_T \bar{G}^2 } + \sqrt{p_1 + \dots + p_T } \\
	& \leq \sqrt{ \bar{G}^2 } + \sqrt{1} \\
	& = \bar{G} + 1,
	\end{align*}
	where the first inequality follows by the fact that $p_t^2 \leq p_t$ (since each $p_t \leq 1$); the second inequality follows by the fact that $|A_t| \leq \bar{G}$ for each $t$; and the last inequality follows by the fact that $\sum_{t=1}^T p_t \leq 1$. 
	
	Having established that $\| \nabla \Gamma \|_2 \leq \bar{G} + 1$, the fact that $\Gamma$ is Lipschitz with constant $\bar{G} + 1$ follows by applying the mean value theorem and the Cauchy-Schwartz inequality. \Halmos
\end{proof}

Armed with this result that $\Gamma$ is Lipschitz, we can now relate the Rademacher complexity of $\Fcal$ (the class of functions which map system realizations to rewards) to the Rademacher complexity of the weight vector set $\Bcal$. We do so by using Maurer's vector contraction inequality \citep{maurer2016vector}, which is a result for analyzing the Rademacher complexity of a function class that arises from composing a vector-valued function with a Lipschitz function. 

\begin{lemma}
	The empirical Rademacher complexity of $\Fcal$ can be bounded as $\hat{R}(\Fcal) \leq \sqrt{2} (\bar{G}+1) \hat{R}(\Bcal)$, where the empirical Rademacher complexity $\hat{R}(\Bcal)$ of the set of feasible weight vectors is defined as 
	\begin{equation}
	\hat{R}(\Bcal) = \frac{1}{\Omega} \Exp \left[ \sup_{\bb \in \Bcal} \sum_{\omega=1}^{\Omega} \sum_{t=1}^T \epsilon_{\omega,t} \bb_t \bullet \Phi(\xb(\omega,t)) \right].
	\end{equation}
	\label{lemma:vector_contraction}
\end{lemma}
\begin{proof}{Proof of Lemma~\ref{lemma:vector_contraction}:}
	To establish this, we will use a specific form of the vector contraction inequality from \citet{maurer2016vector}, which we re-state here:
	
	\begin{lemma}[Corollary 4 of \citet{maurer2016vector}]
		Let $\Xcal$ be any set, $(x_1,\dots,x_n) \in \Xcal^n$, let $F$ be a class of functions $f: \Xcal \to \ell_2$ and let $h_i: \ell_2 \to \mathbb{R}$ have Lipschitz constant $L$. Then 
		\begin{equation}
		\Exp[ \sup_{f \in F} \sum_{i=1}^n \epsilon_i h_i( f(x_i)) ] \leq \sqrt{2} L \Exp[ \sup_{f \in F} \sum_{i,k} \epsilon_{i,k} f_k(x_i)],
		\end{equation}
		where $\ell_2$ is the set of square summable sequences of real numbers, $\{ \epsilon_i \}$ is a collection of independent Rademacher variables, $\{ \epsilon_{i,k} \}$ is a collection of independent (doubly indexed) Rademacher variables, and $f_k(x_i)$ is the $k$th component of $f(x_i)$. 
	\end{lemma}

	With this result in mind, we bound the empirical Rademacher complexity as follows:
	\begin{align*}
	\hat{R}(\Fcal)
	& = \frac{1}{\Omega} \Exp \left[ \sup_{f \in \Fcal} \sum_{\omega=1}^{\Omega} \epsilon_{\omega} f(Y_\omega) \right] \\ 
	& = \frac{1}{\Omega} \Exp \left[ \sup_{\bb \in \Bcal} \sum_{\omega=1}^{\Omega} \epsilon_{\omega} (\Gamma \circ \psi_{\bb})(Y_\omega) \right] \\ 
	& \leq \frac{1}{\Omega} \sqrt{2} (\bar{G}+1) \Exp \left[ \sup_{\bb \in \Bcal} \sum_{\omega=1}^{\Omega} \sum_{t=1}^{2T} \epsilon_{\omega,t} \psi_{\bb,t}(Y_\omega) \right] \\
	& \leq \frac{1}{\Omega} \sqrt{2}  (\bar{G}+1) \Exp \left[ \sup_{\bb \in \Bcal} \sum_{\omega=1}^{\Omega} \sum_{t=1}^{T} \epsilon_{\omega,t} \psi_{\bb,t}(Y_\omega) \right] \\
	& \phantom{\leq} + \frac{1}{\Omega} \sqrt{2}  (\bar{G}+1) \Exp \left[ \sup_{\bb \in \Bcal} \sum_{\omega=1}^{\Omega} \sum_{t=T+1}^{2T} \epsilon_{\omega,t} \psi_{\bb,t}(Y_\omega) \right] \\
	& = \frac{1}{\Omega} \sqrt{2}  (\bar{G}+1) \Exp \left[ \sup_{\bb \in \Bcal} \sum_{\omega=1}^{\Omega} \sum_{t=1}^{T} \epsilon_{\omega,T+t} \bb_t \bullet \Phi(\xb(\omega,t)) \right] \\
	& \phantom{\leq} + \frac{1}{\Omega} \sqrt{2}  (\bar{G}+1) \Exp \left[ \sup_{\bb \in \Bcal} \sum_{\omega=1}^{\Omega} \sum_{t=1}^T \epsilon_{\omega,t} g(t, \xb(\omega,t)) \right] \\
	& = \frac{1}{\Omega} \sqrt{2}  (\bar{G}+1) \Exp \left[ \sup_{\bb \in \Bcal} \sum_{\omega=1}^{\Omega} \sum_{t=1}^{T} \epsilon_{\omega,t} \bb_t \bullet \Phi(\xb(\omega,t)) \right] \\
	& = \sqrt{2} (\bar{G}+1) \hat{R}(\Bcal),
	\end{align*}
	where the first inequality follows by Lemma~\ref{lemma:Gamma_Lipschitz} and Maurer's vector contraction inequality (note that $\psi_{\bb,t}(Y)$ is used to denote the $t$th coordinate of $\psi_{\bb}(Y))$; the second inequality follows by basic properties of suprema and by linearity of expectation; the third equality follows by the definition of $\psi_{\bb}(\cdot)$; and the fourth equality follows because the last $T$ coordinates of $\psi_{\bb}(\cdot)$ do not depend on $\bb$, and thus the expectation of the weighted sum of the Rademacher random variables works out to zero. \Halmos
\end{proof}

We are now in a position to prove Theorem~\ref{theorem:main_rademacher_bound}.

\begin{proof}{Proof of Theorem~\ref{theorem:main_rademacher_bound}:}
	
	To establish each of the three statements, we first bound $\hat{R}(\Bcal)$; combining this bound with Lemma~\ref{lemma:vector_contraction} then establishes the result. We note that the proofs of part (a) and part (b) follow standard arguments for obtaining the Rademacher complexity of hypothesis classes defined by norm balls (for example, see the proofs of Theorem 11 and 12 in \citealt{liang2018notes}). \\ 
	
	\noindent \emph{Proof of Part (a):} For this result, observe that $\Bcal$ is equal to the $L_1$ ball of radius $B$, and is a bounded polyhedron. Therefore, letting $\Bcal^{ext}$ denote the set of extreme points of $\Bcal$, we can write $\Bcal$ as $\Bcal = \conv( \Bcal^{ext})$. By a standard property of Rademacher complexity, we thus have $\hat{R}(\Bcal) = \hat{R}( \Bcal^{\ext})$. 
	
	Each extreme point $\bb \in \Bcal^{ext}$ is either of the form $\bb = + B \eb^{t',k'}$ or $\bb = -B \eb^{t',k'}$, where $\eb^{t,k}$ is the standard unit vector with a one at the $(t,k)$ position, and zeros everywhere else. Thus, given $\bb = \pm B \eb^{t',k'}$, and given $\omega \in [\Omega]$ and $t \in [T]$, we will have
	\begin{align*}
	|\bb_t \bullet \Phi(\xb(\omega,t))| & = | B \eb^{t',k'} \bullet \Phi(\xb(\omega,t)) | \\
	& = B | \phi_{k'}( \xb(\omega,t) ) |
	\end{align*}
	if $t = t'$, and $|\bb_t \bullet \Phi(\xb(\omega,t))| = 0$ if $t \neq t'$. 

	Thus, given $\bb \in \Bcal^{ext}$, the vector $\wb = [ w_{\omega,t} ]_{\omega,t}$ where $w_{\omega,t} = \bb_t \bullet \Phi(\xb(\omega,t))$, has $L_2$ norm of 
	\begin{align*}
	\| \wb \|_2 & = \sqrt{ \sum_{\omega = 1}^{\Omega} \sum_{t=1}^T  w_{\omega,t}^2 } \\
	& = \sqrt{ \sum_{\omega=1}^\Omega w_{\omega,t'}^2 } \\
	& \leq \sqrt{ \sum_{\omega=1}^\Omega B^2 Q^2 } \\
	& = \sqrt{\Omega} B Q.
	\end{align*}
	
	We now recall Massart's finite lemma (see Theorem 3.3 in \citealt{mohri2018foundations}):
	\begin{lemma}[Massart's Finite Lemma]
		Let $A \subset \mathbb{R}^m$ be a finite set, with $r = \max_{\xb \in A} \| \xb \|_2$. Then we have 
		\begin{equation*}
		\Exp[ \sup_{\xb \in A} \sum_{i=1}^m x_i \epsilon_i] \leq r \sqrt{ 2 \log |A| },
		\end{equation*}
		where $\epsilon_1,\dots, \epsilon_m$ are i.i.d. Rademacher variables. 
	\end{lemma}
	
	Let $W$ consist of vectors $\wb$ constructed in the manner described above for each extreme point in $\Bcal^{ext}$. We clearly have that $|W| = |\Bcal^{ext}| = 2KT$. We therefore have
	\begin{align*}
	\hat{R}(\Bcal) & = \frac{1}{\Omega} \Exp \left[ \sup_{\bb \in \Bcal} \sum_{\omega=1}^{\Omega} \sum_{t=1}^T \epsilon_{\omega,t} \bb_t \bullet \Phi(\xb(\omega,t)) \right] \\
	& = \frac{1}{\Omega} \Exp \left[ \sup_{\bb \in \Bcal^{ext}} \sum_{\omega=1}^{\Omega} \sum_{t=1}^T \epsilon_{\omega,t} \bb_t \bullet \Phi(\xb(\omega,t)) \right] \\
	& = \frac{1}{\Omega} \Exp \left[ \sup_{\wb \in W} \sum_{\omega=1}^{\Omega} \sum_{t=1}^T \epsilon_{\omega,t} w_{\omega,t} \right] \\
	& \leq \frac{1}{\Omega} \cdot \sqrt{\Omega} BQ \cdot \sqrt{ 2\log(2KT) } \\
	& = \frac{BQ \sqrt{ 2\log(2KT) } }{\sqrt{\Omega}},
	\end{align*}
	where the inequality follows by Massart's finite lemma. \\
	
	\noindent \emph{Proof of Part (b):} For this case, observe that we can write 
	\begin{align}
	& \Exp \left[ \sup_{\bb \in \Bcal} \sum_{\omega=1}^{\Omega} \sum_{t=1}^T \epsilon_{\omega,t} \bb_t \bullet \Phi(\xb(\omega,t)) \right] \\
	& = \Exp \left[ \sup_{\bb \in \Bcal} \sum_{t=1}^T \bb_t \bullet \left[ \sum_{\omega=1}^{\Omega}  \epsilon_{\omega,t} \Phi(\xb(\omega,t)) \right]  \right] \nonumber \\ 
	& = \Exp \left[ \sup_{\bb \in \Bcal} \bb \bullet \Vb \right] \label{eq:L2_b_bullet_Vb}
	\end{align}
	where $\Vb$ is defined as
	\begin{align*}
	\Vb & = \left[ \begin{array}{c} \sum_{\omega=1}^{\Omega} \epsilon_{\omega,1} \Phi(\xb(\omega,1)) \\ 
	\vdots \\
	\sum_{\omega=1}^{\Omega} \epsilon_{\omega,T} \Phi(\xb(\omega,T)) \end{array} \right] \\
	& = \sum_{\omega=1}^{\Omega} \epsilon_{\omega,1} \left[ \begin{array}{c} \Phi(\xb(\omega,1)) \\ 
	\zerob \\
	\vdots \\
	\zerob \end{array} \right] + \dots + \sum_{\omega=1}^{\Omega} \epsilon_{\omega,T} \left[ \begin{array}{c} \zerob \\ 
	\vdots \\
	\zerob \\
	\Phi(\xb(\omega,T)) \end{array} \right]. 
	\end{align*}
	For convenience let us define the vectors $\Vb_{\omega,1}, \dots, \Vb_{\omega,T} \in \mathbb{R}^{KT}$ as 
	\begin{equation*}
	\Vb_{\omega,1} = \left[ \begin{array}{c} \Phi(\xb(\omega,1)) \\ 
	\zerob \\
	\vdots \\
	\zerob \end{array} \right], \qquad \dots \qquad, \Vb_{\omega,T} = \left[ \begin{array}{c} \zerob \\ 
	\vdots \\
	\zerob \\
	\Phi(\xb(\omega,T)) \end{array} \right],
	\end{equation*}
	so that $\Vb = \sum_{\omega=1}^{\Omega} \sum_{t=1}^T \epsilon_{\omega,t} \Vb_{\omega,t}$. 
	
	Let us now proceed with bounding \eqref{eq:L2_b_bullet_Vb}:
	\begin{align*}
	\Exp \left[ \sup_{\bb \in \Bcal} \bb \bullet \Vb \right] & = B \Exp[ \| \Vb \|_2 ] \\
	& \leq B \sqrt{ \Exp[ \| \Vb \|^2_2 ] } \\
	& =  B \sqrt{ \Exp\left[ \| \sum_{\omega=1}^{\Omega} \sum_{t=1}^T \epsilon_{\omega,t} \Vb_{\omega,t} \|^2_2  \right] } \\
	& =  B \sqrt{ \Exp \left[ \sum_{\omega=1}^{\Omega} \sum_{t=1}^T \epsilon_{\omega,t}^2 \| \Vb_{\omega,t} \|^2_2 \right] } \\
	& =  B \sqrt{ \Exp \left[ \sum_{\omega=1}^{\Omega} \sum_{t=1}^T \| \Vb_{\omega,t} \|^2_2 \right] } \\
	& =  B \sqrt{ \sum_{\omega=1}^{\Omega} \sum_{t=1}^T \| \Vb_{\omega,t} \|^2_2 ] }
	\end{align*}
	where the first step follows because the maximizing $\bb \in \Bcal$ is equal to $\bb = B \Vb / \| \Vb \|_2$; the second step follows by the concavity of $f(x) = \sqrt{x}$ and Jensen's inequality; the third step follows by the definition of the $\Vb_{\omega,t}$'s; the fourth step follows by expanding the square of the norm, and then using the independence of the $\epsilon_{\omega,t}$ to eliminate the cross-terms; and the last step by recognizing that the $\Vb_{\omega,t}$ vectors are not random. 
	
	At this juncture, we observe that the square 2-norm of the $\Vb_{\omega,t}$'s can be bounded as follows:
	\begin{align*}
	\| \Vb_{\omega,t} \|^2_2 & = \left\| \left[ \begin{array}{c} \zerob \\ \vdots \\ \zerob \\ \Phi(\xb(\omega,t)) \\ \zerob \\ \vdots \\ \zerob \end{array} \right] \right\|^2_2 \\
	& = \phi_1(\xb(\omega,t))^2 + \dots + \phi_K(\xb(\omega,t))^2 \\
	& \leq K Q^2.
	\end{align*}
	Thus, returning to our bound, we have
	\begin{align*}
	\Exp[ \sup_{\bb \in \Bcal} \bb \bullet \Vb ] & \leq  B \sqrt{ \sum_{\omega=1}^{\Omega} \sum_{t=1}^T \| \Vb_{\omega,t} \|^2_2 ] } \\
	& \leq B \sqrt{ \sum_{\omega=1}^{\Omega} \sum_{t=1}^T KQ^2 } \\
	& = B \sqrt{ \Omega T K Q^2 }\\
	& = B Q \sqrt{\Omega K T}.
	\end{align*}
	This implies that the empirical Rademacher complexity can be bounded as 
	\begin{align*}
	\hat{R}(\Bcal) & = \frac{1}{\Omega} \Exp[ \sup_{\bb \in \Bcal} \sum_{\omega=1}^{\Omega} \sum_{t=1}^T \epsilon_{\omega,t} \bb_t \bullet \Phi(\xb(\omega,t))]  \\
	& \leq \frac{1}{\Omega} \cdot B Q \sqrt{\Omega K T}. \\
	& = \frac{BQ \sqrt{KT} }{ \sqrt{\Omega}}.
	\end{align*}

	\noindent \emph{Proof of Part (c):} Using the same definition of the vector $\Vb$ as in the proof of part (b), we can write 
	\begin{align}
	& \Exp \left[ \sup_{\bb \in \Bcal} \sum_{t=1}^T \sum_{\omega = 1}^{\Omega} \epsilon_{\omega,t} \bb_t \bullet \Phi(\xb(\omega,t)) \right] \nonumber \\
	& = \Exp \left[ \sup_{\bb \in \Bcal} \sum_{t=1}^T \bb_t \bullet \left[ \sum_{\omega = 1}^{\Omega} \epsilon_{\omega,t} \Phi(\xb(\omega,t)) \right]  \right]  \nonumber \\
	& = \Exp \left[ \sup_{\bb \in \Bcal} \bb \bullet \Vb \right] \label{eq:Linfty_bb_bullet_Vb}
	\end{align}
	We now observe that for an arbitrary vector $\ab \in \mathbb{R}^n$, the optimal solution to $\max_{\xb \in \mathbb{R}^n: \| \xb \|_{\infty} \leq B } \ab \bullet \xb$ is given by $\xb = B \sign(\ab)$, where $\sign(\ab)$ is an $n$-dimensional vector with each entry carrying the sign of the corresponding coordinate of $\ab$. The objective value is given by $B \sign(\ab) \bullet \ab = B \| \ab \|_1$. Thus, we can bound \eqref{eq:Linfty_bb_bullet_Vb} as follows:
	\begin{align*}
	\Exp[ \sup_{\bb \in \Bcal} \bb \bullet \Vb ] & = B \Exp [ \| \Vb \|_1 ] \\
	& = B \Exp \left[ \sum_{t=1}^T \sum_{k=1}^K  \left| \sum_{\omega=1}^{\Omega} \epsilon_{\omega,t} \phi_k(\xb(\omega,t)) \right| \right] \\
	& = B \sum_{t=1}^T \sum_{k=1}^K \Exp\left[  \left| \sum_{\omega=1}^{\Omega} \epsilon_{\omega,t} \phi_k(\xb(\omega,t)) \right| \right] \\
	& \leq B \sum_{t=1}^T \sum_{k=1}^K \sqrt{ \Exp \left[  ( \sum_{\omega=1}^{\Omega} \epsilon_{\omega,t} \phi_k(\xb(\omega,t)) )^2 \right] } \\
	& \leq B \sum_{t=1}^T \sum_{k=1}^K \sqrt{ \Exp \left[   \sum_{\omega=1}^{\Omega} \epsilon_{\omega,t}^2 \phi_k(\xb(\omega,t))^2 \right] } \\
	& \leq B \sum_{t=1}^T \sum_{k=1}^K \sqrt{ \Omega Q^2  } \\
	& = B Q K T \sqrt{\Omega },
	\end{align*}
	where the second step follows by the definition of $\Vb$; the third step follows by the linearity of expectation; the fourth step follows by the concavity of the square root function and Jensen's inequality; the fifth step by expanding the square of the weighted sum of the $\epsilon_{\omega,t}$'s, and using the independence of the $\epsilon_{\omega,t}$'s to eliminate cross terms; the sixth step by using the definition of $Q$ and the fact that $\epsilon_{\omega,t}^2 = 1$; and the remaining steps by algebra. 
	
	We now bound the Rademacher complexity as 
	\begin{align*}
	\hat{R}(\Bcal) & = \frac{1}{\Omega} \Exp \left[ \sup_{\bb \in \Bcal} \sum_{\omega=1}^{\Omega} \sum_{t=1}^T \epsilon_{\omega,t} \bb_t \bullet \Phi(\xb(\omega,t)) \right]  \\
	& \leq \frac{1}{\Omega} \cdot B K T \sqrt{\Omega } Q \\
	& = \frac{BQKT}{\sqrt{\Omega}},
	\end{align*}
	as required. \Halmos
\end{proof}

\section{Proof of Theorem~\ref{theorem:SAA_NPHard}}
\label{proof_theorem:SAA_NPHard}

We will show that the problem is NP-Hard by showing that the decision version of the MAX-3SAT problem is equivalent to decision version of the randomized policy SAA problem. 

The MAX-3SAT problem is a well-known NP-Complete problem, which can be defined as follows. We are given $N$ binary variables, denoted by $y_1,\dots, y_N$. We also have $M$ clauses, $c_1,\dots, c_M$, where each clause is a disjunction involving three literals (one of the binary variables or its negation). As an example, a clause could be $y_1 \vee y_4 \vee \neg y_5$, which is satisfied if $y_1 = 1$, $y_4 = 1$ or $y_5 = 0$. The optimization form of the MAX-3SAT problem is to find values for the binary variables $y_1,\dots, y_N$ that maximizes the number of satisfied clauses. For our purposes, it will be easier to work with the decision form of the problem, which we state below.

\vspace{1.5em}

\begin{center}
\fbox{
\parbox{0.85\textwidth}{
\textbf{MAX-3SAT} \\
\textbf{Inputs}: 
\begin{itemize}
\item Integers $N$, $M$; 
\item Clauses $c_1,\dots, c_M$ of three literals; 
\item Target number of satisfied clauses $W$.
\end{itemize}
\textbf{Question}: Do there exist binary values $y_1,\dots,y_N$ such that the number of satisfied literals $c_1,\dots,c_M$ is at least $W$? 
}
}
\end{center}

\vspace{1.5em}

We similarly define the decision form of the randomized policy SAA problem. 

\vspace{1.5em}

\begin{center}
\fbox{
\parbox{0.85\textwidth}{
\textbf{Randomized Policy SAA} \\
\textbf{Inputs}: 
\begin{itemize}
\item Integers $\Omega$, $K$, $T$; 
\item State space $\Xcal$; 
\item Basis function mapping $\Phi(\cdot)$; 
\item Reward function $g(\cdot,\cdot)$; 
\item Sample of trajectories $\xb(1,\cdot), \dots, \xb(\Omega,\cdot)$; %
\item Set of feasible weight vectors $\Bcal \subseteq \mathbb{R}^{KT}$; %
\item Target expected reward $\theta$. 
\end{itemize}
\textbf{Question}: Does there exist a weight vector $\bb \in \Bcal$ such that the reward $\hat{J}_R(\bb) \geq \theta$? That is, is the inequality
\begin{equation*}
\frac{1}{\Omega}\sum_{\omega=1}^{\Omega} \sum_{t=1}^T g(t,\xb(\omega,t)) \prod_{t'=1}^{t-1} (1 - \sigma (\bb_{t'} \bullet \Phi(\xb(\omega,t')))) \sigma(\bb_t \bullet \Phi(\xb(\omega,t))) \geq \theta
\end{equation*}
satisfied?
}
}
\end{center}

\vspace{1.5em}

We now show how, for any arbitrary instance of the MAX-3SAT decision problem, we can construct a corresponding instance of the randomized policy SAA decision problem such that the two decision problems are equivalent (the answer to the MAX-3SAT decision problem is yes if and only if the answer to the randomized policy SAA decision problem is yes). We begin by constructing the instance, and then show the equivalence. 

\emph{Construction of instance}: Given a MAX-3SAT decision problem instance, let $\Xcal = \mathbb{R}^N$, and let the basis function mapping $\Phi$ be just equal to the identity mapping, i.e., $\Phi(\xb) = \xb$ for any $\xb \in \Xcal$. Thus, the dimension of the basis function vector $K$ is equal to $N$. 

For the trajectories, we will construct $\Omega = M$ trajectories of $T = 3$ periods. For each clause $m \in [M]$, let $i_{m,1}, i_{m,2}, i_{m,3}$ be the indices of the binary variables that participate in the clause, and let $a_{m,1}, a_{m,2}, a_{m,3}$ be equal to +1 or -1 if the literal is the binary variable itself or its negation, respectively. For example, if the clause were $y_3 \vee \neg y_4 \vee y_7$, then $i_{m,1} = 3$, $i_{m,2} = 4$, $i_{m,3} = 7$, and $a_{m,1} = +1$, $a_{m,2} = -1$, $a_{m,3} = +1$. With these definitions, let us define the trajectories as follows, for each $\omega \in [M]$, each $t \in \{1,2,3\}$: 
\begin{align*}
x_{i}(\omega,t) = \left\{ \begin{array}{ll} a_{m,t} & \text{if}\ i = i_{m,t}, \\ 0 & \text{otherwise}. \end{array}  \right.
\end{align*}
For example, for the previous clause, assuming $N = 8$, then the trajectory would be:
\begin{equation*}
\xb(m,\cdot) = \left[ \begin{array}{ccc} 
0 & 0 & 0 \\ %
0 & 0 & 0 \\ %
+1 & 0 & 0 \\ %
0 & -1 & 0 \\ %
0 & 0 & 0 \\ %
0 & 0 & 0 \\ %
0 & 0 & +1 \\ %
0 & 0 & 0  %
\end{array} \right].
\end{equation*}

For the set of feasible weight vectors, we will define $\Bcal$ as 
\begin{equation*}
\Bcal = \{ \bb \in \mathbb{R}^{KT} \mid b_{k,1} = b_{k,2} = b_{k,3} \ \text{for all} \ k \in [K] \}.
\end{equation*}
In words, the weight vector set $\Bcal$ is such that the weight of basis function $k$ is the same in all three periods. For notational convenience, we will drop the time subscript, and just use the subscript $k$ to refer to the weight of basis function $k$, e.g., $b_{k}$ instead of $b_{k,1}$. 

For the reward function $g(\cdot, \cdot)$, we simply set it as $g(t,\xb) = \Omega$ for all $t \in \{1,2,3\}$ and $\xb \in \Xcal$. 

Lastly, for the target objective value $\theta$, we set it equal to $W - 1/2$. 

To understand the strategy of our construction, let us write out the expected reward:
\begin{align}
\hat{J}_R(\bb) & = \frac{1}{\Omega} \sum_{\omega = 1}^{\Omega} \sum_{t=1}^T g(t,\xb(\omega,t)) \prod_{t'=1}^{t-1} (1 - \sigma (\bb_t \bullet \Phi(\xb(\omega,t')))) \sigma(\bb_t \bullet \Phi(\xb(\omega,t))) \nonumber  \\ 
& = \frac{1}{\Omega} \sum_{\omega = 1}^{\Omega} \Omega [ \sigma( a_{\omega,1} b_{i_{\omega,1}})
+ (1 - \sigma( a_{\omega,1} b_{i_{\omega,1}}) ) \sigma( a_{\omega,2} b_{i_{\omega,2}}) + (1 - \sigma( a_{\omega,1} b_{i_{\omega,1}}))(1 - \sigma(a_{\omega,2} b_{i_{\omega,2}})) \sigma( a_{\omega,3} b_{i_{\omega,3}}) ] \nonumber \\
 & =  \sum_{m = 1}^{M} [ \sigma( a_{m,1} b_{i_{m,1}})
+ (1 - \sigma( a_{m,1} b_{i_{m,1}}) ) \sigma( a_{m,2} b_{i_{m,2}}) + (1 - \sigma( a_{m,1} b_{i_{m,1}}))(1 - \sigma(a_{m,2} b_{i_{m,2}})) \sigma( a_{m,3} b_{i_{m,3}}) ]. \label{eq:soft_max3sat_objective}
\end{align}
To gain some intuition for how this last expression will correspond to the number of satisfied clauses, we make a couple of remarks here.

First, we will see shortly that $b_{i}$ will correspond to the binary variable $y_i$ in the MAX-3SAT problem. The weight $b_i$ can be thought of as a ``soft'' / ``continuous'', real-valued counterpart of the binary variable $y_i$; we want to use very large positive values of $b_i$ to correspond to the variable $y_i$ being equal to 1, and very small negative values of $b_i$ to correspond to the variable $y_i$ being equal to 0. 

Second, to understand how the expression in the square brackets corresponds to a clause evaluating to 1 or 0, observe that we can write a disjunction as the sum of products of the literals. For example, the clause $y_3 \vee \neg y_4 \vee y_7$ we could write as 
\begin{align}
& y_3 + (\neg y_3) \cdot (\neg y_4) + (\neg y_3) \cdot( \neg \neg y_4) \cdot y_7 \nonumber \\
& = y_3 + (1 - y_3)(1 - y_4) + (1 - y_3)(y_4)y_7. \label{eq:sum_of_binaries}
\end{align}
In the above expression, observe that if $y_3 = 1$, then the first term evaluates to 1, and the rest evaluate to 0; otherwise, if $y_3 = 0$ and $y_4 = 0$, then the first term evaluates to 0, the second to 1, and the third to 0; otherwise, if $y_3 = 0$, $y_4 = 1$ and $y_7 = 1$, then the first and second terms evaluate to 0, while the last evaluates to 1. Thus, the two expressions -- the original clause $y_3 \vee \neg y_4 \vee y_7$ and the expression~\eqref{eq:sum_of_binaries} -- are equivalent. The term in the square brackets in \eqref{eq:soft_max3sat_objective} has this same form, and we will see shortly that we can use this to establish our needed equivalence. With a slight abuse of terminology, we will refer to the term in the square brackets in \eqref{eq:soft_max3sat_objective} as the reward of a single trajectory $m$. 

We now proceed with showing the equivalence of the MAX-3SAT decision problem and the randomized policy SAA decision problem with the structure described above. \\[1em]

\emph{MAX-3SAT answer is yes $\Rightarrow$ randomized policy SAA answer is yes}: If the MAX-3SAT decision problem answer is yes, then let $y_1,\dots, y_N$ be an assignment with objective at least $W$. Let $\alpha > 0$ be a positive constant, and define a weight vector $\bb$ for the randomized policy SAA problem as follows:
\begin{equation}
b_i = \left\{ \begin{array}{ll} +\alpha & \text{if} \ y_i = 1, \\
-\alpha & \text{if} \ y_i = 0. \end{array} \right. \label{eq:bb_definition_max3sat}
\end{equation}

Observe now that for a given clause/trajectory $m$, taking the limit as $\alpha \to \infty$ of $\sigma( a_{m,t} b_{i_{m,t}})$ gives us the following:
\begin{align*}
& \lim_{\alpha \to +\infty} \sigma( a_{m,t} b_{i_{m,t}}) \\ 
& = \left\{ \begin{array}{ll} \lim_{\alpha \to +\infty} \sigma ( \alpha ) & \text{if}\ a_{m,t} = +1, y_{i_{m,t}} = 1, \\
\lim_{\alpha \to +\infty} \sigma( -\alpha) & \text{if}\ a_{m,t} = -1, y_{i_{m,t}} = 1, \\
\lim_{\alpha \to +\infty} \sigma( -\alpha) & \text{if}\ a_{m,t} = +1, y_{i_{m,t}} = 0, \\
\lim_{\alpha \to +\infty} \sigma( +\alpha) & \text{if}\ a_{m,t} = -1, y_{i_{m,t}} = 0 \end{array} \right. \\
& = \left\{ \begin{array}{ll} 1 & \text{if}\ a_{m,t} = +1, y_{i_{m,t}} = 1, \\
0 & \text{if}\ a_{m,t} = -1, y_{i_{m,t}} = 1, \\
0 & \text{if}\ a_{m,t} = +1, y_{i_{m,t}} = 0, \\
1 & \text{if}\ a_{m,t} = -1, y_{i_{m,t}} = 0 \end{array} \right.  \\
& = \left\{ \begin{array}{ll} y_{i_{m,t}} & \text{if} \ a_{m,t} = +1, \\
\neg y_{i_{m,t}} & \text{if} \ a_{m,t} = -1 \end{array} \right.
\end{align*}
In other words, as $\alpha \to \infty$, $\sigma( a_{m,t} b_{i_{m,t}})$ evaluates to exactly the $t$th literal of clause $m$. By our aforementioned equivalence of a disjunction and a sum of products of binary variables (as in the example in equation~\eqref{eq:sum_of_binaries}), it follows that
\begin{align*}
\lim_{\alpha \to + \infty} \hat{J}_R(\bb) & = \lim_{\alpha \to +\infty} \sum_{m = 1}^{M} [ \sigma( a_{m,1} b_{i_{m,1}})
+ (1 - \sigma( a_{m,1} b_{i_{m,1}}) ) \sigma( a_{m,2} b_{i_{m,2}})   \\
& \phantom{= \lim_{\alpha \to +\infty} \sum_{m = 1}^{M} } + (1 - \sigma( a_{m,1} b_{i_{m,1}}))(1 - \sigma(a_{m,2} b_{i_{m,2}})) \sigma( a_{m,3} b_{i_{m,3}}) ] \\
& = \sum_{m=1}^M c_m,
\end{align*}
i.e., the limit as $\alpha$ goes to infinity is exactly equal to the number of satisfied clauses in the MAX-3SAT solution $y_1,\dots, y_N$. Since the answer to the MAX-3SAT decision problem is yes, we know that $\sum_{m=1}^M c_m \geq W$, so that the limit $\lim_{\alpha \to +\infty} \hat{J}_R(\bb) \geq W$ as well. Since the limit is at least $W$, it follows that there must exist an $\alpha$, and thus a corresponding $\bb$ (as defined in \eqref{eq:bb_definition_max3sat}) such that $\hat{J}_R(\bb) \geq W - 1/2$. \\[1em]

\emph{Randomized policy SAA answer is yes $\Rightarrow$ MAX-3SAT answer is yes}: To show the other direction of the equivalence, let us suppose we have a solution $\bb$ for the randomized policy SAA problem with objective value $\hat{J}_R(\bb) \geq W - 1/2$. We now need to construct a solution for the MAX-3SAT decision problem with objective value at least $W$.

Let us use $c_m( y_1, \dots, y_N )$ to denote the value of clause $m$ as a function of the binary variables $y_1, ..., y_N$. 
We claim that 
\begin{equation}
\hat{J}_R(\bb) = \Exp \left[ \sum_{m=1}^M c_m( \Ibb\{ \xi_1 \leq b_1\}, \dots, \Ibb\{ \xi_N \leq b_N \} ) \right], \label{eq:equivalence_hatJ_expectationoverxis}
\end{equation} 
where $\xi_1,\dots, \xi_N$ are i.i.d. standard logistic random variables (i.e., $\Pbb( \xi_i \leq t ) = \sigma(t)$ for all the variables $i$). Once we show this, we can use the probabilistic method to assert the existence of $y_1,\dots, y_N$ that give an affirmative answer to the MAX-3SAT problem.

To show the equivalence~\eqref{eq:equivalence_hatJ_expectationoverxis}, we argue that for any clause $m$,
\begin{align}
& \Exp[ c_m( \Ibb\{ \xi_1 \leq b_1\}, \dots, \Ibb\{ \xi_N \leq b_N \} )] \nonumber \\
& = \sigma( a_{m,1} b_{i_{m,1}})
+ (1 - \sigma( a_{m,1} b_{i_{m,1}}) ) \sigma( a_{m,2} b_{i_{m,2}}) + (1 - \sigma( a_{m,1} b_{i_{m,1}}))(1 - \sigma(a_{m,2} b_{i_{m,2}})) \sigma( a_{m,3} b_{i_{m,3}}). \label{eq:single_clause_equals_single_reward}
\end{align}
To see why this must be true, we argue by way of an example. Consider again the example clause $y_3 \vee \neg y_4 \vee y_7$. Consider the right-hand side of \eqref{eq:single_clause_equals_single_reward}, which is the reward of the corresponding trajectory, after we substitute in the values of the $a_{m,t}$'s. This right hand side works out to
\begin{equation*}
\sigma(b_3) + (1 - \sigma(b_3))\sigma(-b_4) + (1 - \sigma(b_3)(1 - \sigma(-b_4)) \sigma(b_7). 
\end{equation*}
We now use an important property of the logistic response function $\sigma$, which is that for any real $u$, $\sigma(u) = 1 - \sigma(-u)$. Therefore, we can readily modify the above expression so that the coefficient of any $b_i$ is always +1:
\begin{equation*}
\sigma(b_3) + (1 - \sigma(b_3))(1 - \sigma(b_4)) + (1 - \sigma(b_3)\sigma(b_4) \sigma(b_7). 
\end{equation*}
Letting $\xi_1,\dots, \xi_N$ denote i.i.d. standard logistic random variables, the above can be equivalently written as 
\begin{align}
& \Pbb( \xi_3 \leq b_3) + (1 - \Pbb(\xi_3 \leq b_3))(1 - \Pbb( \xi_4 \leq b_4)) + (1 - \Pbb(\xi_3 \leq b_3))\cdot \Pbb(\xi_4 \leq b_4) \cdot \Pbb( \xi_7 \leq b_7) \\
& = \Exp[ \Ibb\{ \xi_3 \leq b_3\} ] + \Exp[ 1 - \Ibb\{ \xi_3 \leq b_3\}] \Exp[ 1 - \Ibb\{ \xi_4 \leq b_4\}] + \Exp[ 1 - \Ibb\{ \xi_3 \leq b_3\}] \Exp[ \Ibb\{ \xi_4 \leq b_4\}] \Exp[ \Ibb\{ \xi_7 \leq b_7\}] \nonumber \\
& = \Exp[ \Ibb\{ \xi_3 \leq b_3\}  + (1 - \Ibb\{ \xi_3 \leq b_3\})(1 - \Ibb\{ \xi_4 \leq b_4\}) + 
(1 - \Ibb\{ \xi_3 \leq b_3\}) \Ibb\{ \xi_4 \leq b_4\} \Ibb\{ \xi_7 \leq b_7\} ], \label{eq:max3sat_rv_expectedvalue}
\end{align}
where the equality on the final line follows by the independence of the $\xi$'s and the linearity of expectation. Now, let $y_3 = \Ibb\{ \xi_3 \leq b_3 \}$, $y_4 = \Ibb\{ \xi_4 \leq b_4\}$ and $y_7 = \Ibb\{ \xi_7 \leq b_7 \}$. Observe that the expression inside the expectation in \eqref{eq:max3sat_rv_expectedvalue} can be written as 
\begin{equation*}
y_3 + (1 - y_3)(1 - y_4) + (1 - y_3) y_4 y_7
\end{equation*}
which is logically identical to $y_3 \vee \neg y_4 \vee y_7$. Thus, in this example, it follows that equation~\eqref{eq:single_clause_equals_single_reward} holds. Note that there is nothing special in the particular clause that we chose; the same procedure, which involves using the identity $\sigma(-u) = 1 - \sigma(u)$ to eliminate any term of the form $\sigma(-b_i)$ that appears in the right-hand side of \eqref{eq:single_clause_equals_single_reward}, can be used to turn the right-hand side of \eqref{eq:single_clause_equals_single_reward} into the expected value of the clause function $c_m(y_1,\dots, y_N)$ when one replaces each $y_i$ with $\Ibb\{ \xi_i \leq b_i\}$. 

Since \eqref{eq:single_clause_equals_single_reward} holds, by linearity of expectation it must be the case that \eqref{eq:equivalence_hatJ_expectationoverxis} also holds. Consequently, there must exist values $\xi'_1,\dots, \xi'_N$ of the random variables $\xi_1,\dots, \xi_N$ which satisfy the following:
\begin{align}
& \Exp[ \sum_{m=1}^M c_m( \Ibb\{ \xi_1 \leq b_1\}, \dots, \Ibb\{ \xi_N \leq b_N \} ) ] \nonumber \\
& \leq \sum_{m=1}^M c_m( \Ibb\{ \xi'_1 \leq b_1\}, \dots, \Ibb\{ \xi'_N \leq b_N \} ). \label{eq:max3sat_xiprime_geq_expectation}
\end{align}
Define now a candidate solution to the MAX-3SAT problem $y_1,\dots,y_N$ as $y_i = \Ibb\{ \xi'_i \leq b_i\}$ for each $i$. By \eqref{eq:max3sat_xiprime_geq_expectation} and \eqref{eq:equivalence_hatJ_expectationoverxis}, we have
\begin{equation*}
\sum_{m=1}^M c_m( y_1,\dots, y_N ) \geq \hat{J}_R(\bb).
\end{equation*}
Recall that $\hat{J}_R(\bb) \geq W - 1/2$, so we further have that 
\begin{equation*}
\sum_{m=1}^M c_m( y_1, \dots, y_N) \geq W - 1/2.
\end{equation*}
Since $W$ is an integer, and the number of satisfied clauses must also be an integer, the above is equivalent to
\begin{equation*}
\sum_{m=1}^M c_m( y_1, \dots, y_N) \geq W,
\end{equation*}
which shows that the answer to the MAX-3SAT decision problem is yes. \\

We have shown that the MAX-3SAT decision problem and randomized policy SAA decision problem are equivalent for the constructed instance of the randomized policy SAA problem. Since the particular instance of the randomized policy SAA decision problem can be constructed in polynomial time, and since the MAX-3SAT problem is NP-Complete \citep{garey2002computers}, it follows that the randomized policy SAA decision problem is NP-Hard. \Halmos

\section{Additional numerical results}

\subsection{Warm starting of RPO method using LSM}
\label{subsec:additional_numerics_warmstart}

In this section, we briefly describe how we use the LSM solution to warm start each solve of problem~\eqref{prob:backward_period_t_algo}. Suppose that the basis function set contains \textsc{payoff}, i.e., the undiscounted payoff $g'(t)$ is a basis function. Let $\bb_t = (b_{t,1}, \dots, b_{t,K})$ be the vector of weights for the LSM algorithm, as we have defined it in Section~\ref{subsec:solution_methodology_comparison_LSM} (Algorithm~\ref{algorithm:LSM}). The LSM policy stops at time $t$ if and only if 
\begin{equation*}
g(t) > \sum_{k=1}^K b_{t,k} \phi_k(\xb(t)).
\end{equation*}
Using the fact that $g(t) = \beta^{t} g'(t) = \beta^{t} \phi_K(\xb(t))$, we can re-write this as 
\begin{align*}
g(t) - \sum_{k=1}^K b_{t,k} \phi_k(\xb(t)) & > 0\\
\Rightarrow \beta^{t} \phi_K(\xb(t)) - \sum_{k=1}^K b_{t,k} \phi_k(\xb(t)) & > 0 \\
\Rightarrow \sum_{k=1}^K b'_{t,k} \phi_k(\xb(t)) & > 0,
\end{align*}
where the vector $\bb'_t$ is defined as $\bb'_t = ( - \beta^{-t} b_{t,1}, \dots, - \beta^{-t} b_{t,K-1}, 1 - \beta^{-t} b_{t,K})$.

Observe that, as discussed in Section~\ref{subsec:solution_methodology_comparison_LSM}, $\bb'_t$ can be viewed as a weight vector defining a deterministic linear policy at time $t$, that would behave identically to the LSM policy at time $t$. At the same time, one can also treat $\bb'_t$ as a candidate weight vector for a randomized policy at time $t$. Thus, our warm starting strategy is to simply use $\bb'_t$ as the initial solution to problem~\eqref{prob:backward_period_t_algo}.

\subsection{Additional policy performance results for Section~\ref{subsec:numerics_neq8}}
\label{subsec:additional_numerics_neq4_16}

Table~\ref{table:DFM_performance_neq4} displays the results comparing LSM, PO and RPO for instances with $n = 4$ assets, while Table~\ref{table:DFM_performance_neq16} displays analogous results for $n = 16$ assets. Note that for $n = 16$ assets, we omit the results for PO for the basis function architecture containing the second-order price basis functions (\textsc{prices2KO}) due to the significant computational effort required for the PO method in this case. 

\begin{table}
\centering
\footnotesize
\begin{tabular}{lllll} \toprule 
& & \multicolumn{3}{c}{Initial price} \\
Method & Basis function architecture & $\bar{p} = 90$ & $\bar{p} = 100$ & $\bar{p} = 110$ \\ \midrule
  LSM & \textsc{one} & 24.68\enskip (0.019) & 31.78\enskip (0.016) & 37.45\enskip (0.038) \\ 
  LSM & \textsc{one, payoff} & 32.84\enskip (0.030) & 40.02\enskip (0.047) & 43.16\enskip (0.043) \\ 
  PO & \textsc{one} & 30.84\enskip (0.024) & 38.97\enskip (0.019) & 44.57\enskip (0.027) \\ 
  PO & \textsc{one, payoff} & 22.67\enskip (0.167) & 20.77\enskip (0.126) & 16.53\enskip (0.127) \\ 
  RPO & \textsc{one, payoff} & \bfseries 34.48\enskip (0.020) & \bfseries 42.92\enskip (0.020) & \bfseries 49.16\enskip (0.020) \\[0.5em] 
  PO-UB & \textsc{one} & 43.23\enskip (0.032) & 51.11\enskip (0.024) & 56.46\enskip (0.022) \\ 
  PO-UB & \textsc{one, payoff} & 35.11\enskip (0.023) & 43.94\enskip (0.034) & 50.55\enskip (0.032) \\ \midrule

  LSM & \textsc{prices} & 25.74\enskip (0.025) & 32.08\enskip (0.025) & 37.38\enskip (0.040) \\ 
  LSM & \textsc{prices, payoff} & 32.34\enskip (0.021) & 38.14\enskip (0.040) & 40.74\enskip (0.030) \\ 
  PO & \textsc{prices} & 31.40\enskip (0.023) & 38.92\enskip (0.015) & 43.42\enskip (0.017) \\ 
  PO & \textsc{prices, payoff} & 23.04\enskip (0.138) & 19.94\enskip (0.099) & 15.63\enskip (0.095) \\ 
  RPO & \textsc{prices, payoff} & \bfseries 33.96\enskip (0.018) & \bfseries 42.03\enskip (0.013) & \bfseries 47.89\enskip (0.020) \\[0.5em] 
  PO-UB & \textsc{prices} & 40.57\enskip (0.022) & 49.27\enskip (0.011) & 55.62\enskip (0.018) \\ 
  PO-UB & \textsc{prices, payoff} & 35.11\enskip (0.023) & 43.94\enskip (0.034) & 50.53\enskip (0.032) \\ \midrule

  LSM & \textsc{pricesKO} & 28.53\enskip (0.029) & 38.34\enskip (0.018) & 46.55\enskip (0.034) \\ 
  LSM & \textsc{pricesKO, payoff} & 33.45\enskip (0.018) & 41.71\enskip (0.019) & 47.73\enskip (0.016) \\ 
  PO & \textsc{pricesKO} & 32.68\enskip (0.024) & 41.84\enskip (0.016) & 47.78\enskip (0.018) \\ 
  PO & \textsc{pricesKO, payoff} & 32.67\enskip (0.027) & 41.52\enskip (0.020) & 48.02\enskip (0.019) \\ 
  RPO & \textsc{pricesKO, payoff} & \bfseries 33.98\enskip (0.020) & \bfseries 42.14\enskip (0.017) & \bfseries 48.17\enskip (0.016) \\[0.5em] 
  PO-UB & \textsc{pricesKO} & 39.52\enskip (0.020) & 46.89\enskip (0.012) & 51.89\enskip (0.012) \\ 
  PO-UB & \textsc{pricesKO, payoff} & 35.07\enskip (0.020) & 43.79\enskip (0.030) & 50.17\enskip (0.026) \\ \midrule
  
LSM & \textsc{KOind} & 26.19\enskip (0.027) & 35.61\enskip (0.020) & 44.02\enskip (0.048) \\ 
  LSM & \textsc{KOind, payoff} & 33.39\enskip (0.028) & 41.89\enskip (0.028) & 48.06\enskip (0.022) \\ 
  PO & \textsc{KOind} & 31.51\enskip (0.025) & 41.04\enskip (0.018) & 48.43\enskip (0.024) \\ 
  PO & \textsc{KOind, payoff} & 32.22\enskip (0.047) & 42.28\enskip (0.029) & 49.01\enskip (0.016) \\ 
  RPO & \textsc{KOind, payoff} & \bfseries 34.53\enskip (0.020) & \bfseries 43.07\enskip (0.020) & \bfseries 49.39\enskip (0.019) \\[0.5em] 
  PO-UB & \textsc{KOind} & 41.46\enskip (0.028) & 48.38\enskip (0.022) & 52.83\enskip (0.018) \\ 
  PO-UB & \textsc{KOind, payoff} & 35.08\enskip (0.021) & 43.79\enskip (0.031) & 50.18\enskip (0.027) \\ \midrule

  LSM & \textsc{pricesKO, KOind} & 30.23\enskip (0.030) & 39.07\enskip (0.015) & 46.59\enskip (0.029) \\ 
  LSM & \textsc{pricesKO, KOind, payoff} & 32.72\enskip (0.023) & 41.24\enskip (0.023) & 47.74\enskip (0.025) \\ 
  PO & \textsc{pricesKO, KOind} & 31.88\enskip (0.019) & 40.61\enskip (0.027) & 48.41\enskip (0.025) \\ 
  PO & \textsc{pricesKO, KOind, payoff} & 31.40\enskip (0.030) & 40.59\enskip (0.019) & \bfseries 48.45\enskip (0.020) \\ 
  RPO & \textsc{pricesKO, KOind, payoff} & \bfseries 32.95\enskip (0.023) & \bfseries 41.42\enskip (0.025) & 48.09\enskip (0.036) \\[0.5em] 
  PO-UB & \textsc{pricesKO, KOind} & 38.82\enskip (0.016) & 46.45\enskip (0.016) & 51.75\enskip (0.014) \\ 
  PO-UB & \textsc{pricesKO, KOind, payoff} & 35.07\enskip (0.021) & 43.78\enskip (0.030) & 50.16\enskip (0.026) \\ \midrule

  LSM & \textsc{pricesKO, prices2KO, KOind} & 31.92\enskip (0.032) & 40.93\enskip (0.014) & 47.74\enskip (0.019) \\ 
  LSM & \textsc{pricesKO, prices2KO, KOind, payoff} & 33.41\enskip (0.023) & 41.82\enskip (0.021) & 48.02\enskip (0.021) \\ 
  PO & \textsc{pricesKO, prices2KO, KOind} & 32.18\enskip (0.028) & 41.88\enskip (0.017) & 48.73\enskip (0.015) \\ 
  PO & \textsc{pricesKO, prices2KO, KOind, payoff} & 33.66\enskip (0.021) & 42.48\enskip (0.017) & 48.78\enskip (0.015) \\ 
  RPO & \textsc{pricesKO, prices2KO, KOind, payoff} & \bfseries 33.97\enskip (0.026) & \bfseries 42.59\enskip (0.021) & \bfseries 48.93\enskip (0.022) \\[0.5em] 
  PO-UB & \textsc{pricesKO, prices2KO, KOind} & 36.30\enskip (0.010) & 44.56\enskip (0.011) & 50.51\enskip (0.011) \\ 
  PO-UB & \textsc{pricesKO, prices2KO, KOind, payoff} & 35.07\enskip (0.021) & 43.74\enskip (0.025) & 50.08\enskip (0.023) \\ \midrule

  LSM & \textsc{pricesKO, KOind, maxpriceKO, max2priceKO} & 32.93\enskip (0.023) & 41.37\enskip (0.020) & 47.81\enskip (0.025) \\ 
  LSM & \textsc{pricesKO, KOind, maxpriceKO, max2priceKO, payoff} & 32.99\enskip (0.025) & 41.38\enskip (0.018) & 47.79\enskip (0.024) \\ 
  PO & \textsc{pricesKO, KOind, maxpriceKO, max2priceKO} & 32.52\enskip (0.024) & 40.92\enskip (0.020) & 48.48\enskip (0.019) \\ 
  PO & \textsc{pricesKO, KOind, maxpriceKO, max2priceKO, payoff} & 32.23\enskip (0.027) & 41.12\enskip (0.020) & \bfseries 48.49\enskip (0.018) \\ 
  RPO & \textsc{pricesKO, KOind, maxpriceKO, max2priceKO, payoff} & \bfseries 33.23\enskip (0.024) & \bfseries 41.59\enskip (0.022) & 48.16\enskip (0.035) \\[0.5em] 
  PO-UB & \textsc{pricesKO, KOind, maxpriceKO, max2priceKO} & 35.38\enskip (0.020) & 43.84\enskip (0.029) & 50.17\enskip (0.025) \\ 
  PO-UB & \textsc{pricesKO, KOind, maxpriceKO, max2priceKO, payoff} & 35.06\enskip (0.022) & 43.77\enskip (0.030) & 50.16\enskip (0.025) \\ \bottomrule

\end{tabular}

\caption{Out-of-sample performance for different policies, for $n = 4$ assets. \label{table:DFM_performance_neq4}}

\end{table}

\begin{table}
\centering
\footnotesize
\begin{tabular}{lllll} \toprule 
& & \multicolumn{3}{c}{Initial price} \\
Method & Basis function architecture & $\bar{p} = 90$ & $\bar{p} = 100$ & $\bar{p} = 110$ \\ \midrule
  LSM & \textsc{one} & 39.08\enskip (0.015) & 43.20\enskip (0.016) & 47.14\enskip (0.017) \\ 
  LSM & \textsc{one, payoff} & 43.15\enskip (0.033) & 45.15\enskip (0.016) & 47.47\enskip (0.020) \\ 
  PO & \textsc{one} & 46.29\enskip (0.018) & 48.93\enskip (0.014) & 51.07\enskip (0.009) \\ 
  PO & \textsc{one, payoff} & 18.10\enskip (0.142) & 15.89\enskip (0.277) & 34.50\enskip (0.257) \\ 
  RPO & \textsc{one, payoff} & \bfseries 51.52\enskip (0.028) & \bfseries 52.73\enskip (0.040) & \bfseries 53.60\enskip (0.028) \\[0.5em] 
  PO-UB & \textsc{one} & 57.57\enskip (0.008) & 60.29\enskip (0.011) & 61.87\enskip (0.007) \\ 
  PO-UB & \textsc{one, payoff} & 53.21\enskip (0.035) & 56.11\enskip (0.037) & 57.40\enskip (0.039) \\ \midrule

  LSM & \textsc{prices} & 38.97\enskip (0.019) & 43.12\enskip (0.017) & 47.06\enskip (0.018) \\ 
  LSM & \textsc{prices, payoff} & 42.22\enskip (0.026) & 44.55\enskip (0.019) & 47.13\enskip (0.021) \\ 
  PO & \textsc{prices} & 45.57\enskip (0.016) & 48.05\enskip (0.013) & 50.37\enskip (0.007) \\ 
  PO & \textsc{prices, payoff} & 18.14\enskip (0.098) & 16.35\enskip (0.242) & 34.82\enskip (0.081) \\ 
  RPO & \textsc{prices, payoff} & \bfseries 50.00\enskip (0.033) & \bfseries 52.07\enskip (0.028) & \bfseries 53.57\enskip (0.032) \\[0.5em] 
  PO-UB & \textsc{prices} & 57.47\enskip (0.004) & 60.27\enskip (0.011) & 61.84\enskip (0.008) \\ 
  PO-UB & \textsc{prices, payoff} & 53.16\enskip (0.033) & 56.03\enskip (0.034) & 57.31\enskip (0.037) \\ \midrule
  
    LSM & \textsc{pricesKO} & 50.31\enskip (0.009) & 53.39\enskip (0.011) & 54.70\enskip (0.008) \\ 
  LSM & \textsc{pricesKO, payoff} & 50.28\enskip (0.011) & 52.93\enskip (0.010) & 54.46\enskip (0.009) \\ 
  PO & \textsc{pricesKO} & 50.84\enskip (0.010) & 53.44\enskip (0.011) & 55.03\enskip (0.008) \\ 
  PO & \textsc{pricesKO, payoff} & 50.83\enskip (0.009) & 53.45\enskip (0.008) & 54.95\enskip (0.006) \\ 
  RPO & \textsc{pricesKO, payoff} & \bfseries 50.92\enskip (0.010) &\bfseries 53.60\enskip (0.010) & \bfseries 55.22\enskip (0.010) \\[0.5em] 
  PO-UB & \textsc{pricesKO} & 53.31\enskip (0.007) & 55.44\enskip (0.007) & 56.70\enskip (0.006) \\ 
  PO-UB & \textsc{pricesKO, payoff} & 52.49\enskip (0.022) & 55.07\enskip (0.017) & 56.41\enskip (0.016) \\ \midrule

LSM & \textsc{KOind} & 49.83\enskip (0.015) & 53.79\enskip (0.012) & 55.15\enskip (0.007) \\ 
  LSM & \textsc{KOind, payoff} & 50.66\enskip (0.015) & 53.36\enskip (0.008) & 54.84\enskip (0.008) \\ 
  PO & \textsc{KOind} & 51.59\enskip (0.012) & 54.46\enskip (0.012) & 55.73\enskip (0.006) \\ 
  PO & \textsc{KOind, payoff} & 51.38\enskip (0.012) & 53.96\enskip (0.008) & 55.31\enskip (0.007) \\ 
  RPO & \textsc{KOind, payoff} & \bfseries 51.93\enskip (0.011) & \bfseries 54.58\enskip (0.013) & \bfseries 55.97\enskip (0.007) \\[0.5em] 
  PO-UB & \textsc{KOind} & 53.47\enskip (0.009) & 55.49\enskip (0.007) & 56.74\enskip (0.006) \\ 
  PO-UB & \textsc{KOind, payoff} & 52.50\enskip (0.021) & 55.06\enskip (0.014) & 56.40\enskip (0.015) \\ \midrule

  LSM & \textsc{pricesKO, KOind} & 50.38\enskip (0.011) & 53.70\enskip (0.010) & 54.99\enskip (0.009) \\ 
  LSM & \textsc{pricesKO, KOind, payoff} & 50.50\enskip (0.013) & 53.28\enskip (0.010) & 54.79\enskip (0.009) \\ 
  PO & \textsc{pricesKO, KOind} & \bfseries 51.60\enskip (0.011) & 54.34\enskip (0.010) & 55.55\enskip (0.005) \\ 
  PO & \textsc{pricesKO, KOind, payoff} & 51.27\enskip (0.011) & 53.91\enskip (0.008) & 55.29\enskip (0.008) \\ 
  RPO & \textsc{pricesKO, KOind, payoff} &  51.41\enskip (0.013) & \bfseries 54.38\enskip (0.014) & \bfseries 55.87\enskip (0.008) \\[0.5em] 
  PO-UB & \textsc{pricesKO, KOind} & 53.30\enskip (0.008) & 55.43\enskip (0.005) & 56.69\enskip (0.005) \\ 
  PO-UB & \textsc{pricesKO, KOind, payoff} & 52.48\enskip (0.021) & 55.04\enskip (0.014) & 56.38\enskip (0.015) \\ \midrule

  LSM & \textsc{pricesKO, KOind, maxpriceKO, max2priceKO} & 50.50\enskip (0.008) & 53.26\enskip (0.013) & 54.79\enskip (0.014) \\ 
  LSM & \textsc{pricesKO, KOind, maxpriceKO, max2priceKO, payoff} & 50.49\enskip (0.009) & 53.26\enskip (0.013) & 54.79\enskip (0.014) \\ 
  PO & \textsc{pricesKO, KOind, maxpriceKO, max2priceKO} & 51.23\enskip (0.010) & 53.89\enskip (0.012) & 55.28\enskip (0.011) \\ 
  PO & \textsc{pricesKO, KOind, maxpriceKO, max2priceKO, payoff} & 51.23\enskip (0.011) & 53.89\enskip (0.011) & 55.28\enskip (0.011) \\ 
  RPO & \textsc{pricesKO, KOind, maxpriceKO, max2priceKO, payoff} & \bfseries 51.39\enskip (0.017) & \bfseries 54.37\enskip (0.016) & \bfseries 55.84\enskip (0.012) \\[0.5em] 
  PO-UB & \textsc{pricesKO, KOind, maxpriceKO, max2priceKO} & 52.48\enskip (0.027) & 55.04\enskip (0.019) & 56.38\enskip (0.018) \\ 
  PO-UB & \textsc{pricesKO, KOind, maxpriceKO, max2priceKO, payoff} & 52.40\enskip (0.039) & 54.90\enskip (0.082) & 56.38\enskip (0.019) \\ \midrule

  LSM & \textsc{pricesKO, prices2KO, KOind} & 50.32\enskip (0.014) & 53.19\enskip (0.010) & 54.61\enskip (0.008) \\ 
  LSM & \textsc{pricesKO, prices2KO, KOind, payoff} & 50.25\enskip (0.016) & 53.05\enskip (0.010) & 54.60\enskip (0.008) \\ 
  RPO & \textsc{pricesKO, prices2KO, KOind, payoff} & \bfseries 50.94\enskip (0.021) & \bfseries 53.78\enskip (0.019) & \bfseries 55.24\enskip (0.033) \\ \bottomrule
\end{tabular}

\caption{Out-of-sample performance for different policies, for $n = 16$ assets. \label{table:DFM_performance_neq16}}
\end{table}

\end{document}